\documentclass[11pt]{amsart}
\usepackage[ left=3cm, right=3cm, top = 2.8cm, bottom = 2.8cm ]{geometry}

\usepackage[cmtip, all]{xy}
\usepackage{MnSymbol}
\usepackage{comment}

\usepackage{color}
\definecolor{winered}{rgb}{0.8,0,0}
\definecolor{awesome}{rgb}{1.0, 0.13, 0.32}

\usepackage[pdfauthor={}, %
				pdftitle={blochchow}%
			dvips,colorlinks=true]{hyperref}
\hypersetup{
	bookmarksnumbered=true,
	linkcolor=awesome,
	citecolor=blue,
}

\newtheorem{thm}{Theorem}[section]
\newtheorem{prop}[thm]{Proposition}
\newtheorem{lem}[thm]{Lemma}
\newtheorem{cor}[thm]{Corollary}

\theoremstyle{definition}

\newtheorem{defn}[thm]{Definition}
\theoremstyle{remark}
\newtheorem{remk}[thm]{Remark}
\newtheorem{remks}[thm]{Remarks}

\newtheorem{exm}[thm]{Example}
\newtheorem{exms}[thm]{Examples}
\newtheorem{notat}[thm]{Notation}
\numberwithin{equation}{section}

{\hfill$\square$\end{defn}}
\newenvironment{rem}{\begin{remk}}%
{\hfill$\square$\end{remk}}
{\hfill$\square$\end{remks}}
{\hfill$\square$\end{exm}}
{\hfill$\square$\end{exms}}
{\hfill$\square$\end{notat}}

\newcommand{\thmref}{Theorem~\ref}
\newcommand{\propref}{Proposition~\ref}
\newcommand{\corref}{Corollary~\ref}

\newcommand{\lemref}{Lemma~\ref}

\newcommand{\sE}{{\mathcal E}}
\newcommand{\sF}{{\mathcal F}}

\newcommand{\sI}{{\mathcal I}}

\newcommand{\sK}{{\mathcal K}}
\newcommand{\sL}{{\mathcal L}}

\newcommand{\sO}{{\mathcal O}}

\newcommand{\sQ}{{\mathcal Q}}
\newcommand{\sR}{{\mathcal R}}
\newcommand{\sS}{{\mathcal S}}

\newcommand{\sZ}{{\mathcal Z}}
\newcommand{\A}{{\mathbb A}}

\newcommand{\G}{{\mathbb G}}

\newcommand{\M}{{\mathbb M}}

\renewcommand{\P}{{\mathbb P}}
\newcommand{\Q}{{\mathbb Q}}

\newcommand{\Z}{{\mathbb Z}}

\newcommand{\fp}{{\mathfrak p}}

\newcommand{\CH}{{\rm CH}}

\newcommand{\surj}{\twoheadrightarrow}
\newcommand{\inj}{\hookrightarrow}
\newcommand{\red}{{\rm red}}
\newcommand{\codim}{{\rm codim}}

\newcommand{\Pic}{{\rm Pic}}

\newcommand{\Hom}{{\rm Hom}}

\newcommand{\Spec}{{\rm Spec \,}}
\newcommand{\sing}{{\rm sing}}

\newcommand{\divf}{{\rm div}}

\newcommand{\id}{{\operatorname{id}}}

\newcommand{\Sch}{{\operatorname{\mathbf{Sch}}}}

\newcommand{\Sm}{{\mathbf{Sm}}}

\newcommand{\GL}{{\operatorname{\rm GL}}}

\newcommand{\ds}{{/\kern-3pt/}}

\newcommand{\Gr}{{\text{\rm Gr}}}

\newcommand{\Supp}{{\operatorname{Supp}}}

\newcommand{\Tor}{{\operatorname{Tor}}}

\newcommand{\sm}{{\operatorname{sm}}}

\newcommand{\un}{\underline}
\newcommand{\ov}{\overline}

\renewcommand{\dim}{\text{\rm dim}}

\newcommand{\tuborg}{\left\{\begin{array}{ll}}
\newcommand{\sluttuborg}{\end{array}\right.}

\newcommand{\zar}{{\rm zar}}
\newcommand{\nis}{{\rm nis}}

\newcommand{\ab}{{\rm ab}}
\newcommand{\LW}{{\rm LW}}
\newcommand{\reg}{{\rm reg}}

\newcommand{\PGL}{{\rm PGL}}

\newcommand{\nsm}{{\rm nsm}}
\newcommand{\ncm}{{\rm nCM}}
\newcommand{\cm}{{\rm CM}}

\newcommand{\wt}{\widetilde}
\newcommand{\wh}{\widehat}

\def\cO{\mathcal{O}}

\def\ol#1{\overline{#1}}

\newcounter{elno}   

\newcounter{elno-prime}

\newcounter{elno-abc}   
\newenvironment{listabc}{
                         \begin{list}{\alph{elno-abc})
                                     }{\usecounter{elno-abc}}
                      }{
                         \end{list}}
\newcounter{elno-abc-prime}

\begin{document}
\title[Bloch's formula and higher dimensional CFT]{Bloch's formula 
for 0-cycles with modulus and higher dimensional class field theory}   
\author{Federico Binda, Amalendu Krishna and Shuji Saito}

\address{Dipartimento di Matematica ``Federigo Enriques'', Universit\`a degli Studi di Milano, Via C.~Saldini 50, 
20133, Milano, Italy}
\email{federico.binda@unimi.it}
\address{School of Mathematics, Tata Institute of Fundamental Research,  
1 Homi Bhabha Road, Colaba, Mumbai, India}
\email{amal@math.tifr.res.in}
\address{Graduate School of Mathematical Sciences, University of Tokyo,  3-8-1 Komaba, Tokyo, 153-8914, Japan}
\email{sshuji@msb.biglobe.ne.jp}

\thanks{F.B. is supported by the DFG SFB/CRC 1085 ``Higher Invariants''. S.S.~is supported by JSPS KAKENHI Grant (15H03606) and the DFG SFB/CRC 1085 ``Higher Invariants''}

%\dedicatory{}


\setcounter{tocdepth}{1}
\maketitle

\begin{quote}\emph{Abstract.}  
We prove Bloch's formula for the Chow group of 0-cycles with modulus
on a smooth quasi-projective 
surface over a field. We use this formula to give a simple
proof of the rank one case of a conjecture of Deligne and Drinfeld
on lisse $\ov{\Q}_{\ell}$-sheaves. This was 
originally solved by Kerz and Saito in characteristic $\neq 2$.
\end{quote}
%\end{abstract}
\setcounter{tocdepth}{1}
\maketitle
\tableofcontents  

%\maketitle
%\tableofcontent
\section{Introduction}\label{sec:Intro}
Let $X$ be a normal and complete variety over a finite field
$k$. Let $U \subset X$ be a quasi-projective
open subvariety which is smooth and 
whose complement is the support of an effective Cartier divisor on $X$.
Let $(V_Z)_Z$ be a family of semisimple lisse $\ov{\Q}_{\ell}$-sheaves on
the normalizations $Z^N$, where $Z$ runs through all closed integral curves
on $U$. The family $(V_Z)_Z$ is said to be a 2-skeleton sheaf if 
for two curves $Z_1 \neq Z_2$, the sheaves $V_{Z_1}$ and $V_{Z_2}$
become isomorphic up to semi-simplification after pull-back to
the support of $Z^N_1 \times_U Z^N_2$. 

Let $\psi_Z \colon \ov{Z}^N \to \ov{Z}$ be the normalization of the closure 
$\ov{Z}$ of an integral curve $Z \subset U$ in $X$ 
and let $Z_{\infty}
= \psi^{-1}_Z(\ov{Z} \setminus U)$. The family $(V_Z)_Z$ is said to have
bounded ramification if there exists an effective Cartier divisor
$D \subset X$ supported on $X \setminus U$ such that for 
all integral curves $Z \subset U$, one has the inequality of
Cartier divisors 
${\underset{y \in Z_\infty}\sum} 
{\rm ar}_y(V_Z)[y] \le \psi^*_Z(D)$ on $\ov{Z}^N$,
where ${\rm ar}_y(V_Z)$ is the local Artin conductor of $V_Z$ at $y$ 
(see \cite{Serre}). 

Motivated by the work of Drinfeld \cite{Drinfeld}, a conjecture of 
Deligne \cite{EK} states that given a 2-skeleton sheaf
$(V_Z)_Z$ with bounded ramification, there exists a lisse 
$\ov{\Q}_{\ell}$-sheaf $V$ on $U$ such that $V|_{Z^N} \cong V_Z$ after
semi-simplification for all integral curves $Z \subset U$.

Let $\CH_0(X|D)$ denote the Chow group of 0-cycles on $X$ with modulus
(see \cite{BS} or \cite{KeS}) for any effective Cartier divisor $D \subset X$.
Let $C(U) = {\underset{D \subset X \setminus U}\varprojlim} \CH_0(X|D)$, 
where $D$ runs through all effective Cartier divisors supported on
$X \setminus U$. Let $\pi^{\ab}_1(U)$ denote the abelianized {\'e}tale
fundamental group of $U$. 

Kerz and Saito \cite{KeS} showed that $C(U)$ is independent of the 
choice of the compactification $X$ and that there is a reciprocity map $\rho_U \colon
C(U) \to \pi^{\ab}_1(U)$. This induces a map $\rho^0_U \colon
C(U)^0 \to \pi^{\ab}_1(U)^0$ on the degree zero parts (see \ref{sec:KST}).
The rank one case of the above conjecture of Deligne
is then a direct consequence of another conjecture, namely, the map
$\rho^0_U$ is an isomorphism of topological pro-finite abelian groups.
Kerz and Saito proved this latter conjecture when ${\rm char}(k) \neq 2$
through a series of several highly non-trivial reductions.

In this paper, we give a simple and independent proof of 
the theorem of Kerz and Saito, including  the
missing characteristic $2$ case. As a byproduct, we obtain a simple proof of
the rank one case of Deligne's conjecture in all positive characteristics.
We deduce these results by proving 
Bloch's formula for $\CH_0(X|D)$ when $X$ is a smooth quasi-projective surface over an arbitrary base field $k$.
%over a field $k$ and $D$ is an effective Cartier divisor on $X$. When $D = \emptyset$, this reduces to the classical formula of Bloch, equating $\CH_0(X)$ with $H^2(X, \sK_{2,X})$.

\subsection{The main results}\label{sec:Results}
We now state our precise results. 
For a local ring $A$ and an ideal $I \subset A$, let $K^M_*(A,I)$
denote the relative Milnor $K$-theory (e.g., see \cite[\S~1.3]{Kato-Saito}). 
Let $A^h$ denote the Henselization of $A$ with respect to its maximal ideal.
For a closed immersion of Noetherian schemes $D \subset X$
defined by the sheaf of ideals $\sI_D$, 
we let $\sK^M_{n,X}$ denote the Zariski (resp. Nisnevich) sheaf on
$X$ whose stalk at a point $x \in X$ is the relative Milnor $K$-group
$K^M_n(\sO_{X,x}, \sI_{D,x})$ (resp. $K^M_n(\sO^h_{X,x}, \sI^h_{D,x})$)

\begin{thm}\label{thm:Main-1}
Let $X$ be a smooth quasi-projective surface over a field $k$ and let
$D \subset X$ be an effective Cartier divisor. Then there are canonical
isomorphisms
\begin{equation}\label{eqn:Main-1-0}
\rho_{X|D} \colon \CH_0(X|D) \xrightarrow{\cong} 
H^2_{\zar}(X, \sK^M_{2, (X,D)}) \xrightarrow{\cong} H^2_{\nis}(X, \sK^M_{2, (X,D)}).
\end{equation}
\end{thm}

\vskip .2cm

When $D = 0$, \thmref{thm:Main-1} is classical and was first 
proven by Bloch \cite{Bloch} for Quillen $K$-theory and by
Kato \cite{Kato} for Milnor $K$-theory. 
For general $D$, \thmref{thm:Main-1} was proven earlier
by Binda and Krishna \cite[Theorem~1.8]{BK} when $k$ is algebraically closed.
When $X$ is affine and $k$ is perfect,
this theorem was proven by Gupta and Krishna \cite[Theorem~1.4]{GK}.
We shall refer to this theorem as the Bloch-Kato formula.

\vskip .2cm

Combining \thmref{thm:Main-1} 
with an induction argument due to Kerz and Saito \cite{KeS} and
the Kato-Saito class field theory \cite{Kato-Saito}, we obtain the
following.

\begin{thm}\label{thm:Main-2}
Let $U$ be a smooth quasi-projective variety of dimension $d \ge 1$ over
a finite field $k$. Suppose there is an open immersion $U \subset X$
such that $X$ is normal and proper over $k$ and $(X \setminus U)_{\red}$ is
the support of an effective Cartier divisor on $X$. Then the  
reciprocity map
\[
\rho^0_U \colon C(U)^0 \to \pi^{\ab}_1(U)^0
\]
is an isomorphism of  pro-finite topological abelian groups.
\end{thm}

\vskip .3cm

%As explained in \cite{KeS},
Since the rank one lisse 
$\ov{\Q}_{\ell}$-sheaves on $U$ are same as the characters of
the pro-finite group $\pi^{\ab}_1(U)$, the following is 
an immediate consequence of \thmref{thm:Main-2} and the 
Pontryagin duality theorem for pro-finite groups.

\begin{cor}\label{cor:Main-3}
The rank one case of Deligne's conjecture is true for any smooth
quasi-projective variety $U$ over a finite field.
\end{cor}

\vskip .3cm

We list some more applications of \thmref{thm:Main-1}.
Apart from its application to higher dimensional class field theory
and solution of the rank one case of Deligne's conjecture,
\thmref{thm:Main-1} has further applications. It is used in
\cite{BK3} to prove a restriction isomorphism for the
relative 0-cycle group on a regular and flat projective scheme over
a henselian discrete valuation ring with arbitrary reduction.
\thmref{thm:Main-1} is used in \cite{BK2} to
give an explicit comparison between 
the motivic cohomology (in the sense of Suslin-Voevodsky)
and the Levine-Weibel Chow group of a normal crossing scheme.
Finally, it forms a crucial step in the recent proof of Bloch's
formula for the Chow group of 0-cycles with modulus in higher dimensions
in \cite{GK-20}.

\vskip .3cm

\subsection{Outline of proofs}\label{sec:Outline}
Let $U$ be as in \thmref{thm:Main-2}.
Kerz and Saito \cite{KeS} defined a reciprocity map
$\rho_U \colon C(U) \to \pi^{\ab}_1(U)$, and the goal is to show that this is an
isomorphism on the degree zero parts. As a first step, it is possible to use an alteration (in the sense of de Jong) to assume that the compactification $X$ of $U$ is smooth and projective, with complement $(X\setminus U)_{\red}$ the support of a strict normal crossing divisor. 
Next,  using a generalized version of
the Bertini theorem of Poonen \cite{Poonen} and the 
Lefschetz hyperplane theorem of Kerz and Saito \cite{KeS-1}, one
  reduces to the case when $d \le 2$.
  This is the main point, and it is the case for which we provide a new and
simple proof in this paper.

The $d = 1$ case is classical. 
Our strategy for proving \thmref{thm:Main-2} in $d = 2$ case
is the following. We define a cycle class map
$\rho_{X|D} \colon \CH_0(X|D) \to H^2_{\nis}(X, \sK^M_{2, (X,D)})$
for any effective Cartier divisor $D \subset X$ supported on $X \setminus U$.
\thmref{thm:Main-1} says that this map is an isomorphism.
Taking the limit over all such $D \subset X$, we get an
isomorphism
$\wt{\rho}_U \colon C(U) \to  {\underset{D} \varprojlim} \
H^2_{\nis}(X, \sK^M_{2, (X,D)})$. 
Let us denote the limit on the right side as $C^{\rm KS}(U)$.

Kato and Saito \cite{Kato-Saito} constructed a reciprocity isomorphism 
$\wh{\rho}^0 \colon C^{\rm KS}(U)^0 \xrightarrow{\cong} \pi^{\ab}_1(U)^0$.
We thus have maps
\[
C(U)^0 \xrightarrow{\wt{\rho}_U^0} C^{\rm KS}(U)^0 \xrightarrow{\wh{\rho}^0}
\pi^{\ab}_1(U)^0.
\]
It is easy to check that the composite map is the reciprocity map
$\rho^0_U$ of Kerz and Saito. The map
$\wh{\rho}^0$ was shown to be an isomorphism by Kato and Saito 
 \cite{Kato-Saito}. This provides a proof of \thmref{thm:Main-2}.

The heart of this paper is the proof of \thmref{thm:Main-1}.
This proof has two main ingredients: (1) the fundamental exact sequence
of \cite{BK} which relates the 0-cycles group with modulus on smooth 
varieties with a modified version of the Levine-Weibel 0-cycle group 
(referred to as the lci Chow group) on a 
singular variety, and (2) a proof of Bloch's formula for the modified version
of the Levine-Weibel 0-cycle group.
For (2), we use the existence of a good theory of pull-back and
push-forward for the modified Levine-Weibel Chow group and a 
pro-$\ell$ extension trick to reduce it
to the case when the ground field is infinite.

When the ground field is algebraically closed, 
Levine \cite{Levine-2} constructed 
a geometric theory of Chow ring on singular varieties.
As a consequence, he showed that for a reduced quasi-projective
surface $X$ over an algebraically closed field, 
the cycle class map from the Levine-Weibel Chow group to $K_0(X)$
is injective. Even as Levine's complete proof is yet unpublished, 
a published account of the dimension two case of his proof
is available in \cite{BSri}.
A major part of this paper is devoted to showing that
Levine's program can be carried out for surfaces over an arbitrary infinite field.
In particular, the injectivity of the cycle class map (for surfaces) holds over
any infinite field. A known relation between the Levine-Weibel
and the lci Chow groups (shown in \cite{BK}) then shows that this holds
for the latter group too. The injectivity of the cycle class map
implies Bloch's formula.

In \S~\ref{sec:Cycle}, we recall definitions of all cycle groups
and recollect some results from \cite{BK} that play key roles
in our proofs.
In \S~\ref{sec:CCM*}, we construct the cycle class map and the 
Bloch-Kato map which are used in the proof of Bloch's formula.
We prove some results regarding these maps which are later used in the
main proofs. The key results in this section are
Theorems~\ref{thm:CCM-main} and ~\ref{thm:CCM-main-mod}.
The next four sections form the heart of the proofs, where
the goal is to carry out Levine's strategy over an arbitrary infinite fields.
Here, the main idea is solely due to Levine. However, due to the
complications arising from the arbitrariness
of the ground field, we need to include new arguments at several places
to ensure that Levine's results remain valid over such fields.
In \S~\ref{sec:BFormula}, we finish the proofs of our main results.

\section{The Chow group of 0-cycles}\label{sec:Cycle}
In this section,  we fix our notations and recall the definitions of various 
0-cycle groups that we use in our proofs. We also prove some
results relating these groups.

\subsection{Notations}\label{sec:Notn}
For a field $k$, we shall let $\Sch_k$ denote the category of
quasi-projective schemes over $k$. We let $\Sm_k$ denote the
subcategory of $\Sch_k$ consisting of smooth schemes over $k$.
For $X \in \Sch_k$, we let $X_{\rm reg}$ denote the largest open
subscheme of $X$ which is regular. We let $X_{\rm sing}$ denote 
the complement of $X_{\rm reg}$ in $X$ with its reduced induced closed
subscheme structure. Recall here the convention that if $\dim(X) = d$, then
$X_{\rm sing}$ contains every point of $X$ which lies on an
irreducible component of $X$ of dimension less than $d$.
We shall let $X_{\sm}$ denote the set of points in $X$ where
the map $X \to \Spec(k)$ is smooth. It is clear by the definition 
of smooth morphisms that $X_{\sm} \subset X$ is open.
We let $X_{\nsm} = X \setminus X_\sm$ with the reduced induced
closed subscheme structure. Note that
$X_\sm \subset X_\reg$ and hence $X_\sing \subset X_\nsm$.

For $X, Y \in \Sch_k$, we shall write the
product $X \times_k Y$ as $X \times Y$.
If $X_1 \times \cdots \times X_n$ is a product of schemes, then
$p_i$ will denote the projection from this product to its $i$-th
factor (unless we use a specific notation in some
context).
We shall often write $X \times \P^n_k$ as $\P^n_X$.
If $z \in Z := X \times Y$ is a closed point, then it is not true in general 
that $z$ is uniquely determined by its projections
$p_1(z) \in X$ and $p_2(z) \in Y$. However, this is indeed true if
$z$ happens to be a $k$-rational point of $X \times Y$.
In this case, we can uniquely write $z = (p_1(z), p_2(z)) \in X \times Y$.
More generally, if $S$ is a scheme and $X, Y \in \Sch_S$,
then the canonical map $(X \times_S Y)(S)
\to X(S) \times Y(S)$ of sets is bijective.  
This fact will be used frequently in this paper.
If $f \colon X' \to X$ is a morphism in $\Sch_k$, then a rational fiber
of $f$ will mean the scheme-theoretic fiber over a $k$-rational point of $X$.

We shall let $\sZ(X)$ denote the free abelian group
generated by integral closed subschemes of $X$.
If $f \colon X' \to X$
is a proper map and $V \subset X'$ is an integral closed subscheme,
then we shall let $f_*(V) \in \sZ(X)$ 
denote the push-forward on $X$ of the cycle
$[V]$ in the sense of \cite{Fulton1}.
 
If $X \in \Sch_k$ and $V_1, V_2 \subset X$ are two irreducible closed
subschemes of $X$, then we 
shall say that they intersect properly if
$\codim_X(V_1 \cap V_2) \ge \codim_X(V_1) + \codim_X(V_2)$.
If $V_1$ or $V_2$ is not necessarily irreducible, then
we shall say that they intersect properly if every 
irreducible component of $V_1$ intersects every irreducible component of
$V_2$ properly. If $V_1, V_2$ are two closed subschemes intersecting
properly, we shall let $V_1 \cdot V_2$ be their intersection product
as a cycle on $X$ in the sense of \cite[Chapter~6]{Fulton1}.
We shall use this intersection product only in the case
when $V_1 \cap V_2 \subset X_\reg$. 
%We shall use this intersection product when $X$ is either a homogeneous
%space or the regular locus of a singular surface. 
We shall let
$\codim_X(\emptyset) = \infty$. If $X$ is integral and $f \in k(X)^{\times}$
is a rational function, we shall let $\divf(f)$ denote the cycle
associated to $f$ in $\sZ(X)$ as in \cite{Fulton1}.

\subsection{The Levine-Weibel Chow group of a surface}
\label{sec:LWCG}
Since we use the Levine-Weibel Chow group \cite{LW} only for surfaces, we
do not recall its definition in full generality. In the case of
surfaces, we shall use the following variant of the Levine-Weibel Chow group
which was defined by Levine \cite{Levine-2}. It was shown
in {\sl ibid.} that this variant is actually isomorphic to the 
Levine-Weibel Chow group, but we shall have no occasion to use this
comparison. We fix an arbitrary field $k$.

\begin{defn}\label{defn:LW-defn}
Let $X$ be a reduced quasi-projective surface over $k$ and let
$Y \subset X$ be a nowhere dense closed subscheme containing $X_{\sing}$.
Let $\sZ_0(X,Y)$ be the free abelian group on the set of closed points
lying in $X \setminus Y$. Let $\sR^L_0(X,Y)$ denote the subgroup of
$\sZ_0(X,Y)$ generated by cycles $C_0 - C_{\infty} \in \sZ_0(X,Y)$
such that the following hold:

There exists a closed subscheme $C \subset X \times \P^1_k$ 
of pure dimension one such that
\begin{enumerate}
\item
$C \cap \P^1_{Y}$ is finite and $C \cap (Y \times \{0, \infty\}) =
\emptyset$.
\item
$C \cap (X \times \{0, \infty\})$ is finite.
\item
The projection map $C \to \P^1_k$ is flat over a neighborhood of 
$\{0, \infty\}$.
\item
If $p_1 \colon \P^1_X \to X$ is the projection, then
$(p_1)_*(C \cdot (X \times \{0\})) = C_0$ and
$(p_1)_*(C \cdot (X \times \{\infty\})) = C_\infty$.
\item
The ideal of $C$ in $\P^1_X$ is a complete intersection
in the local ring of every point $x \in C \cap \P^1_{Y}$.  
\end{enumerate}

We define $\CH^L_0(X,Y)$ to be the quotient 
$\frac{\sZ_0(X,Y)}{\sR^L_0(X,Y)}$. When $Y = X_\sing$, we write
$\CH^L_0(X,Y)$ simply as $\CH^L_0(X)$. 
In this paper, we shall refer to $\CH^L_0(X)$ as the
`Levine-Weibel Chow group' of $X$.
We shall call a subscheme $C$ as above a {\sl Cartier curve}
relative to $Y$. A Cartier curve in general will mean the one
relative to $X_\sing$.
\end{defn}

In the sequel, we shall use the notation $\CH^{\LW}_0(X,Y)$ for the original
definition of the Chow group of 0-cycles given by Levine and Weibel \cite{LW}.  
The notion of ``Cartier curves'' was introduced in \cite[Definition 1.2]{LW}:
we shall reserve the term `Levine-Weibel Cartier curves' for those
defined in \cite{LW}.

The following lemma allows us to assume $C$ to be reduced in the
definition of the Levine-Weibel Chow group if $k$ is infinite.

\begin{lem}\label{lem:Reduced}
Assume that $k$ is infinite.
Let the pair $(X,Y)$ be as in Definition~\ref{defn:LW-defn}.
Then $\sR^L_0(X,Y)$ is the subgroup of $\sZ_0(X,Y)$ generated by cycles 
$C_0 - C_{\infty} \in \sZ_0(X,Y)$, where $C$ is a reduced Cartier curve
in $\P^1_X$. If $X$ is integral, we can assume $C$ to be integral.
\end{lem}
\begin{proof}
Let $C \subset \P^1_X$ be a Cartier curve. Then it is also a Levine-Weibel
Cartier curve on $\P^1_X$. In the latter case, it
was shown in \cite[Lemma~1.4]{Levine-3} that we can find
reduced Levine-Weibel Cartier curves $C_1, \ldots , C_r$ on $\P^1_X$ 
which miss $Y \times \{0, \infty\}$ (since its codimension
in $\P^1_X$ is at least two, see \cite[Lemma~1.3]{ESV})
such that we have
\[
C \cdot (X \times \{0\})-
C \cdot (X \times \{\infty\}) = 
\stackrel{r}{\underset{i =1} \sum}
\left[C_i \cdot (X \times \{0\})-
C_i \cdot (X \times \{\infty\})\right]
\]
in $\sZ_0(\P^1_X, \P^1_Y)$. 
In particular, we get
\[
C_0 - C_\infty = \stackrel{r}{\underset{i =1} \sum}
\left[(p_1)_*(C_i \cdot (X \times \{0\}))-
(p_1)_*(C_i \cdot (X \times \{\infty\}))\right]
= \stackrel{r}{\underset{i =1} \sum}
\left[(C_i)_0 - (C_i)_\infty\right].
\]

Since each $C_i$ is a Cartier curve in our sense and is reduced,
we are done. If $X$ was integral, it was shown in 
\cite[Lemma~1.4]{Levine-3} that we could have chosen each
$C_i$ to be integral. This finishes the proof.
\end{proof}

\begin{lem}\label{lem:Reduced-finite}
Assume that $C \subset \P^1_X$ is a reduced Cartier curve.
We can then find another reduced Cartier curve $C' \subset \P^1_X$
such that the projection $C' \to X$ is finite and
$C_0 - C_\infty = C'_0 - C'_\infty$.
\end{lem}
\begin{proof}
Write $C = C_{1}\cup\ldots\cup C_{r}$ as union of irreducible components. 
Suppose that all the components $C_{i}$ for $1\leq i\leq s$ are such that 
the composition $C_{i}\to X$ is finite and $p_1(C_{i} )=x_i$ is a closed 
point of $X$ for $s<i\leq r$. Note that $p_1$ is projective.
Let $C_{i}$ be one of these non-finite components.
Then condition (1) in Definition~\ref{defn:LW-defn} implies that
$x_i \in X_\reg$.
Moreover, we have
\[
(p_{1})_*(C_{i} \cdot (X\times  \{ 0\}))  - (p_1)_{*}(C_{i} \cdot (X\times  
\{ \infty \}) ) = (p_1)_{*} (C_{i} \times_{\P^1} \{0\}  -  C_{i} 
\times_{\P^1} \{\infty \} )=0 
\]
as $0$-cycles on $X_{\rm reg}$. Note that the above intersection makes sense, 
thanks to conditions (1) and (2) in Definition~\ref{defn:LW-defn}.

We let $C' := \ol{C \setminus (C_{s+1} \cup \cdots \cup C{r})}$
(where the closure is taken in $\P^1_X$)
with the reduced induced closed subscheme structure.
Since $x_i$ is a closed point of $X_\reg$ for every $s < i \le r$
and $C$ is reduced, it follows that the schemes $C$ and $C'$ 
agree in an open neighborhood of $\P^1_{X_\sing}$.  
We conclude that $C'$ is a reduced Cartier curve on $\P^1_X$,
the projection $p_1 \colon C' \to X$ is finite and
$C_0 - C_\infty = C'_0 - C'_\infty$.
\end{proof}

Note that the above proof shows that
an integral Cartier curve $C$ is either finite over $X$
or $C_0 - C_\infty = 0$ in $\sZ_0(X,Y)$.
Combining Lemmas~\ref{lem:Reduced} and ~\ref{lem:Reduced-finite}, we
therefore get the following.

\begin{cor}\label{cor:Reduced-finite*}
Assume that $k$ is infinite.
Let the pair $(X,Y)$ be as in Definition~\ref{defn:LW-defn}.
Then $\sR^L_0(X,Y)$ is the subgroup of $\sZ_0(X,Y)$ generated by cycles 
$C_0 - C_{\infty} \in \sZ_0(X,Y)$, where $C$ is a reduced Cartier curve
in $\P^1_X$ which is finite over $X$.
If $X$ is integral, we can assume $C$ to be integral.
\end{cor}

\subsection{The lci Chow group of a singular scheme}\label{sec:lci}
Let $k$ be any field. The lci Chow group of a singular scheme was
introduced in \cite{BK}. We recall it here.
Let $C$ be a reduced equidimensional curve over $k$ and let
$k(C)$ denote its ring of total quotients. Let
$C_1, \ldots , C_r$ be the irreducible components of $C$.
Then $k(C)$ is the product of the quotient fields of all 
$C_i$'s. For $f \in k(C)^{\times}$, we can therefore write
$f = (f_i) \in \stackrel{r}{\underset{i =1}\prod} k(C_i)^{\times}$.
We let $\divf(f)$ be the sum  $\stackrel{r}{\underset{i =1}\sum} \divf(f_i)
\in \sZ(C)$. If $Z \subset C$ is a finite closed subset, we 
let $\sO_{C,Z}$ denote the semi-local ring of $C$ at $Z$.
Since $\sO^{\times}_{C, Z} \subset k(C)^{\times}$, the
cycle $\divf(f) \in \sZ(C)$ makes sense for any
$f \in \sO^{\times}_{C, Z}$. For such finite set $Z$, we let $\sZ_0(C,Z)$
denote the subgroup of $\sZ(C)$ generated by closed points away from
$Z$.

Let $X$ be a reduced quasi-projective scheme
over $k$ and let $Y \subset X$ be a nowhere dense closed subscheme 
containing $X_{\sing}$. 

\begin{defn}\label{defn:0-cycle-S-1}
Let $C$ be a reduced curve in $\Sch_k$ and let $\nu\colon C \to X$ be a 
finite morphism. We  say that $\nu\colon (C, Z)\to (X,Y)$ is 
\emph{a good curve relative to $Y$} if $Z$ is  a closed 
subscheme of $C$  such that the following hold.
\begin{enumerate}
\item
No component of $C$ is contained in $Z$.
\item
$\nu^{-1}(Y) \subseteq Z$.
\item
$\nu$ is a local complete intersection morphism  at every 
point $x \in C$ such that $\nu(x) \in Y$.
\end{enumerate}
\end{defn}

Given any good curve $(C,Z)$ relative to $Y$, we can consider the 
push-forward  (see \S~\ref{sec:Notn})
$\sZ_0(C,Z)\xrightarrow{\nu_{*}} \sZ_0(X,Y)$.
We let $\sR_0(C, Z, X)$ be the subgroup
of $\sZ_0(X,Y)$ generated by the set 
$\{\nu_*({\divf}(f))| f \in \cO^{\times}_{C, Z}\}$.
We write $\sR_0(X,Y)$ for the subgroup of $\sZ_0(X,Y)$ defined as  the 
image of the map
\begin{equation}\label{eqn:0-cycle-S-2}
{\underset{\nu\colon (C,Z)\to (X,Y) \ {\rm good}}\bigoplus} \sR_0(C, Z, X) \to 
\sZ_0(X, Y),
\end{equation}
where the index set runs over the set of good curves relative to $(X,Y)$.
We define the \emph{Chow group of 0-cycles on $X$} (relative to $Y$) to be the 
quotient
\begin{equation}\label{eqn:0-cycle-S-3}
\CH_0(X,Y) = \frac{\sZ_0(X, Y)}{\sR_0(X,Y)}.
\end{equation}

If $Y = X_\sing$, we shall write $\CH_0(X,Y)$ simply as $\CH_0(X)$.
To avoid any conflict of notations, we shall denote the Chow group
of $n$-dimensional cycles on $X$ modulo rational equivalence
in the sense of \cite[Chapter~1]{Fulton1} by $\CH^F_n(X)$.
To distinguish $\CH_0(X,Y)$ from $\CH^L_0(X,Y)$, we shall often
refer to it as the `lci Chow group' of $(X,Y)$.
The relation between $\CH^L_0(X,Y)$ and $\CH_0(X,Y)$ is given by the
following lemma.

\begin{lem}\label{lem:Levine-lci}
Let $k$ be an infinite field.  Let $X$ be a reduced quasi-projective surface
over $k$ and let $Y \subset X$ be a nowhere dense closed subscheme 
containing $X_{\sing}$. Then the identity map of $\sZ_0(X,Y)$ induces
a canonical surjection
\[
\phi_{(X,Y)} \colon \CH^L_0(X,Y) \surj \CH_0(X,Y).
\]
\end{lem}
\begin{proof}
Let $C \subset \P^1_X$ be a Cartier curve relative to $Y$
and let $Z = \P^1_Y \cap C$. We can assume
that $C$ is reduced and finite over $X$ 
by \corref{cor:Reduced-finite*}. It is then easy to see that
$(C,Z)$ is a good curve relative to $Y$ via the composite map
$\nu \colon C \inj \P^1_X \xrightarrow{p_X} X$.
Furthermore, if we let $f \colon C \to \P^1_k$
be the projection map, then it defines an element of $k(C)$.
The conditions (1) and (2) of Definition~\ref{defn:LW-defn} imply in fact that
$f \in \sO^{\times}_{C,Z}$.
It is clear that $\nu_*(\divf(f)) = C_0 - C_{\infty}$.
Hence, $C_0 - C_{\infty} \in \sR_0(X,Y)$. This finishes the proof.
\end{proof}

The lci Chow group $\CH_0(X)$ possesses a reasonable theory of
pull-back and push-forward maps. 
This property will play a very important role in the proofs of our
main theorems. In contrast, the Levine-Weibel Chow
group $\CH^L_0(X)$ is not known to have any good theory of push-forward maps
induced by finite maps between two varieties. 
A particular result that we shall use in the main proofs is the
following. 
We let $(X,Y)$ be as in \S~\ref{sec:lci}.
For any field extension $k \inj k'$, we let
$X' = X_{k'}:= X \otimes_k k'$ and $Y' = Y_{k'}:= Y \otimes_k k'$.
We let ${\rm pr}_{{k'}/{k}}: X' \to X$ denote the
projection map.
 
\begin{prop}\label{prop:PF-PB}
Let $k \inj k'$ be a separable algebraic (possibly infinite)
extension of fields. Then the following hold.
\begin{enumerate}
\item
There exists a pull-back map
%${\rm pr}^*_{{k'}/{k}}: \CH^{LW}_0(X) \to \CH^{LW}_0(X')$ and
${\rm pr}^*_{{k'}/{k}} \colon \CH_0(X,Y) \to \CH_0(X',Y')$.
%which commute with the canonical map $\CH^{LW}_0(X) \to \CH_0(X)$.
\item
If there exists a sequence of separable  
algebraic field extensions $k = k_0 \subset k_1 \subset
\cdots \subset k'$ with $k' = \cup_i k_i$ such that $X_i := X_{k_i}$
and $Y_i := Y_{k_i}$ for each $i \ge 1$, then
we have ${\underset{i}\varinjlim} \ \CH_0(X_i,Y_i) \xrightarrow{\simeq} 
\CH_0(X',Y')$. 
%The same holds for $\CH^{LW}_0(-)$ as well.
\item
If $k \inj k'$ is finite, then there exists a push-forward
$({\rm pr}_{{k'}/{k}})_* \colon \CH_0(X',Y') \to \CH_0(X,Y)$ such that
$({\rm pr}_{{k'}/{k}})_* \circ {\rm pr}^*_{{k'}/{k}}$ is multiplication by
$[k': k]$.
\end{enumerate}
\end{prop}
\begin{proof}
See \cite[Proposition~6.1]{BK}.
\end{proof}

\subsection{Zero cycles with modulus}\label{sec:Modulus}
Let $k$ be any field.
Given an integral normal curve ${C}$ over $k$ and an effective divisor
$E \subset {C}$, we say that a rational function $f$ on ${C}$ has modulus
$E$ if $f \in {\rm Ker}(\sO^{\times}_{C, E} \to \sO^{\times}_E)$.
Here, $\sO_{C, E}$ is the semi-local ring of $C$ at the union of
$E$ and the generic point of $C$.
In particular, ${\rm Ker}(\sO^{\times}_{C, E} \to \sO^{\times}_E)$ is just
$k(C)^{\times}$ if $|E| = \emptyset$.
Let $G({C}, E)$ denote the group of such rational functions.

Let ${X}$ be a reduced quasi-projective scheme over $k$ and let $D$ be an 
effective Cartier divisor on $X$ (we allow $D$ to be empty).
Let $Z_0(X,D)$ be the free abelian group on 
the set of closed points of $X\setminus D$. Let ${C}$ be an integral normal 
curve over $k$ and
let $\varphi_{{C}}\colon{C}\to {X}$ be a finite morphism such that 
$\varphi_{{C}}({C})\not \subset D$.  
The push forward of cycles along $\varphi_{{C}}$  
gives a well defined group homomorphism
\[
(\varphi_C)_* \circ \divf \colon G({C},\varphi_{{C}}^*(D)) \to Z_0(X,D).
\]

\begin{defn}[Kerz-Saito]\label{def:DefChowMod-Definition}
We define the Chow group $\CH_0({X}|D)$ of 0-cycles of ${X}$ with 
modulus $D$ as the cokernel of the homomorphism 
\begin{equation}\label{eqn:DefChowMod-0}
\bigoplus_{\varphi_{{C}}\colon {C}\to {X}}G({C},
\varphi_{{C}}^*(D)) \to Z_0(X,D),
\end{equation}
where the sum is taken over the set of finite morphisms 
$\varphi_{{C}}\colon {C} \to {X}$ from an integral normal curve such that
$\varphi_{{C}}({C}) \not\subset D$.
\end{defn}

A general theory of higher Chow groups with modulus $\CH_p(X|D,q)$ 
was introduced in \cite{BS}, where it was shown that $\CH_0(X|D,0)
\cong \CH_0({X}|D)$. It is clear from the definition that
the inclusion $ Z_0(X|D) \subset  Z_0(X)$ defines a canonical 
`forget modulus' map $\CH_0({X}|D) \to \CH_0^{F}(X)$ (the latter being Fulton's Chow group of zero cycles).
It is known that the Chow group of 0-cycles with modulus is covariantly
functorial for the proper maps: if $f\colon {X'} \to {X}$ is a proper map,
$D$ and $D'$ are effective Cartier divisors on ${X}$ and ${X'}$
such that $f^*(D) = D'$, then there is a push-forward
map $f_*\colon \CH_0({X'}|D') \to \CH_0({X}|D)$.
If $f$ is flat (but not necessarily proper)  of relative dimension zero, 
then there is a
pull-back map $f^*\colon \CH_0({X}|D) \to \CH_0({X'}|D')$
(see \cite[Lemma~2.7]{BS} or \cite[Propositions~2.10, 2.12]{KP}).

Analogous to the case of the lci Chow group, the
Chow group of 0-cycles with modulus has the following nice 
behavior with respect to the change of the base field.

\begin{prop}\label{prop:PF-fields-mod} 
Let $k \inj k'$ be a separable algebraic (possibly infinite)
extension of fields. Let
$X$ be a non-singular quasi-projective scheme 
over $k$ with an effective Cartier divisor $D$.
Let $X' = X_{k'}$ and $D' = D_{k'}$ denote the base change of $X$ and
$D$, respectively. Let ${\rm pr}_{{k'}/{k}}: X' \to X$ be the
projection map. Then the following hold.
\begin{enumerate}
\item
There exists a pull-back 
${\rm pr}^*_{{k'}/{k}}: \CH_0(X|D) \to \CH_0(X'|D')$.
\item
If there exists a sequence of separable algebraic
field extensions $k = k_0 \subset k_1 \subset
\cdots \subset k'$ with $k' = \cup_i k_i$, then
we have ${\underset{i}\varinjlim} \ \CH_0(X_{k_i}|D_{k_i}) \xrightarrow{\simeq} 
\CH_0(X'|D')$. 
\item
If $k \inj k'$ is finite, then there exists a push-forward
${\rm pr}_{{k'}/{k} \ *}: \CH_0(X'|D') \to \CH_0(X|D)$ such that
$({\rm pr}_{{k'}/{k}})_* \circ {\rm pr}^*_{{k'}/{k}}$ is multiplication by
$[k': k]$.
\end{enumerate}
\end{prop}
\begin{proof}
See \cite[Proposition~6.2]{BK}.
\end{proof}

\begin{rem}The reader should be aware of the fact that in the literature there exists another notion of Chow group of zero cycles with modulus, introduced by H.\ Russell \cite{RusANT} in the study of a higher dimensional analogue of the generalized Jacobian of Rosenlicht-Serre. Russell used his generalized Albanese with modulus to rephrase Lang's class field theory of function fields of varieties over finite fields in explicit terms. Note however that the modulus condition in \emph{loc.cit.}\ is different from the one used here (and in \cite{KeS}), and our work is not directly related to his. 
\end{rem}

\subsection{The double construction}\label{sec:Double}
Let $X$ be a non-singular quasi-projective scheme of dimension $d$ over $k$
and let $D \subset X$ be an effective Cartier divisor. Recall from 
\cite[\S~2.1]{BK} that the double of $X$ along $D$ is a quasi-projective
scheme $S(X,D) = X \amalg_D X$ so that
\begin{equation}\label{eqn:rel-et-2}
\begin{array}{c}
\xymatrix@R=1pc{
D \ar[r]^-{\iota} \ar[d]_{\iota} & X \ar[d]^{\iota_+} \\
X \ar[r]_-{\iota_-} & S(X,D)}
\end{array}
\end{equation}
is a co-Cartesian square in $\Sch_k$.
In particular, the identity map of $X$ induces a finite map
$\nabla:S(X,D) \to X$ such that $\nabla \circ \iota_\pm = {\rm id}_X$
and $\pi = \iota_+ \amalg \iota_-: X \amalg X \to S(X,D)$ 
is the normalization map.
We let $X_\pm = \iota_\pm(X) \subset S(X,D)$ denote the two irreducible
components of $S(X,D)$.  We use $S_X$ as a shorthand for $S(X, D)$
when the divisor $D$ is understood. $S_X$ is a reduced quasi-projective
scheme whose singular locus is $D_{\rm red} \subset S_X$. 
It is projective (resp.  affine) whenever
$X$ is so. It follows from \cite[Lemma~2.2]{Krishna-2} that  
~\eqref{eqn:rel-et-2}
is also a Cartesian square.

It is clear that the map
$\sZ_0(S_X,D)\xrightarrow{(\iota^*_{+}, \iota^*_{-})} \sZ_0(X_+,D) \oplus 
\sZ_0(X_-,D)$ is an isomorphism, and there are push-forward inclusion maps
${p_{\pm}}_* \colon \sZ_0(X,D) \to \sZ_0(S_X,D)$ such that 
$\iota^*_{+} \circ {p_{+}}_* = \id$
and $\iota^*_{+} \circ {p_{-}}_* = 0$.

The fundamental result that connects the 0-cycles with modulus on $X$ and
0-cycles on $S_X$ is the following.

\begin{thm}\label{thm:BS-main}
Let $k$ be any field.
Let $X$ be a smooth quasi-projective surface over $k$
and let $D \subset X$ be an effective Cartier divisor. Then there
is a split short exact sequence
\[
0 \to \CH_0(X|D) \xrightarrow{ {p_{+}}_*} \CH_0(S_X) \xrightarrow{\iota^*_-}
\CH_0(X) \to 0.
\]
\end{thm}
\begin{proof}
This is the restatement of the dimension two case of
\cite[Theorem~1.9]{BK} if we assume $k$ is perfect.
However, the only place in the proof of this theorem where the perfectness
is used is the construction of the map
$\tau^*_X \colon \CH^{\rm LW}_0(S_X) \to \CH_0(X|D)$ in
\cite[\S~5]{BK} when $k$ is infinite. The only reason for this
assumption was to be able to apply Bertini theorems for
reduced schemes. However, the relevant Bertini theorem 
was known only for schemes which are geometrically reduced
(the two notions coincide over a perfect field) when
\cite{BK} was written.
But we can now remove this extra assumption in view of the new
Bertini theorem \cite[Corollary~3.10]{GhK} for reduced schemes over any
infinite field.
\end{proof}

\section{The cycle class and Bloch-Kato maps}\label{sec:CCM*}
In this section, we shall recall the connection between 
various Chow groups defined in \S~\ref{sec:Cycle} and the
algebraic $K$-groups. We shall construct the Bloch-Kato map
which will be used in the proof of the Bloch-Kato formula.
We fix a field $k$.
For a $k$-scheme $X$, we let $K_0(X)$ denote $\pi_0(K(X))$, where
$K(X)$ is the Bass-Thomason-Trobaugh $K$-theory spectrum of $X$.
If $X$ is quasi-projective over $k$, then $K_0(X)$ coincides with 
the Grothendieck group of locally free sheaves on $X$.
For a closed immersion of $k$-schemes $D \inj X$, 
let $K(X,D)$ denote the homotopy fiber of the restriction map
of Bass-Thomason-Trobaugh $K$-theory spectra $K(X) \to K(D)$.
We let $K_i(X,D)$ denote the $i$-th stable homotopy group of $K(X,D)$.

\subsection{The cycle class maps}\label{sec:CCMaps}
Let us assume that $X$ is a reduced quasi-projective scheme
over $k$ and let $Y \subset X$ be any closed
subscheme containing $X_\sing$ (but not necessarily nowhere dense).
If $x \in X \setminus Y$ is a closed point, then it lies in the regular 
locus of $X$. In particular, the inclusion map
$\Spec(k(x)) \inj X$ has finite tor-dimension. This yields the maps of 
spectra $K(\Spec(k(x))) \xrightarrow{(\iota_x)_*} K(X) 
\xrightarrow{\iota^*_Y} K(Y)$.
Since the composite map is canonically null-homotopic, $(\iota_x)_*$
has a canonical factorization $(\iota_x)_* \colon K(\Spec(k(x))) \to K(X,D)$.
In particular, we have a map $(\iota_x)_* \colon \Z 
= K_0(\Spec(k(x))) \to K_0(X,D)$. Letting $cyc_{(X,Y)}([x]) =
(\iota_x)_*(1)$, and extending it linearly on $\sZ_0(X,D)$, we
get a cycle class map 
\begin{equation}\label{eqn:Cycle-class-rel}
cyc_{(X,Y)} \colon \sZ_0(X,Y) \to K_0(X,Y).
\end{equation}

Composing $cyc_{(X,Y)}$ with the canonical map
$K_0(X,Y) \to K_0(X)$, we obtain the cycle class map
$cyc_{(X,Y)} \colon \sZ_0(X,Y) \to K_0(X)$.
It is shown in \cite[Lemma~3.13]{BK} that if $Y$ is nowhere dense in $X$, 
this map kills $\sR_0(X,Y)$ and hence defines a cycle class map
\begin{equation}\label{eqn:CCM-lci}
cyc_{(X,Y)} \colon \CH_0(X,Y) \to K_0(X).
\end{equation}

We shall write $cyc_{(X,X_\sing)}$ in short as $cyc_X$.
We shall denote the image of $cyc_X$ by $F^2K_0(X)$.
\lemref{lem:Levine-lci} implies the following.

\begin{cor}\label{cor:CCM-Levine-lci}
Assume that $k$ is infinite, $\dim(X) = 2$ and $Y$ is
nowhere dense in $X$. Then the cycle class map $\sZ_0(X,Y) \to 
K_0(X)$ induces the commutative diagram
\begin{equation}\label{eqn:CCM-Levine-lci-0}
\xymatrix@C1pc{
\CH^L_0(X,Y) \ar@{->>}[r]^-{\phi_{(X,Y)}} 
\ar@/_1pc/[rr]_-{cyc^L_{(X,Y)}} & \CH_0(X,Y) \ar[r]^-{cyc_{(X,Y)}} &
K_0(X),}
\end{equation}
where the map $cyc^L_{(X,Y)}$ is the cycle class map for the Levine Chow group defined in \cite{LW} (see \cite[Section 3]{BK}).
\end{cor}

We next recall the cycle class map for the 0-cycle group with modulus.
Let $X$ be a non-singular quasi-projective surface over $k$
and let $D \subset X$ be an effective Cartier divisor.
By ~\eqref{eqn:Cycle-class-rel}, there is a cycle class map
$cyc_{(X,D)} \colon \sZ_0(X,D) \to K_0(X,D)$.
It follows from \cite[Theorem~1.2]{Krishna-1} that this map factors through
$\CH_0(X|D)$ so that there is a cycle class map
\begin{equation}\label{eqn:CCM-modulus}
cyc_{X|D} \colon \CH_0(X|D) \to  K_0(X,D).
\end{equation}
When $X$ is non-singular, it is clear from the above construction that the
composition
$\CH_0(X|D)$ \
$\xrightarrow{cyc_{X|D}} K_0(X,D) \to K_0(X)$
coincides with the composition $\CH_0(X|D) \to \CH^F_0(X)
\xrightarrow{cyc_X} K_0(X)$,
where $cyc_X$ is as in ~\eqref{eqn:CCM-lci}.

\subsection{The Cohomology of $K$-theory sheaves}\label{sec:B-map*}
We shall now establish a key step in the proof of Bloch-Kato formula
for the Chow group of 0-cycles. 
We first recall the definition of relative Milnor $K$-theory.

Let $A$ be a Noetherian local ring and let $I \subset A$ be an
ideal. We let $K^M_n(A)$ denote the Milnor
$K$-group of $A$ as defined by Kato-Saito \cite[\S~1.3]{Kato-Saito}.
We let $K^M_n(A,I) = {\rm Ker}(K^M_n(A) \surj K^M_n(A/I))$.
The Milnor $K$-theory of local rings was also defined by Kerz in
\cite{Kerz09}. The two definitions agree for $n \le 2$.
However, this may not be the case when $n \ge 3$ and 
$A$ has finite residue field.
Since we are interested only in $n \le 2$ case in this paper,
we shall ignore this subtlety. 

We let $\wh{K}^M_n(A)$ denote the improved Milnor $K$-group of $A$ 
as defined in \cite{Kerz}.
Then the canonical map from the Milnor to the Quillen
$K$-theory of $A$ has a factorization
$K^M_n(A) \to \wh{K}^M_n(A) \to K_n(A)$.
The first map is surjective for all $n \ge 0$ by
\cite[Theorem~13]{Kerz}. It is an isomorphism when
$A$ is a field. The second map
is an isomorphism when $n \le 2$ by \cite[Propositions~2, 10]{Kerz}.
Since we shall only use $\wh{K}^M_n(A)$ for $n \le 2$, we shall make
no difference between $\wh{K}^M_n(A)$ and $K_n(A)$ and denote the
common group by the latter notation.
When the residue field of $A$ is infinite, then
the map $K^M_n(A) \to K_n(A)$ is an isomorphism for 
$n \le 2$ by \cite[Proposition~10]{Kerz}. It follows that the map
$K^M_n(A,I) \to \ov{K}_n(A,I)$ is also an isomorphism in this case,
where the latter group is defined to be the kernel of the
restriction map $K_n(A) \to K_n(A/I)$.

We now fix a field $k$.
Let $X$ be a reduced quasi-projective surface over $k$
and let $Y \subset X$ be a nowhere dense closed subscheme containing
$X_\sing$. Let $\sI_Y$ denote the sheaf of ideals on $X$ which 
defines $Y$. We let $\sI_{Y,x}$ denote the
stalk of $\sI_Y$ at a point $x \in X$.
For any integer $n \ge 0$, we let $\sK^M_{n, (X,Y)}$ denote the
Zariski (resp. Nisnevich) sheaf on $X$ whose stalk at a point $x \in X$ is the 
Milnor $K$-group $K^M_n(\sO_{X,x}, \sI_{Y,x})$ (resp.
$K^M_n(\sO^h_{X,x}, \sI^h_{Y,x})$). The sheaf $\sK_{n,(X,Y)}$ is defined similarly.
We write $\sK^M_{n, (X,\emptyset)}$ as $\sK^M_{n, X}$.
We use a similar notation for $\sK_{n, (X,\emptyset)}$.

Let $\tau$ denote Zariski or
Nisnevich topology and assume $n \le 2$.
Then it follows from the above that the canonical map
of $\tau$-sheaves $\sK^M_{n,X} \to {\sK}_{n,X}$ is surjective and
generically an isomorphism. It is clear that the map
$\sK^M_{n, (X,Y)} \to \ov{\sK}_{n,(X,Y)}$ is generically an isomorphism.
It is shown in \cite[Theorem~2.1]{DS} (see the proof of
\cite[Proposition~9.9]{Kato-Saito}) that this map is surjective.
If $k$ is infinite, then both maps are isomorphisms. 
Since the surjection
$\sK_{n,(X,Y)} \surj \ov{\sK}_{n,(X,Y)}$ is an isomorphism away from $Y$,
we therefore get the following.

\begin{lem}\label{lem:K-iso}
Suppose $n \le 2$. Then the following hold.
\begin{enumerate}
\item
$H^2_{\tau}(X, \sK_{n,(X,Y)})  \to H^2_{\tau}(X, \ov{\sK}_{n,(X,Y)})$
is an isomorphism. 
\item
The map $\sK^M_{n, X} \to \sK_{n,X}$ induces
an isomorphism
\[
H^2_{\tau}(X, \sK^M_{n,X}) \xrightarrow{\cong} H^2_{\tau}(X, \sK_{n,X}).
\]
\item 
The map $\sK^M_{n,(X,Y)} \surj \ov{\sK}_{n,(X,Y)}$ induces 
an isomorphism
\[
H^2_{\tau}(X, \sK^M_{n,(X,Y)}) \xrightarrow{\cong} H^2_{\tau}(X, \sK_{n,(X,Y)}). 
\]
\end{enumerate}
\end{lem}

We next show the following.

\begin{lem}\label{lem:Zar-Nis}
The change of topology maps
\begin{equation}\label{eqn:Zar-Nis-0}
\lambda_X \colon H^2_{\zar}(X, \sK^M_{2,X}) 
\to H^2_{\nis}(X, \sK^M_{2,X});
\end{equation} 
\begin{equation}\label{eqn:Zar-Nis-1}
\lambda_X \colon H^2_{\zar}(X, \sK_{2,X}) 
\to H^2_{\nis}(X, \sK_{2,X})
\end{equation} 
are injective.
\end{lem}
\begin{proof}
In view of \lemref{lem:K-iso}, both statements are equivalent.
So we shall show the injectivity of the second map.

The Thomason-Trobaugh descent spectral sequence
and its compatibility with respect to change of topologies
yield a commutative diagram with exact rows
\begin{equation}\label{eqn:Coh-iso-0}
\xymatrix@C.8pc{
K_1(X) \ar@{=}[d] \ar[r] &  H^0_\zar(X, \sK_{1,X}) \ar[r]^-{d^{0,1}} \ar[d] &
H^2_\zar(X, \sK_{2,X}) \ar[r]^-{\gamma_X} \ar[d] & K_0(X) \ar@{=}[d] \\
K_1(X) \ar[r] &  H^0_\nis(X, \sK_{1,X}) \ar[r]^-{d^{0,1}} &
H^2_\nis(X, \sK_{2,X}) \ar[r]^-{\gamma_X} & K_0(X),}
\end{equation}
where the vertical arrows are the change of topology maps and the maps $\gamma_X$ are the edge homomorphisms for the Thomason-Trobaugh spectral sequence.

It is shown in \cite[Lemma~2.1]{Krishna-1} (see also 
\cite[p.~162]{KSri}) that the map
$K_1(X) \to H^0_\tau(X, \sK_{1,X})$ is naturally split surjective
when $\tau$ is either Zariski or Nisnevich topology.
Hence, we get injective maps
\begin{equation}\label{eqn:Coh-iso-1}
H^2_\zar(X, \sK_{2,X}) \inj H^2_\nis(X, \sK_{2,X}) \inj K_0(X).
\end{equation}
This finishes the proof.
\end{proof}

\subsection{The Bloch-Kato map for singular surfaces}\label{sec:B-map}
We now assume again that $X$ is a reduced quasi-projective surface
over an arbitrary field $k$ and $Y$ is any closed
subscheme of $X$ containing $X_\sing$.
Let $x \in X \setminus Y$ be a closed point. The excision theorem for
the cohomology with support shows that the restriction map
$H^2_{\{x\}}(X, \sK^M_{2,X}) \to H^2_{\{x\}}(X_\reg, \sK^M_{2,X_\reg})$
is an isomorphism.
On the other hand, we have a commutative diagram of Zariski cohomology
groups
\begin{equation}\label{eqn:Quillen-Gersten}
\xymatrix@C.8pc{
H^2_{\{x\}}(X_\reg, \sK^M_{2,X_\reg}) \ar[r] \ar[d] &
H^2(X, \sK^M_{2,X}) \ar[d]^-{\cong} \\
H^2_{\{x\}}(X_\reg, \sK_{2,X_\reg}) \ar[r] \ar@{.>}[ur]^-{\alpha_x} &
H^2(X, \sK_{2,X}),}
\end{equation}
where the horizontal arrows are the `forget support' maps and the
vertical arrows are induced by the canonical Milnor to Quillen 
$K$-theory map. The right vertical arrow is an isomorphism
by \lemref{lem:K-iso}. It follows that there is a unique
arrow $\alpha_x$ as indicated such that the inner triangles and the outer
square in ~\eqref{eqn:Quillen-Gersten} commute.
We remark that the left vertical arrow in this diagram is
also an isomorphism if $X_\reg$ is smooth over $k$
(see \cite[Theorem~2]{Kato}), but we do not need it.

The Gersten resolution of $\sK_{2,X_\reg}$ 
\cite[Theorem~5.11]{Quillen} shows that
$H^2_{\{x\}}(X_\reg, \sK_{2,X_\reg}) \cong K^{\{x\}}_0(X) \cong K_0(k(x)) \cong 
\Z$. Hence, $x$ defines a unique cycle class $[x] \in 
H^2_{\{x\}}(X_\reg, \sK_{2,X_\reg})$. 
Taking its image under $\alpha_x$, we get a well-defined
cycle class $\rho_{(X,Y)}([x]) \in H^2_{\zar}(X, \sK^M_{2,X})$.
Extending this assignment linearly on $\sZ_0(X, Y)$,
we get a map $\rho_{(X,Y)} \colon \sZ_0(X, Y) \to H^2_{\zar}(X, \sK^M_{2,X})$.
Composing this with the change of topology map
$\lambda_X \colon
H^2_{\zar}(X, \sK^M_{2,X}) \to H^2_{\nis}(X, \sK^M_{2,X})$,
we obtain maps

\begin{equation}\label{eqn:Bloch-map-0-0}
\sZ_0(X, Y) \xrightarrow{\rho_{(X,Y)}} 
H^2_{\zar}(X, \sK^M_{2,X}) \xrightarrow{\lambda_X}
H^2_{\nis}(X, \sK^M_{2,X}).
\end{equation}

We also have the cycle class map $cyc_{(X,Y)} \colon \sZ_0(X, Y) \to
K_0(X,Y) \to K_0(X)$ from ~\eqref{eqn:CCM-lci}. These maps are related by the
the following.

\begin{lem}\label{lem:Bloch-cycle-maps}
There is a commutative diagram
\begin{equation}\label{eqn:Bloch-cycle-maps-0}
\xymatrix{
\sZ_0(X, Y) \ar[r]^-{\rho_{(X, Y)}} \ar[drr]_-{cyc_{(X,Y)}} & 
H^2_\zar(X, \sK^M_{2,X})
\ar@{^{(}->}[r]^-{\lambda_X} & H^2_\nis(X, \sK^M_{2,X}) 
\ar@{^{(}->}[d]^-{\gamma_X} \\
& & K_0(X).}
\end{equation}
In particular, $\rho_{(X,Y)}$ factors through its quotient
$\CH_0(X,Y)$ if $Y$ is nowhere dense in $X$.
\end{lem}
\begin{proof}
The commutativity of the diagram is \cite[Lemma~3.2]{GK}.
Note that the proof of the cited result uses no assumption on the
nature of the field $k$.
The injectivity of $\lambda_X$ and $\gamma_X$ is ~\eqref{eqn:Coh-iso-1}.
The last assertion follows from the injectivity of $\gamma_X \circ 
\lambda_X$ and ~\eqref{eqn:CCM-lci}.
\end{proof}

\begin{lem}\label{lem:Bloch-surjection}
The map $\sZ_0(X,Y) \xrightarrow{\rho_{(X,Y)}} H^2_\zar(X, \sK^M_{2,X})$ 
is surjective if $Y$ is nowhere dense in $X$. 
The map $\lambda_X$ is an isomorphism if $X \setminus Y$
is furthermore smooth over $k$ (e.g., if $k$ is perfect).
\end{lem}
\begin{proof}
Suppose first that $k$ is infinite.
Let $\alpha \in \sZ_0(X, X_\sing)$ be a 0-cycle.
Since $k$ is infinite, we can use \cite[Lemma~1.3]{Levine-3}
to find a complete intersection reduced curve $C \subset X$
such that $C \cap Y$ is finite and $C \cap X_\reg$ is regular.
Moreover, the latter contains the support of
$\alpha$. This yields a push-forward map
(e.g., see \cite[Lemma~1.8]{ESV}) $\Pic(C) \to \CH_0(X, X_{\sing})$.
Since $Y$ is nowhere dense in $X$,
it is well known that $\Pic(C)$ is generated by
closed points away from $C \cap Y$ (note that
$C_\sing \subset Y$ because $X_\sing \subset Y$ and 
$C \cap X_\reg$ is regular), it follows that
the canonical map $\sZ_0(X,Y) \to \CH^L_0(X, X_\sing) = \CH^L_0(X)$
is surjective. On the other hand, the composite map
$\CH^L_0(X) \surj \CH_0(X) \xrightarrow{\rho_X} 
H^2_\zar(X, \sK^M_{2,X})$ is surjective by \cite[p.~169]{Levine-1}
and \lemref{lem:K-iso}.
We have thus shown that 
the map $\rho_{(X,Y)}$, which factorizes as
$\sZ_0(X,Y) \surj \CH^L_0(X) \surj \CH_0(X) \xrightarrow{\rho_X}
H^2_\zar(X, \sK^M_{2,X})$, is surjective.

If $k$ is finite, then $X \setminus Y$
is smooth over $k$ and dense open in $X$.
So it suffices now to prove the lemma when $X \setminus Y$
is smooth over $k$ and dense open in $X$.
In the latter case, $\lambda_X \circ \rho_{(X,Y)}$ is surjective by 
\cite[Theorem~2.5]{Kato-Saito}. In particular,
$\lambda_X$ is surjective. We are therefore done by 
\lemref{lem:Zar-Nis}.
\end{proof}

We can summarize the above results in the form of the following.

\begin{thm}\label{thm:CCM-main}
Let $X$ be a reduced quasi-projective surface over a field $k$.
Let $Y \subset X$ be a nowhere dense closed subscheme containing $X_\sing$.
Then there are canonical maps
\begin{equation}\label{eqn:CCM-main-0}
\CH_0(X,Y) \xrightarrow{\rho_{(X,Y)}} H^2_\zar(X, \sK^M_{2,X})
\stackrel{\lambda_X}{\inj} H^2_\nis(X, \sK^M_{2,X}) 
\stackrel{\gamma_X}{\inj} K_0(X)
\end{equation}
such that the composite $\gamma_X \circ \lambda_X \circ \rho_{(X,Y)}$
is the cycle class map $cyc_{(X,Y)}$ of ~\eqref{eqn:CCM-lci}.
The map $\rho_{(X,Y)}$ is surjective.
The map $\lambda_X$ is an isomorphism if $X \setminus Y$ is smooth over $k$ 
(e.g., if $k$ is perfect).
\end{thm}

\vskip .2cm

When $X$ has only one singular closed point and $k$ is infinite, then
$\rho_{(X,Y)}$ was shown to be an isomorphism by Collino \cite{Collino}. 
When $X$ is a normal surface, $\rho_{(X,Y)}$ was shown to be an
isomorphism by Pedrini and Weibel \cite{PW-1}. 
Note here that Collino and Pedrini-Weibel prove their results
only for $H^2_\zar(X, \sK_{2,X})$ (i.e., they use Quillen $K$-theory
sheaf). When $k$ is
algebraically closed, $\rho_{(X,Y)}$ was shown to be an
isomorphism by Levine \cite{Levine-1}.
We shall extend Levine's result to all infinite fields.

\subsection{The Bloch-Kato map for cycles with modulus}
\label{sec:BQM-mod}
We shall now prove the modulus version of \thmref{thm:CCM-main}.
In order to do this, we need to generalize the construction of
\S~\ref{sec:B-map} little bit.
We fix an arbitrary field $k$.

Let $X$ be a reduced quasi-projective surface over $k$ and let
$Y \subset X$ be a closed subscheme such that 
$X^o := X \setminus Y$ is smooth over $k$.
Notice that we did not impose this extra condition
in \S~\ref{sec:B-map}.

Let $x \in X^o$ be a closed point. The excision theorem for
the cohomology with support shows that the restriction map
$H^2_{\{x\}}(X, \sK^M_{2,(X,Y)}) \to H^2_{\{x\}}(X^o, \sK^M_{2,X^o})$
is an isomorphism (in the Zariski topology). 
Since $X^o$ is smooth over $k$,
Kato showed that the Gersten complex for $\sK^M_{2,X^o}$ is exact
except at its left-most end. This implies that
$H^2_{\{x\}}(X^o, \sK^M_{2,X^o}) \cong K^{\{x\}}_0(X) \cong K_0(k(x)) \cong 
\Z$ (see \cite[Theorem~2]{Kato}).
Hence, $x$ defines a unique cycle class $[x] \in 
H^2_{\{x\}}(X, \sK^M_{2,(X,Y)})$. The image of this cycle class under the
`forget support' map $H^2_{\{x\}}(X, \sK^M_{2,(X,Y)}) \to H^2(X, \sK^M_{2,(X,Y)})$
yields a cycle class $[x] \in H^2_{\zar}(X, \sK^M_{2,(X,Y)})$.
Extending this assignment linearly on $\sZ_0(X, Y)$,
we get a map $\rho_{(X,Y)} \colon \sZ_0(X, Y) \to H^2_{\zar}(X, \sK^M_{2,(X,Y)})$.
Composing this with the change of topology map
$\lambda_{(X,Y)} \colon
H^2_{\zar}(X, \sK^M_{2,(X,Y)}) \to H^2_{\nis}(X, \sK^M_{2,(X,Y)})$, we
obtain maps
\begin{equation}\label{eqn:Bloch-map-0}
\sZ_0(X, Y) \xrightarrow{\rho_{(X,Y)}} 
H^2_{\zar}(X, \sK^M_{2,(X,Y)}) \xrightarrow{\lambda_{(X,Y)}}
H^2_{\nis}(X, \sK^M_{2,(X,Y)}).
\end{equation}

We now let $k$ be a as above and let $X$ be a smooth quasi-projective
surface over $k$. Let $D \subset X$ be an effective Cartier divisor.
Recall that the Thomason-Trobaugh descent spectral sequence
yields the Nisnevich descent map
$H^2_\nis(X, \sK_{2,(X,D)}) \to K_0(X,D)$.
We let $\gamma_{(X,D)}$ denote its composition with the 
canonical map $H^2_\nis(X, \sK^M_{2,(X,D)}) \to
H^2_\nis(X, \sK_{2,(X,D)})$.

\begin{thm}\label{thm:CCM-main-mod}
Let $X$ be a smooth quasi-projective surface over $k$.
Then there are canonical maps
\begin{equation}\label{eqn:CCM-main-mod-0}
\CH_0(X|D) \xrightarrow{\rho_{X|D}} H^2_\zar(X, \sK^M_{2,(X,D)})
\stackrel{\lambda_{(X,D)}}{\underset{\cong}{\to}} H^2_\nis(X, \sK^M_{2,(X,D)}) 
\xrightarrow{\gamma_{(X,D)}} K_0(X,D)
\end{equation}
such that the composite $\gamma_{(X,D)} \circ \lambda_{(X,D)} \circ \rho_{X|D}$
is the cycle class map $cyc_{X|D}$ of ~\eqref{eqn:CCM-modulus}.
The map $\rho_{X|D}$ is surjective.
\end{thm}
\begin{proof}
%We let $\rho_{X|D} := \rho_{(X,D)}$, where the latter is defined in ~\eqref{eqn:Bloch-map-0}.
We first show that $cyc_{X|D} = 
\gamma_{(X,D)} \circ \lambda_{(X,D)} \circ \rho_{(X,D)}$ on $\sZ_0(X,D)$,
where $\rho_{(X,D)}$ is as in ~\eqref{eqn:Bloch-map-0}.
It is enough to check this for every closed point $x \in X \setminus D$.
We let $\rho^{\nis}_{(X,D)} = \lambda_{(X,D)} \circ \rho_{(X,D)}$.

We let $S_X$ be the double of $X$ along $D$. We identify $X$ as $X_+ \subset
S_X$. It is then immediate from the construction of the cycle class
map on $\sZ_0(X,D) = \sZ_0(X_+, D)$ in  \S~\ref{sec:CCMaps} that
the composite map
\begin{equation}\label{eqn:CCM-main-mod-1}
\sZ_0(X_+, D) = \sZ_0(S_X, X_-) \xrightarrow{cyc_{(S_X, X_-)}} 
K_0(S_X, X_-) \xrightarrow{\iota^*_+}
K_0(X,D)
\end{equation}
coincides with $cyc_{X|D}$.
We now consider the diagram
\begin{equation}\label{eqn:CCM-main-mod-2}
\xymatrix@C.8pc{
& \sZ_0(S_X, X_-) \ar[r]^-{\rho^{\nis}_{(S_X,X_-)}} &
H^2_\nis(S_X, \sK^M_{2,(S_X,X_-)}) \ar[r]^-{\gamma_{(S_X,X_-)}} 
\ar[d]^-{\iota^*_+} & K_0(S_X, X_-) \ar[d]^-{\iota^*_+} \\
\sZ_0(X, D) \ar@{=}[r] & \sZ_0(X_+, D) \ar@{^{(}->}[u] 
\ar[r]^-{\rho^{\nis}_{(X,D)}}& H^2_\nis(X, \sK^M_{2,(X,D)}) 
\ar[r]^-{\gamma_{(X,D)}} &  K_0(X,D).}
\end{equation}

Here, $\gamma$ denotes again the edge homomorphism for the Thomason-Trobaugh
spectral sequence. The map $\rho^{\nis}_{(S_X,X_-)}$ is the composition of the
two maps in ~\eqref{eqn:Bloch-map-0}.
Like $cyc_{X|D}$, it follows immediately from the constructions
of $\rho^{\nis}_{(S_X,X_-)}$ and $\rho^{\nis}_{(X,D)}$ that the left 
square in ~\eqref{eqn:CCM-main-mod-2} commutes. 
%Note here that as $X$ is smooth over $k$, the left vertical arrow
%in ~\eqref{eqn:Quillen-Gersten} is an isomorphism for $S_X$.
The right square in ~\eqref{eqn:CCM-main-mod-2}  commutes
by the functoriality of the Thomason-Trobaugh spectral sequence
with respect to the inclusion $(X_+, D) \inj (S_X, X_-)$.
Using ~\eqref{eqn:CCM-main-mod-1}, it suffices therefore to show
that composite map on the top row of ~\eqref{eqn:CCM-main-mod-2}
is $cyc_{(S_X,X_-)}$.

To prove this last assertion, we consider the commutative diagram
\begin{equation}\label{eqn:CCM-main-mod-3}
\xymatrix@C.8pc{
0 \ar[r] & \sZ_0(S_X, X_-) \ar[d]_-{\rho^{\nis}_{(S_X, X_-)}} \ar[r]^-{p_{+ *}} &
\sZ_0(S_X, D) \ar[d]^-{\rho^{\nis}_{(S_X,D)}} \ar[r]^-{\iota^*_-}  
& \sZ_0(X \setminus D)
\ar[d]^-{\rho^{\nis}_{X}} \ar[r]  & 0 \\
0 \ar[r] & H^2_\nis(S_X, \sK^M_{2,(S_X,X_-)}) \ar[r]^-{p_{+ *}}
\ar[d]_-{\gamma_{(S_X,X_-)}} & 
H^2_\nis(S_X, \sK^M_{2,S_X}) \ar[d]^-{\gamma_{S_X}} 
\ar[r]^-{\iota^*_-} & H^2_\nis(X, \sK^M_{2,X}) \ar[d]^-{\gamma_{X}}
\ar[r] & 0 \\
0 \ar[r] & K_0(S_X, X_-) \ar[r]^-{p_{+ *}} & K_0(S_X) \ar[r]^-{\iota^*_-} &
K_0(X) \ar[r] & 0.}
\end{equation}
 
The three rows are split exact. 
It follows from \lemref{lem:Bloch-cycle-maps} that 
\begin{equation}\label{eqn:CCM-main-mod-4}
p_{+ *} \circ cyc_{(S_X,X_-)} = cyc_{(S_X, D)} \circ p_{+ *} =
\gamma_{S_X} \circ {\rho^{\nis}_{(S_X,D)}} \circ p_{+ *}.
\end{equation}
Since ~\eqref{eqn:CCM-main-mod-3} commutes, we also have
$p_{+ *} \circ \gamma_{(S_X,X_-)} \circ \rho^{\nis}_{(S_X, X_-)} =
\gamma_{S_X} \circ {\rho^{\nis}_{(S_X,D)}} \circ p_{+ *}$.
Combining this with ~\eqref{eqn:CCM-main-mod-4} and using that
$p_{+ *}$ is injective, we get
$cyc_{(S_X,X_-)} = \gamma_{(S_X,X_-)} \circ \rho^{\nis}_{(S_X, X_-)}$.
We have therefore shown that 
$\gamma_{(X,D)} \circ \lambda_{(X,D)} \circ \rho_{(X,D)}$
in ~\eqref{eqn:CCM-main-mod-0} is $cyc_{X|D}$,
which we wanted to show.

We now show that $\rho_{(X,D)}$ factors through the rational 
equivalence (this is due to R\"ulling-Saito \cite{RS} if $D_{\rm red}$ is a SNCD on $X$). For this, we note that ~\eqref{eqn:CCM-main-mod-2} and
~\eqref{eqn:CCM-main-mod-3} remain valid if we replace the Nisnevich 
cohomology by the Zariski cohomology.
Suppose now that $\alpha \in \sR_0(X|D)$.
Then $p_{+ *}(\alpha) \in \sR_0(S_X,D)$ by 
\thmref{thm:BS-main}. It follows from \thmref{thm:CCM-main} (applied to
$S_X$) that
$\rho_{(S_X, D)} \circ p_{+ *}(\alpha) = 0$.
This implies by ~\eqref{eqn:CCM-main-mod-3} that
$\rho_{(S_X, X_-)}(\alpha) = 0$ in $H^2_\zar(S_X, \sK^M_{2,(S_X,X_-)})$.
Since the left square in ~\eqref{eqn:CCM-main-mod-2}
is commutative, we deduce that $\rho_{(X,D)}(\alpha)  = 0$
in $H^2_\zar(X, \sK^M_{2,(X,D)})$. We have thus shown that
$\rho_{(X,D)}$ factors through its quotient $\CH_0(X|D)$.
We let $\rho_{X|D} \colon \CH_0(X|D) \to H^2_\zar(X, \sK^M_{2,(X,D)})$
be the induced map. 

The map $\lambda_{(X,D)}$ is an isomorphism by a combination of
\lemref{lem:K-iso} and \cite[Proposition~9.8]{Kato-Saito}
(or \cite[Lemma~2.1]{Krishna-1}).
It remains to show that $\rho_{X|D}$ is surjective.
For this, it suffices to show that the map
$\sZ_0(X,D) \xrightarrow{\rho_{(X,D)}} H^2_\zar(X, \sK^M_{2,(X,D)})$
is surjective. 
Now, we note in ~\eqref{eqn:CCM-main-mod-3} that all rows are
compatibly split, where all splittings are given by the pull-back
via the flat projection $\pi \colon S_X \to X$. 
It follows therefore from \thmref{thm:CCM-main} that 
$\rho_{(S_X,D)}$ and $\rho_X$ are surjective. 
But this implies that $\rho_{(S_X, X_-)}$ is also surjective.
Using ~the left square of \eqref{eqn:CCM-main-mod-2},
it suffices now to show that the restriction map
$\iota^*_+ \colon H^2_\zar(S_X, \sK^M_{2,(S_X,X_-)}) \to
H^2_\zar(X, \sK^M_{2,(X,D)})$ is surjective.
But this is clear because the map of
Zariski sheaves $\sK^M_{2,(S_X,X_-)} \xrightarrow{\iota^*_+} \sK^M_{2,(X,D)}$
is an isomorphism away from $D$ and $\dim(S_X) = 2$.
\end{proof}

\subsection{Some functoriality properties}\label{sec:Functor-Chow}
We shall use the following functoriality properties of the
Chow group of 0-cycles and their cycle classes. For any $Y \in \Sch_k$,
recall our notation that  $\CH_*^F(Y)$ is the direct sum of all
Chow groups of $Y$ in the sense of \cite[Chapter~1]{Fulton}.

\begin{prop}\label{prop:Localization}
Let $k$ be any field and let $X$ be a reduced quasi-projective surface over
$k$. Let $j \colon U \inj X$ be an open immersion such that
$X_\sing \subset U$. Then we have an exact sequence
of abelian groups
\[
\CH^F_0(X \setminus U) \xrightarrow{i_*} \CH^L_0(X) 
\xrightarrow{j^*} \CH^L_0(U) \to 0.
\]
The same holds for the lci Chow group.
\end{prop}
\begin{proof}
The proofs of the proposition for the lci and the Levine-Weibel Chow groups 
are completely identical. We give the proof for the latter case.
We have to first explain the maps $i_*$ and $j^*$.
It is clear that the restriction of cycles with respect to the open
immersion $j$ indeed gives the map $j^*$. It  is clearly surjective
because $X_\sing \subset U$. 

The map $i_*$ is the push-forward map
with respect to the closed immersion $i \colon Z := X \setminus U \inj X$.
It is clearly well-defined at the level of cycles. If 
$C \subset Z \times \P^1_k$ is an irreducible 1-cycle dominant over
$\P^1_k$, then it is also clear (note that $X_\sing \subset U$)
that $i_*(C)$ is a Cartier curve on
$\P^1_X$ which defines the rational equivalence between
$i_*(C_0 - C_\infty)$. Hence, the maps in the above sequence are
defined. 
We only need to see that this sequence is exact in the middle.

So suppose that $\alpha = \stackrel{r}{\underset{i =1}\sum} n_i[x_i]$
is a 0-cycle on $X$ such that $j^*(\alpha) = 0$. 
We can assume without loss of generality that $x_i \in U$ for
each $i$. We let $C^1, \ldots , C^s$ be Cartier curves on
$\P^1_U$ such that 
$j^*(\alpha) = \stackrel{s}{\underset{i =1}\sum} (C^i_0 - C^i_\infty)$.
We let $\ov{C}^{i}$ be the scheme-theoretic closure of $C^i$ is $\P^1_X$.
Since $X_\sing \subset U$, it follows that $\ov{C}^{i}$ is
a Cartier curve on $\P^1_X$ and $\ov{C}^{i}_0 - \ov{C}^{i}_\infty =
(C^i_0 - C^i_\infty) + \beta_i$, where $\beta_i$ is a 0-cycle supported on
$Z$. Letting $\beta = \stackrel{s}{\underset{i =1}\sum} \beta_i$,
we get $\alpha - \beta = \stackrel{s}{\underset{i =1}\sum} 
(\ov{C}^{i}_0 - \ov{C}^{i}_\infty)$. This finishes the proof.
\end{proof}

\begin{prop}\label{prop:Desingularization}
Let $k$ be an infinite 
field and let $X$ be a reduced quasi-projective surface over
$k$. Let $\pi \colon \wt{X} \to X$ be a resolution of singularities of $X$.
Then we have a commutative diagram
\begin{equation}\label{eqn:Desingulization-0}
\xymatrix@C.8pc{
\CH^L_0(X) \ar[r]^-{cyc^L_X} \ar[d]_-{\pi^*} & K_0(X) \ar[d]^-{\pi^*} \\
\CH^F_0(\wt{X}) \ar[r]^-{cyc_{\wt{X}}} & K_0(\wt{X}).}
\end{equation}
The same holds for the lci Chow group  over any field.
\end{prop}
\begin{proof}
Using \corref{cor:CCM-Levine-lci}, it suffices to prove the
proposition for the lci Chow group.
We need to first define the vertical arrow on the left. All other
maps are defined. Since the map $\pi^{-1}(X_\reg) \to X_\reg$ is an
isomorphism, the pull-back $\pi^* \colon \sZ_0(X, X_\sing) \to \sZ_0(\wt{X})$
is simply the inclusion map. It is now an easy exercise using the
definitions of rational equivalences to check that
this map preserves the subgroups of rational equivalences.
We shall nonetheless provide a proof which simultaneously shows that
~\eqref{eqn:Desingulization-0} is commutative.

We now use \lemref{lem:Bloch-cycle-maps} to get a diagram
\begin{equation}\label{eqn:Desingulization-1}
\xymatrix@C.8pc{
\sZ_0(X, X_\sing) \ar[r]^-{\rho_X} \ar[d]_-{\pi^*} & H^2_{\zar}(X, \sK^M_{2, X})
\ar[d]^-{\pi^*} \ar[r]^-{\lambda_X} & K_0(X) \ar[d]^-{\pi^*} \\
\CH^F_0(\wt{X}) \ar[r]^-{\rho_{\wt{X}}} & H^2_{\zar}(\wt{X}, \sK^M_{2, \wt{X}})
\ar[r]^-{\lambda_{\wt{X}}} & K_0(\wt{X}).}
\end{equation}

It is clear from the various definitions above that this diagram
is commutative. The classical Bloch's formula for non-singular surfaces
and \lemref{lem:K-iso} together imply that $\rho_{\wt{X}}$ is an isomorphism. 
We now conclude from \lemref{lem:Bloch-cycle-maps} that the 
left vertical arrow kills $\sR_0(X, X_\sing)$ and 
~\ref{eqn:Desingulization-0} commutes.
\end{proof}

\begin{lem}\label{lem:Cycle-K-0}
Let $k$ be any field and let $X$ be a reduced quasi-projective 
surface over $k$. Let $k \inj k'$ be a separable algebraic extension 
(possibly infinite) of
fields. Let $X' = X_{k'}$ and let ${\rm pr}^*_{{k'}/{k}} \colon X' \to X$ be
the projection map. Then the diagram
\begin{equation}\label{eqn:Cycle-K-0-0}
\xymatrix@C.8pc{
\CH_0(X) \ar[r]^-{cyc_X} \ar[d]_-{{\rm pr}^*_{{k'}/{k}}} &
K_0(X) \ar[d]^-{{\rm pr}^*_{{k'}/{k}}} \\
\CH_0(X') \ar[r]^-{cyc_{X'}} & K_0(X')}
\end{equation}
is commutative.
\end{lem}
\begin{proof}
Let $x \in X_\reg$ be a closed point and let $Y = \Spec(k(x)) \times_X X'$.
Since $k'$ is separable and algebraic over $k$, it follows that
$Y$ is a 0-dimensional reduced closed subscheme of $X'$.
We let $[Y] = {\rm pr}^*_{{k'}/{k}}([x]) \in \sZ_0(X', X'_\sing)$. 
We let $\iota_x \colon \Spec(k(x)) \inj X$ and
$\iota_Y \colon Y \inj X'$ be the inclusion maps.
Then the diagram
\begin{equation}\label{eqn:Cycle-K-0-1}
\xymatrix@C.8pc{
K_0(\Spec(k(x))) \ar[r]^-{(\iota_x)_*} \ar[d]_-{{\rm pr}^*_{{k'}/{k}}} &
K_0(X) \ar[d]^-{{\rm pr}^*_{{k'}/{k}}} \\
K_0(Y) \ar[r]^-{(\iota_Y)_*} & K_0(X')}
\end{equation}
is commutative by \cite[Proposition~3.18]{TT}. Since
$(\iota_Y)_*(1) = cyc_{X'}([Y])$, the lemma follows.
\end{proof}

\begin{cor}\label{cor:Cycle-K-mod}
Let $k$ be any field and let $X$ be a smooth quasi-projective 
surface over $k$. Let $D \subset X$ be an effective Cartier divisor.
Let $k \inj k'$ be a separable algebraic extension 
(possibly infinite) of
fields. Let $X' = X_{k'}, D' = D_{k'}$ and 
let ${\rm pr}_{{k'}/{k}} \colon X' \to X$ be
the projection map. Then the diagram
\begin{equation}\label{eqn:Cycle-K-mod-0}
\xymatrix@C.8pc{
\CH_0(X|D) \ar[r]^-{cyc_{X|D}} \ar[d]_-{{\rm pr}^*_{{k'}/{k}}} &
K_0(X,D) \ar[d]^-{{\rm pr}^*_{{k'}/{k}}} \\
\CH_0(X'|D') \ar[r]^-{cyc_{X'|D'}} & K_0(X',D')}
\end{equation}
is commutative.
\end{cor}
\begin{proof}
We consider the diagram
\begin{equation}\label{eqn:Cycle-K-mod-1}
\xymatrix@C.8pc{
\sZ_0(X,D) \ar[r]^-{cyc_{(S_X,X_-)}} \ar[d]_-{{\rm pr}^*_{{k'}/{k}}} &
K_0(S_X, X_-) \ar[r]^-{\iota^*_+} \ar[d]^-{{\rm pr}^*_{{k'}/{k}}} &
K_0(X,D) \ar[d]^-{{\rm pr}^*_{{k'}/{k}}} \\
\sZ_0(X',D') \ar[r]^-{cyc_{(S_{X'},X'_-)}} &
K_0(S_{X'}, X'_-) \ar[r]^-{\iota^*_+} &
K_0(X',D').}
\end{equation}
We saw in ~\eqref{eqn:CCM-main-mod-1} that the composite horizontal
arrows in this diagram are the cycle class maps $cyc_{X|D}$ and $cyc_{X'|D'}$.
Hence, it suffices to show that the two squares in
~\eqref{eqn:Cycle-K-mod-1} commute. The right square clearly
commutes. The left square commutes by applying \lemref{lem:Cycle-K-0} to 
$S_X$ and $X$ and then using  ~\eqref{eqn:CCM-main-mod-3}.
\end{proof}

\section{The family of good cycles}
\label{sec:Transversality}
Our next goal is to show that the cycle class
map $\CH^L_0(X, Y) \to K_0(X)$ is injective
if $X$ is a reduced quasi-projective surface over an infinite field.
As we stated in \S~\ref{sec:Intro}, this was proven by
Levine in his yet unpublished manuscript \cite{Levine-2}
when the base field is algebraically closed.
A published account of Levine's proof
is available in \cite{BSri}.
In the next few sections, our goal is to revisit Levine's proof
and show that it can be carried over to all infinite fields 
with the help of new arguments at every step of the proof.

In this section, we shall construct families $\phi \colon \Gamma_U \to U$
of cycles on $X$. We shall state and prove an upgraded version of
Kleiman's transversality theorem \cite[Theorem~2]{Kleiman} and
its generalization by Levine \cite[Lemma~1.2]{Levine-2}
(see also \cite[Lemma~1.1]{Bloch-1}) which works over an
arbitrary infinite field. This theorem will be used to ensure that
every member of a family $\phi$ as above gives rise to a 
well defined 0-cycle on $X$.
The proof closely follows the one
in \cite[Lemma~1.2]{Levine-2}. Unless stated otherwise, $k$ will always
denote an infinite ground field in this section.  

\subsection{Some recollection}\label{sec:Recall}
We begin with some standard algebraic geometry results that 
we shall use repeatedly. We collect them here for
reader's convenience.

\begin{lem}\label{lem:Irr-fibers}
Let $f \colon X \to Y$ be a continuous map of topological spaces
such that the following hold.
\begin{enumerate}
\item
$Y$ is irreducible.
\item
$f$ is open.
\item
There exists a dense set of points $y \in Y$ such that $f^{-1}(y)$
is irreducible. 
\end{enumerate}
Then $X$ is irreducible.
\end{lem}
\begin{proof}
See \cite[{Tag 004Z}]{stacks-project}.
\end{proof}

\begin{lem}\label{lem:Fiber-dim}
Let $f \colon X \to Y$ be an open morphism in $\Sch_k$. Assume that
$X$ is equidimensional, $Y$ is irreducible and every irreducible
component of $X$ is dominant over $Y$. Then $f$ has equidimensional
fibers. That is,
\[
\dim(f^{-1}(y)) = \dim(X) - \dim(Y) \ \ \mbox{for \ all} \ y \in Y.
\]
\end{lem}
\begin{proof}
See \cite[Theorem~14.114]{Gortz}.
\end{proof}

\begin{cor}\label{cor:Fiber-dim-0}
Let $f \colon X \to Y$ be an open as well as closed morphism in $\Sch_k$.
Assume that $X$ is equidimensional and $Y$ is irreducible.
Then $f$ has equidimensional fibers. 
\end{cor}
\begin{proof}
Under our assumption, every irreducible component of $X$ must be
dominant over $Y$ and therefore \lemref{lem:Fiber-dim} applies.
\end{proof}

\begin{lem}\label{lem:Equi-dim}
Let $X, Y \in \Sch_k$ be equidimensional schemes.
Then $X \times Y$ is equidimensional.
\end{lem}
\begin{proof}
%Since $(X \times Y)_\red = (X_\red \times Y_\red)_\red$, we can assume $X, Y$ and $X \times Y$ are reduced. 
Let $\ov{k}$ be
an algebraic closure of $k$.
Since $\dim(X \times Y) = \dim((X \times Y)_{\ov{k}})$ (e.g., see
\cite[Proposition~5.38]{Gortz}) and
$(X \times Y)_{\ov{k}} \cong X_{\ov{k}} \times_{\ov{k}} Y_{\ov{k}}$,
we can assume $k$ is algebraically closed.
In this case, the lemma follows from our assumption because
$X' \times Y'$ is irreducible of dimension equal to $\dim(X') + \dim(Y')$ if
$X', Y' \in \Sch_k$ are irreducible.
\end{proof}

\subsection{Morphisms to homogeneous spaces}\label{sec:Set-up}
We recall homogeneous spaces and describe some properties of
morphisms from schemes to these spaces. The set-up of
this section will be used throughout the next few sections of this
paper.

Let $G$ be a connected reductive algebraic group over $k$ and let
$H = G/P$ be a projective homogeneous space for $G$.
Recall that a smooth closed subgroup $P' \subset G$ is parabolic
if and only if the variety $G/{P'}$ is complete (equivalently, projective).
It follows that $P$ is a parabolic subgroup of $G$.
In particular, it is connected (e.g., see \cite[Theorem~11.16]{Borel}).
We shall let $d_P = \dim(P)$.

Let $\pi \colon G \to H$ be the resulting $P$-torsor. 
Note that since $P$ is affine, $\pi$ is an fppf locally trivial
$P$-torsor, it follows that $\pi$ is fppf locally an affine morphism.
But this implies that $\pi$ is an affine morphism. 
Since $G$ is reductive, it
is smooth over $k$ and  an fppf descent  argument shows that
$H$ must also be smooth over $k$. Furthermore, as $G$ is connected,
it must be geometrically connected (any $X \in \Sch_k$ which is
connected and $X(k) \neq \emptyset$ is geometrically connected).
It follows that $G$ is geometrically integral. 
In particular, $H = G/P$ is also geometrically integral.

Another property we shall use frequently is that, being reductive, $G$
is a uni-rational variety over $k$ (i.e., admits a dominant $k$-morphism
from a dense open subset of an affine space over $k$)
by \cite[Theorem~18.2]{Borel}.
It follows that $H$ is also uni-rational. Since $k$ is infinite,
this implies that for any dense open $U \subset H$ (or in $G$),
the set of $k$-rational points $U(k)$ is Zariski 
dense in $H$ (or in $G$). 

Let $G$ act on a reduced quasi-projective scheme $X$ over $k$ and let
$\mu \colon G \times X \to X$ be the action map. Let 
$\Phi = (\mu, \id_X) \colon G \times X \to X \times X$ denote the map
$\Phi(g, x) = (\mu(g,x), x)$.
Let $\wt{\mu} = (\id_G, \mu) \colon G \times X \to G \times X$ denote the map
$\Phi(g, x) = (g, \mu(g,x))$. 
Since $\wt{\mu}$ is an isomorphism and $G$ is smooth, it follows that
$\mu$ is a smooth surjective morphism.

For a closed point $g \in G$, the composite morphism
${\lambda}_g \colon
\Spec(k(g)) \times X = X_{k(g)} \inj G \times X \xrightarrow{\mu} X$ can be
easily seen to be closed. In particular, $gY := {\lambda}_g(Y)$ is closed 
in $X$ for any closed subscheme $Y \subset X$.
If $g \in G(k)$, we can identify $X_{k(g)}$ with $X$ and then
${\lambda}_g$ defines an automorphism of $X$.
It is easy to see that $\{\lambda_g\}_{g \in G(k)}$ define
a group homomorphism $\lambda \colon G(k) \to {\rm Aut}_k(X)$.

\subsection{The parameter space of good cycles}\label{sec:PS}
We now let $H = G/P$ be a homogeneous space for $G$ as above.
Then $G$ acts transitively on $H$. This action gives rise to the maps
$\mu, \wt{\mu}$ and $\Phi$ be as above. 
Let $Y \subset H$ be an equidimensional closed subscheme. 
Let $X$ be an equidimensional reduced
quasi-projective $k$-scheme and let $f \colon X \to H$ be a $k$-morphism. 

We consider the commutative diagram
\begin{equation}\label{eqn:Homog-space-0}
\xymatrix@C.8pc{
& \Gamma \ar[r]^-{\delta'} \ar[d]^-{f'} \ar[dl]_-{\phi} & Y \times X
\ar[d]^-{\iota_Y \times f} \\
G & G \times H \ar[l]_-{p} \ar[r]^-{\delta} & H \times H,}
\end{equation}
where $\iota \colon G \to G$ is the inverse morphism
($\iota(g) = g^{-1}$), $\delta$ is the composite 
\[
G \times H \xrightarrow{\iota \times \id_H} G \times H \xrightarrow{\Phi}
H \times H
\]
and $p$ is the projection.
The scheme $\Gamma$ is defined so that the right square is
Cartesian. The map $\phi$ is defined so that the left triangle
is commutative. For $U \subset G$, we let $\Gamma_U := U \times_G \Gamma$.
We study some properties of $\Gamma$.

Since $\Phi = (\mu, \id_H)$ and $\mu$ is affine, it follows that
$\Phi$ is affine. Hence, $\delta$ is affine.
It is also clear from the definition of $\Phi$ that for every
point $z \in H \times H$, the scheme-theoretic fiber $\Phi^{-1}(z)$ is 
isomorphic to $P_{k(z)}$. In particular, $\Phi$ is equidimensional
of relative dimension $d_P$. Since the source and the
target of $\Phi$ are both regular, it follows from 
\cite[Exc.~III.10.9]{Hartshorne} that $\Phi$ is flat.
Since $P$ is smooth over $k$, it follows that $\Phi$ is a
smooth morphism. Since $\iota \times \id_H$ is an isomorphism, 
we conclude that $\delta$ is a smooth surjective affine morphism
of relative dimension $d_P$. 
This implies that $\delta'$ is also a smooth surjective affine morphism
of relative dimension $d_P$. 

Since $X$ and $Y$ are equidimensional, it follows from  
\lemref{lem:Equi-dim} that $Y \times X$ is also equidimensional.
As $P$ is geometrically integral, we see that
over every irreducible component $V$ of $Y \times X$,
the map $\delta'$ is smooth surjective with irreducible fibers.
It follows therefore from \lemref{lem:Irr-fibers} that
$\delta'^{-1}(V)$ is irreducible of dimension $\dim(V) + d_P$. 
We conclude that $\Gamma$ is equidimensional and 
\begin{equation}\label{eqn:Homog-space-3}
\dim(\Gamma) = \dim(X) + \dim(Y) + d_P.
\end{equation}

\begin{lem}\label{lem:The-Open-set}
There exists a dense open subscheme $U \subset G$ such that $\phi \colon \Gamma_U\to U$ is flat, and for all $g \in U(k)$, the scheme $f^{-1}(gY) = \phi^{-1}(g)$ is equidimensional. Furthermore, 
\[
\dim(f^{-1}(gY)) = \dim(X) -(\dim(H) - \dim(Y)).
\]
 Equivalently, $\codim_X(f^{-1}(gY)) = \codim_H(Y)$ for all  $g \in U(k)$.
\end{lem}
\begin{proof}
Since $\phi$ is surjective, at least one irreducible 
component of $\Gamma$ is dominant over $G$. 
Let us write $\Gamma_{\red}  = \Gamma' \cup \Gamma''$ as a union of
closed subschemes such that $\Gamma'$ is the union of all
irreducible components of $\Gamma$ which are dominant over $G$
and $\Gamma''$ is the union of those irreducible components which are
not dominant over $G$. We can choose a dense open
subscheme $U' \subset G$ such that $\phi^{-1}(U') \cap \Gamma'' = \emptyset$.
In particular, $\phi^{-1}(U') = \phi^{-1}(U') \cap \Gamma'$.

We let $\Gamma_{U'} = \phi^{-1}(U')$. 
Then $\phi \colon \Gamma_{U'} \to U'$ is surjective morphism
with irreducible base. Hence, it follows from the generic flatness theorem
(see EGA $\mathrm{IV_{2}}$ 6.9.1) that there is a dense open subscheme
$U \subset U'$ such that $\phi \colon \Gamma_U \to U$ is flat and
surjective. Since $\Gamma_{U} = \Gamma'_U$ and $\Gamma$ is
equidimensional, it follows that
we have a morphism $\phi \colon \Gamma_U \to U$ 
which is flat, whose source is equidimensional and each irreducible
component of the source is dominant over the base.
We can therefore apply \lemref{lem:Fiber-dim} to conclude that
$\phi \colon \Gamma_U \to U$ has equidimensional fibers.

In particular, using the identification $f^{-1}(gY)= \phi^{-1}(g)$, induced by 
the commutative diagram \eqref{eqn:Homog-space-0}, we have 
for any $g \in U(k)$,
\begin{equation}\label{eqn:Homog-space-1}
\begin{array}{lll}
\dim(f^{-1}(gY)) & = & \dim(\phi^{-1}(g)) = \dim(\Gamma') -
\dim(G) \\
& = & \dim(\Gamma) - \dim(G) \\
& = & \dim(X) + \dim(Y) + d_P - \dim(G) \\
& = & \dim(X) -(\dim(H) - \dim(Y)). 
%& = & = \dim(X) - \codim_H(Y).
\end{array}
\end{equation}
Since $U \subset G$ is dense open, the proof of the lemma is complete.
\end{proof}

We let $\alpha \colon \Gamma \to G \times X$ be the map induced by
the projections $\phi \colon \Gamma \to G$ and
$\Gamma \xrightarrow{\delta'} Y \times X \to X$. 
We let $\beta \colon G \times X \to H$ be the composite
$p_1 \circ \delta \circ (\id_G, f)$,
where $p_1 \colon H  \times H \to H$ is the first projection. 

If we let $G$ act on $G \times X$ by right multiplication
on itself and trivially on $X$ (i.e., $g_1 \cdot (g,x) = (gg_1^{-1}, x)$),
then $\beta \colon G \times X \to H$ is $G$-equivariant.
Since $\beta$ is surjective and $H$ is reduced, it follows 
from the generic flatness theorem that
$\beta$ is flat over a dense open subscheme of $H$.
Since $G$ acts transitively on $H$ and $\beta$ is $G$-equivariant,
it follows that it must actually be flat everywhere.

We now consider the diagram
\begin{equation}\label{eqn:Homog-space-2}
\xymatrix@C.8pc{
\Gamma' \ar[dd]_-{\iota'_Y} \ar[drrr]^-{\beta'} & & & \\
& \Gamma \ar[r]^-{\delta'} \ar[d]^-{f'} \ar[dl]_-{\alpha} 
\ar@{.>}[ul]^-{\gamma} & Y \times X
\ar[d]^-{\id_Y \times f} \ar[r]^-{p_1} & Y \ar[d]^-{\iota_Y}  \\
G\times X \ar@/_2pc/[rrr]^-{\beta} \ar[r]^-{\id_G \times f} 
& G \times H  \ar[r]^-{\delta} & H \times H \ar[r]^-{p_1} &
H,}
\end{equation}
where $\iota_Y$ is the inclusion, and
$\Gamma' := (G \times X) \times_H Y$ with respect to maps
$\beta$ and $\iota_Y$.
It is easy to check from the definition of $\Gamma$ in 
~\eqref{eqn:Homog-space-0} that
$\beta \circ \alpha = \iota_Y \circ p_1 \circ \delta'$.
It follows that there is a unique morphism
$\gamma \colon \Gamma \to \Gamma'$ such that $\iota'_Y \circ \gamma = 
\alpha$ and $\beta' \circ \gamma = p_1 \circ \delta'$. 
Furthermore, it is easy to check that $\gamma$ is an
isomorphism. In particular, $\beta'$ is flat (since
$p_1$ and $\delta'$ are) and $\alpha$ is
a closed immersion.

We let $U \subset G$ be the dense open as in 
\lemref{lem:The-Open-set} and let $g \in U(k)$.
We consider the (equivalent) Cartesian diagrams
\begin{equation}\label{eqn:Homog-space-4}
\xymatrix@C.8pc{
f^{-1}(gY) \ar[r]^-{\lambda_{g^{-1}} \circ f} \ar[d] & Y \ar[d]^-{\iota_Y} \\
X \ar[r]^-{\lambda_{g^{-1}} \circ f} & H,} \quad \quad
\xymatrix@C.8pc{
f^{-1}(gY) \ar[r]^-{ f} \ar[d] & gY \ar[d]^-{\iota_{gY}} \\
X \ar[r]^-{ f} & H.}
\end{equation}

\begin{lem}\label{lem:Tor-ind}
The squares in ~\eqref{eqn:Homog-space-4} are Tor-independent.
\end{lem}
\begin{proof}
We consider another commutative diagram
\begin{equation}\label{eqn:Kleiman-Levine-0}
\xymatrix@C.8pc{
f^{-1}(gY) \ar[r] \ar[d] & \Spec(k(g)) \times X \ar[r] \ar[d]^-{\iota_g} & 
\Spec(k(g)) \ar[d] \\
\Gamma_U \ar[r]^-{\alpha} & U \times X \ar[r]^-{p_1} & U,}
\end{equation}
where the composition of the horizontal arrows on the bottom is
$\phi$. Since $\phi$ is flat, it follows that the big outer square
is Tor-independent. Since $p_1$ is flat and the vertical arrows
are closed immersions, it follows by an elementary verification that
the left square is also Tor-independent. 

Let us now consider a resolution $\mathcal{E}^\bullet \to \cO_Y\to 0 $ of 
$\cO_Y$ by locally free $\cO_H$ modules of finite ranks.
Since $\beta$ is flat, it follows that $\beta^*(\mathcal{E}^\bullet)$
is a locally free resolution of $\sO_{\Gamma}$ under the
closed immersion $\alpha \colon \Gamma \inj G \times X$.
In particular, $\sE^{\bullet}_U := \beta^*(\mathcal{E}^\bullet)|_{U \times X}$ is 
a locally free resolution of $\sO_{\Gamma_U}$. 
Since the left square in ~\eqref{eqn:Kleiman-Levine-0} is Tor-independent,
it follows that $\iota^*_g \sE^{\bullet}_U \to \iota^*_g \sO_{\Gamma_U} \to 0$
is a locally free resolution on $X$.
Equivalently, $\iota^*_g \sE^{\bullet}_U \to \sO_{f^{-1}(gY)} \to 0$
is a locally free resolution on $X$.
That is,
$\Tor^{(\beta \circ \iota_g)^{-1} \sO_H}_i((\beta \circ \iota_g)^{-1} 
\sO_Y, \sO_X) = 0$ for all $i > 0$.
Since $\beta \circ \iota_g = \lambda_{g^{-1}} \circ f$ and
since $\lambda_g$ is an isomorphism, this implies that
$\Tor^{f^{-1}\sO_H}_i(f^{-1}\sO_{gY}, \sO_X) = 0$ for all $i > 0$.
Equivalently, ~\eqref{eqn:Homog-space-4} is Tor-independent.
\end{proof}

\begin{cor}\label{cor:Generic-lci}
Let $V \subset H$ be an open subscheme such that $Y \cap V$ is
a local complete intersection in $V$. Then $f^{-1}(gY \cap gV)$
is a local complete intersection in $f^{-1}(gV)$ for every $g \in U(k)$.
\end{cor}
\begin{proof}
Follows directly from \lemref{lem:Tor-ind} and \cite[16.4]{Matsumura}.
\end{proof}

\subsection{Kleiman-Levine transversality theorem}
\label{sec:KLTT}
We shall now prove the following transversality result over $k$.
We let $G$ and $H$ be as above. We shall follow
the notations of \S~\ref{sec:Set-up} and \S~\ref{sec:PS}.

\begin{thm}\label{thm:Kleiman-Levine}
Let $X$ be an equidimensional reduced quasi-projective $k$-scheme and let
$f \colon X \to H$ be a $k$-morphism. Let $Y \subset H$ be
an equidimensional reduced closed subscheme. Then there exists an open
dense subscheme $U(f,Y) \subset G$ such that for every
$g \in U(f,Y)(k)$, the following hold.
\begin{enumerate}
\item
The scheme $f^{-1}(gY)$ is either empty or is equidimensional of dimension
$\dim(X) + \dim(Y) - \dim(H)$.
\item
$\Tor^{f^{-1} \sO_H}_i(f^{-1} \sO_{gY}, \sO_X) = 0$ for all $i > 0$.
\item
The inclusion $f^{-1}(gY) \inj X$ is a local complete intersection
at every generic point of $f^{-1}(gY) \cap X_\sing$.
\end{enumerate}
\end{thm}
\begin{proof}
Let $U \subset G$ be the open subscheme obtained in 
\lemref{lem:The-Open-set}. The item (1) then 
follows directly from ~\eqref{eqn:Homog-space-1} and
(2) follows from \lemref{lem:Tor-ind}.

We now prove (3). First of all, we can apply ~\eqref{eqn:Homog-space-1}
to every irreducible component of $X_\sing$ to see that
after shrinking $U$ if necessary, every $g \in U(k)$ has
the property that either $f^{-1}(gY) \cap X_\sing$ is empty or
$\dim(f^{-1}(gY) \cap X_\sing) = \dim(X_\sing) + \dim(Y) - \dim(H)$.
Combining this with ~\eqref{eqn:Homog-space-1}, we get
the inequality
\begin{equation}\label{eqn:Kleiman-Levine-1}
\dim(f^{-1}(gY) \cap X_\sing) \le \dim(X_\sing) +
\dim(f^{-1}(gY)) - \dim(X),
\end{equation}
where the equality holds if $f^{-1}(gY) \cap X_\sing\neq \emptyset$. In other words, $f^{-1}(gY)$ and $X_\sing$ intersect properly in $X$.
The same token shows that by possibly shrinking $U$ further,
we have that $f^{-1}(gY_\sing)$ and $X_\sing$ intersect properly in $X$
for all $g \in U(k)$. This means, in particular,  that every generic point of 
$f^{-1}(gY) \cap X_{\rm sing}$ is contained in $f^{-1}(gY_\reg)$. 

We now observe that $Y_\reg \inj H$ is a local complete intersection morphism 
because $H$ is regular. The same is true for $g Y_\reg \inj H$,
since $\lambda_{g^{-1}} \in {\rm Aut}_k(H)$. We can therefore apply
\corref{cor:Generic-lci} to conclude (3). We take
$U(f,Y)$ to be the above $U$ to finish the proof.
\end{proof}

Let us write $\P^1_k$ as the homogeneous space ${{\rm PGL}_{2,k}}/B$, where
$B$ is the image of the upper-triangular matrices under the quotient map
${\rm GL}_{2,k} \surj {\rm PGL}_{2,k}$.
Let $G$ and $H$ be as above. Then $H \times \P^1_k$
becomes a homogeneous space for $\wt{G} := G \times {\rm PGL}_{2,k}$ via
the coordinate-wise action.
We shall now apply the previous Theorem in this setting 
to get the following result on generic translates in $H\times \P^1_k$.

Let $X$ be an equidimensional 
reduced quasi-projective scheme over $k$, and let
$f\colon X\to H$ be a $k$-morphism as before. 
Let $W\subset  H\times \P^1_k$ be an equidimensional reduced closed 
subscheme. 
Write $gW$ for the pullback of $W$ along $\lambda_{g^{-1}}\times \id \colon 
H\times \P^1_k \to H \times \P^1_k$.
Suppose that the composition $W\to \P^1_k$ of the inclusion of 
$W$ in $H\times \P^1_k$ followed by the second projection 
$H\times \P^1_k \to \P^1_k$ is flat over a neighborhood of $\{0, \infty\}$. 
Let $(f\times \id_{\P^1_k})^{-1}(gW)$ denote the fiber product 
$gW\times_{(H\times \P^1_k)} (X\times \P^1_k)$.
We write $f\times \id_{\P^1_k}$ as $\wt{f}$.
Let $G$ (resp. ${\rm PGL}_{2,k}$) act on $H \times \P^1_k$ by acting 
trivially on $\P^1_k$ (resp. on $H$).
These actions of $G$ and ${\rm PGL}_{2,k}$ on $H \times \P^1_k$
commute with each other.

\begin{prop}\label{prop:Kleiman-Levine-P1}
There exists an open dense subscheme $U=U(\wt{f}, W)$ of $G$ such 
that for every $k$-point $g\in U(k)$, the following hold.
\begin{enumerate}
\item
$\wt{f}^{-1}(gW)$ is either empty or equidimensional of dimension
\[
\dim(\wt{f}^{-1}(gW)) = \dim(X) + \dim(W) - \dim(H).
\]
\item
$\wt{f}^{-1}(gW)$ intersects $\P^1_{X_\sing}$ and 
$X_{\rm sing}\times \{0, \infty\}$ properly.
\item
The composition  $\wt{f}^{-1}(gW) \inj \P^1_X \to \P^1_k$ is flat over a 
neighborhood of $\{0, \infty\}$.
\item
The inclusion $\wt{f}(gW) \inj \P^1_X$ is a local complete intersection 
at each generic point of $\wt{f}(gW)\cap \P^1_{X_{\rm sing}}$.
\end{enumerate} 
\end{prop}
\begin{proof}
We shall use the structure of the homogeneous space on $\P^1_H$ for the group
$\wt{G}$ as above.
We then observe that for every $t \in  \PGL_{2,k}(k)$, the
map $\lambda_t\colon \P^1_X \to \P^1_X$ is an 
isomorphism which keeps every fiber of the projection
$\P^1_X \to X$ invariant. It follows that
for any $(g, t) \in \wt{G}(k)$, the assertions (1), (2) and (4) of the
proposition will hold for $\wt{f}^{-1}(gW)$ if and only if
they hold for $t (\wt{f}^{-1}(gW))$.

On the other hand, \thmref{thm:Kleiman-Levine} says that there is
a dense open $\wt{U} \subset \wt{G}$ such that for every
$(g,t) \in \wt{U}(k)$, the scheme $\wt{f}^{-1}((g,t)W)$
satisfies (1), (2) and (4). Since
\begin{equation}\label{eqn:prod-translation}
\wt{f}^{-1} ((g,t)W) = t (\wt{f}^{-1}(gW)),
\end{equation}
we conclude that for every
$(g,t) \in \wt{U}(k)$, the scheme $t(\wt{f}^{-1}(gW))$
satisfies (1), (2) and (4).
Letting $U$ be the image of $\wt{U}$ under the projection
$p_1 \colon \wt{G} \to G$, we see that $U \subset G$ is dense open.
Furthermore, for any $g \in G(k)$, the fiber $p^{-1}_1(g) \cong
\PGL_{2,k}$ has a dense set of $k$-rational points.
For any such point $t$, the scheme
$t(\wt{f}^{-1}(gW))$ satisfies (1), (2) and (4).
But we have seen in the beginning of the proof that this
is equivalent to saying that $\wt{f}^{-1}(gW)$ satisfies (1), (2) and (4).

For (3), it is enough to show that the map 
$\wt{f}^{-1}(gW) \to \P^1_k$ is flat in a neighborhood of
each point $\epsilon \in \{0, \infty\}$.
However, we know that $W$ satisfies this property. Hence, $gW$ too satisfies
this property for every $g \in G(k)$ as $G$ acts trivially
on $\P^1_k$. Now, if we replace $\P^1_k$ by $S = \Spec(\sO_{\P^1_k, \epsilon})$
and correspondingly replace all schemes over $\P^1_k$ by their
base change to $S$, then $\wt{f}^{-1}(gW)$ will not be flat over $S$ 
if and only if it is supported on the closed point $\{\epsilon\} \subset S$.
The latter condition is easily seen to imply that
\begin{equation}\label{eqn:prod-translation-0}
\Tor^{\sO_S}_1(k(\epsilon), \sO_{gW \times_{H_S} X_S}) \neq 0.
\end{equation}
Since $H_S$ and $gW$ are both flat over $S$, an elementary
homological algebra shows that ~\eqref{eqn:prod-translation-0}
implies that 
$\Tor^{\sO_{H_\epsilon}}_1(\sO_{gW_\epsilon}, \sO_{X_\epsilon}) \neq 0$.
But this contradicts \lemref{lem:Tor-ind} if we choose
$g \in U(k)$ (after possibly shrinking $U$). 
It follows that (3) holds if $g \in U(k)$.
Letting $U(\wt{f},W) = U$, we conclude the proof of the
proposition.
\end{proof}

\section{The family of rational equivalences}\label{sec:Rat-eq}
In \S~\ref{sec:Transversality}, we constructed families of good cycles
on a reduced quasi-projective scheme over an infinite field $k$ which are
obtained by pulling back cycles from homogeneous spaces.
In this section, we shall construct a family which will
parameterize rational equivalences between the members of a 
given family of good cycles. We shall prove some properties of this family
that will be used in the following sections.
We fix an infinite field $k$ throughout this section.

\subsection{$\A^1$-path connectivity of reductive groups}
\label{sec:Hom*}
We write $(\P^1_k)^m$ as $\ov{\square}^m_k$ for any integer $m \ge 1$.
For any integer $n \ge 1$, we fix a compactification
$\GL_{n,k} \subset \A^{n^2}_k \inj \ov{\square}^{n^2}_k$,
where the first inclusion is the inverse image of
$\G_{m,k}$ under the determinant map
$\det \colon \A^{n^2}_k \cong \M_{n,k} \to \A^1_k$ and the second
inclusion is the product of the canonical inclusion 
$\A^1_k \subset \ov{\square}^1_k$.
We remark here that there are many other choices of
a smooth compactification of $\GL_{n,k}$. But we choose 
the above one with some purpose. This will be evident in 
\lemref{lem:homotopy-2}.

We let $\GL^{\times}_{n,k} := \GL_{n,k} \cap \G_{m,k}^{n^2}
\subset \A^{n^2}_k$. It is clear that
$\GL^{\times}_{n,k}$ is a dense open subscheme of $\GL_{n,k}$ whose
complement is its intersection with the union of the
coordinate axes of $\A^{n^2}_k$.
Let $\mu \colon \A^n_k \times \A^1_k \to \A^n_k$ denote the multiplication 
operation induced by the $k$-algebra homomorphism
$k[t_1, \ldots , t_n] \to k[t, t_1, \ldots , t_n]$ which sends $t_i$ to 
$tt_i$. Note that $\mu$ is flat everywhere of relative dimension one
and is smooth over $\A^n_k \setminus \{0\}$.
We consider the morphisms
\begin{equation}\label{eqn:homotopy}
\xymatrix@C.8pc{
\A^n_k \times \A^n_k \times \A^1_k 
\ar[r]^-{\tau'} \ar[d]_-{\Phi'_n} &
\A^n_k \times \A^n_k \times \A^1_k \times \A^1_k \ar[r]^-{\tau}
& (\A^n_k \times \A^1_k) \times (\A^n_k \times \A^1_k) \ar[d]^-{\mu \times \mu}
\\
\A^n_k & & \A^n_k \times \A^n_k \ar[ll]_-{\mu^+},}
\end{equation}
where $\tau'(x, x', y) = (x, x', y, 1-y)$,
$\tau(x, x', y, y') = (x, y, x', y')$
and $\mu^+$ is the group operation map for the structure of the 
additive group $\G^n_{a,k}$ on $\A^n_k$. 

It is clear that $\tau$ is an isomorphism, $\mu^+$ is a
smooth morphism and $\mu$ is smooth over $\A^n_k \setminus \{0\}$.
One can easily check that $\Phi'_n$ (defined as the composition)  is a surjective
morphism between regular schemes.
Furthermore, its restriction on 
$W_n:= \G^n_{m,k} \times \G^n_{m,k} \times (\A^1_k \setminus \{0,1\})$
is surjective with each fiber 
isomorphic to the base change of
$\G^n_{m,k} \setminus (\A^1_k \setminus \{0,1\})$ by some field
extension of $k$.
It follows that $\Phi'_n$ is a smooth surjective morphism on $W_n$
of relative dimension $m+1$.

Note that each of the arrows in ~\eqref{eqn:homotopy} is defined over $k$.
In particular, the composition of all arrows $\Phi'_n$ is a $k$-morphism.
On the rational points, it is given by
$\Phi'_n(x, x', t) = tx + (1-t)x'$.
It is also easy to see that $\Phi'_n$ defines a rational map
$\Phi'_n \colon \A^n_k \times \A^n_k \times \P^1_k \dasharrow \ov{\square}^n_k$
whose base locus is $Z_n \times \{\infty\}$, where
$Z_n$ is the union of coordinate axes on $\A^n_k \times \A^n_k$.
In particular, if we let $m = n^2$, we get a $k$-morphism
\begin{equation}\label{eqn:homotopy-0}
\Phi'_m \colon \GL^{\times}_{n,k} \times \GL^{\times}_{n,k} \times \P^1_k \to
\ov{\square}^m_k.
\end{equation}

More generally, if $n_1, \ldots , n_r$ are positive integers and we let $m =
\stackrel{r}{\underset{i = 1}\sum} n^2_i$, then $\Phi'_m$
defines a rational map
$\Phi'_m \colon (\stackrel{r}{\underset{i =1}\prod} \A^{n^2_i}_k)
\times (\stackrel{r}{\underset{i =1}\prod} \A^{n^2_i}_k) \times \P^1_k
 \dasharrow \ov{\square}^m_k$ which is a morphism on the open subscheme
$(\stackrel{r}{\underset{i =1}\prod} \G_{m  ,k}^{n^2_i})
\times (\stackrel{r}{\underset{i =1}\prod} \G_{m, k}^{n^2_i}) \times \P^1_k$.
We let $G = \stackrel{r}{\underset{i =1}\prod} \GL_{n_i, k}$
and $G^{\times} = \stackrel{r}{\underset{i =1}\prod}\GL^{\times}_{n_i,k}$.
We then get a $k$-morphism
\begin{equation}\label{eqn:homotopy-1}
\Phi'_m \colon G^{\times} \times G^{\times} \times \P^1_k \to
\ov{\square}^m_k.
\end{equation}

We let $B = G^{\times} \times G^{\times}$
and let $\Phi_m$ be the composite
$B \times \P^1_k 
\xrightarrow{\id_{B} \times \eta}
B \times \P^1_k \xrightarrow{\Phi'_m} \ov{\square}^m_k$,
where $\eta(t) = {t}/{(t-1)}$.
We fix the open embedding $j \colon G \inj \ov{\square}^m_k$
via the composition of open embeddings
$G = \stackrel{r}{\underset{i =1}\prod} \GL_{n_i, k}
\inj  \stackrel{r}{\underset{i =1}\prod} \M_{n_i, k} 
\cong \A^m_k \inj \ov{\square}^m_k$.
We also have open embedding $G^{\times} \inj G$.
Let $\un{\infty} \in \ov{\square}^m_k$ be the closed point whose
every coordinate is $\infty$. The following lemma says, along with other 
things, that $G^{\times}$ is `path-connected' in the sense of
$\A^1$-homotopy theory.

\begin{lem}\label{lem:homotopy-2}
The $k$-morphism 
\[
\Phi_m \colon  B \times \P^1_k \to \ov{\square}^m_k
\]
has following properties.
\begin{enumerate}
\item
$\Phi_m(x, x', t) = \frac{t}{t-1} x + (1 - \frac{t}{t-1})x'$.
\item
$\Phi_m(B \times \{0, \infty\}) \subset G^{\times}$.
\item
$\Phi_m(B \times \{1\}) = \un{\infty}$.
\item
For every pair of points $g_1, g_2 \in G^{\times}(k)$,
we have $\Phi_m(g, 0) =  g_1$ and $\Phi_m(g, \infty) = g_2$,
where $g = (g_2, g_1) \in B(k)$.
\item
$G^{\times} \subset \Phi_m(B \times \P^1_k)$. In particular,
$\Phi_m$ is dominant.
\item
$\Phi_m$ is flat of relative dimension $m+1$ over $G^{\times}$.
\item
$\Phi_m$ is smooth on $B \times (\P^1_k \setminus \{0,1, \infty\})$.
\item
The projection map $\Phi^{-1}_m(G^{\times}) \to \P^1_k$ is flat.
\end{enumerate}
\end{lem}
\begin{proof}
All properties (except possibly (6) and (8)) 
are clear from our explicit construction of $\Phi_m$. The property (6)
follows because (5) shows that $\Phi_m$ is a surjective
morphism between regular schemes over $G^{\times}$ and (1) shows that
all fibers of $\Phi_m$ over $G^{\times}$ have relative dimension $m+1$.
Hence, it must be flat (see \cite[Exc.~III.10.9]{Hartshorne}).
The property (8) follows easily from (2) since the image of
$\Phi^{-1}_m(G^{\times})$ in $\P^1_k$ is anyway open.
\end{proof}

\begin{remk}\label{remk:homotopy-3}
Recall that every $k$-rational variety $X$ (e.g., the variety $G$ above)
is separably uni-ruled.
That is, there exists a separable dominant rational map
$\Phi \colon X' \times \P^1_k \dasharrow X$, where $\dim(X') = \dim(X)-1$.
However, the purpose of \lemref{lem:homotopy-2} is to show
that the reductive group $G$ satisfies many other properties
which can not be directly deduced from uni-ruledness.
We shall need \lemref{lem:homotopy-2} to construct our
parameter space of rational equivalences between good cycles 
in a family. Remark also that not every uni-rational variety in positive
characteristic in separably uni-ruled.
\end{remk}

\subsection{The partial parameter space $\Sigma^o$}
\label{sec:Sigma}
Let $G$ be as in \S~\ref{sec:Hom*}. Let $H = G/P$ be a homogeneous
space for $G$ as in \S~\ref{sec:Set-up}.
Define a morphism $\psi\colon \Phi_m^{-1}(G)\times G\times H\to H$ as the 
composition
\[  
\Phi_m^{-1}(G)\times G \times H \xrightarrow{ (\Phi_m, \id_{G\times H})} G\times G 
\times H \xrightarrow{(\iota, \iota)} G\times G\times H 
\xrightarrow{( m, \id_H)} G\times H \xrightarrow{\mu_H} H,\]
where $\iota \colon G \to G$ is the inverse operation,
$m \colon G \times G \to G$ is the multiplication in $G$ and
$\mu_H$ is the $G$-action on $H$
(we denoted this action by $\mu$ in \S~\ref{sec:Hom*}).
Let $Y \subset H$ be an integral closed subscheme
and let $\Sigma^o$ be the pullback
\begin{equation}\label{eqn:Sigma-0}
\xymatrix@C.8pc{
\Sigma^o \ar@{^{(}->}[r] \ar[d]_{\psi'} & 
\Phi_m^{-1}(G) \times G\times H \ar[d]^-{\psi} \\
Y\ar@{^{(}->}[r] & H.} 
\end{equation}

We note that the set of $k$-points of $\Sigma^o$ is given by
\[ 
\Sigma^o(k) = \{ (g_1, g_2, t, g, h) \,|\, \Phi_m(g_1, g_2, t)\in G 
\text{ and } (\Phi_m(g_1, g_2, t))^{-1} g^{-1} h \in Y   \}. 
\]
There is an action of $G$ on $\Sigma^o$ induced by the diagonal action of 
$G$ on $G\times H$. Explicitly, 
\begin{equation}\label{eqn:Sigma-0-1}
g'\cdot(g_1, g_2, t, g, h) = (g_1, g_2, t, g'g, g'h ). 
\end{equation}
Or, in other words, $\Sigma^o$ is stable for the $G$ action on 
$\Phi_m^{-1}(G)\times G\times H$ given by the trivial action on the first 
component and by the canonical action on the second and the third 
components.

We let $\Gamma := Y \times_H (G \times H)$
with respect to the composite map $\mu_H \circ (\iota \times \id_H)
\colon G \times H \to H$. It is easy to check that the
diagram
\begin{equation}\label{eqn:Sigma-0-0}
\xymatrix@C.8pc{
\Gamma \ar[r]^-{\iota_Y'} \ar[d]_-{\wt{\mu}'_H} & G \times H  
\ar[d]^-{{\mu}'_H} &  \\
G \times Y \ar[r]^-{\id \times \iota_Y} & G \times H}
\end{equation}
is Cartesian, where $\iota_Y: Y \inj H$ is the inclusion, $\iota_Y'$ is the second projection of $\Gamma$ 
and ${\mu}'_H$ is the automorphism of $G \times H$ given by 
$(g,h) \mapsto (g, g^{-1}h)$ (see \S~\ref{sec:Set-up}).
We thus have the following simple expression for the scheme $\Gamma$.

\begin{lem}\label{lem:Gamma}
The action map ${\mu}'_H$ induces an isomorphism of schemes
\[
\wt{\mu}'_H \colon \Gamma \xrightarrow{\cong} G \times Y.
\]
In particular, $\Gamma$ is integral.
\end{lem}
\begin{proof}
We only need to prove the second part. But this follows from the
known fact that the product two schemes, one of which is integral and
the other is geometrically integral, is integral. 
\end{proof}

Let $\phi$ be the composite map
$\Phi_m^{-1}(G) \times G \xrightarrow{(\Phi_m, \id_G)} G \times G
\xrightarrow{(\iota, \iota)} G \times G \xrightarrow{m} G$
so that $\psi = \mu_H \circ (\phi, \id_H)$.
It is then easy to see that ~\eqref{eqn:Sigma-0} being Cartesian is 
equivalent to saying that the square
\begin{equation}\label{eqn:Sigma1-0}
\xymatrix@C.8pc{
\Sigma^o \ar[r]^-{\phi'} \ar[d]_-{\theta'} & \Gamma \ar[d]^-{\theta} \\
\Phi_m^{-1}(G) \times G \ar[r]^-{\phi} & G}
\end{equation}
is Cartesian, where the vertical arrows are the projections.

Since the projection $G \times H \to G$ is projective (because  $H$ is projective), and
$\Gamma$ is closed in $G \times H$ (because $Y$ is closed in $H$), it follows that
$\theta$ is a projective morphism. Hence, $\theta'$ is also a projective morphism.

Let $G$ act on itself by left multiplication and diagonally on $G \times H$.
Then  $\Gamma \subset G \times H$ is $G$-invariant. In particular,
$\Gamma$ is equipped with $G$-action such that $\theta$ is
$G$-equivariant. Moreover, $\theta$ is surjective and
equidimensional of relative dimension equal to $\dim(Y)$. 
As $G$ is integral, the generic flatness theorem says that
$\theta$ is flat over a dense open subscheme of $G$. Since
$\theta$ is $G$-equivariant, $G(k)$ is Zariski dense in
$G$ whose action is transitive on the base of $\theta$, it follows
that the latter is flat everywhere on $G$. It follows in
particular that $\theta'$ is flat and surjective.

Since $\Gamma$ and $G$ are integral, it follows that the
generic fiber of $\theta$  is integral. Moreover, the fiber of
$\theta$ over every $k$-rational point is isomorphic to $Y$.
In particular, every $k$-rational fiber of $\theta$ is integral.
Hence, the same holds for $\theta'$ as well. Since 
$\Phi_m^{-1}(G) \times G$ is irreducible,
$\theta'$ is flat (and hence open), $(\Phi_m^{-1}(G) \times G)(k)$
is Zariski dense in $\Phi_m^{-1}(G) \times G$ (as the latter 
is a rational variety) and every $k$-rational fiber of 
$\theta'$  is irreducible, it follows from
\lemref{lem:Irr-fibers} that $\Sigma^o$ is irreducible.

It follows from \lemref{lem:homotopy-2} (7) that
the map $\phi \colon \Spec(k(\Phi_m^{-1}(G) \times G)) \to
\Spec(k(G))$ is smooth. Since $\theta$ is generically integral, it follows
that $\theta'$ is generically reduced. Since it is also surjective,
it follows that $\Sigma^o$ is generically reduced. Hence, it is
generically integral.
We have thus shown that $\Sigma^o$ is irreducible and generically integral.
Although this will suffice for our main proofs, the following lemma says more.

\begin{lem}\label{lem:Sigma1}
$\Sigma^o$ is an integral scheme. 
\end{lem}	
\begin{proof}
  We have seen above that $\Sigma^o$ is irreducible. So we only need to show that it is
  reduced. One knows that a Noetherian scheme is reduced if and only if all its local rings
  satisfy Serre's $R_0$ and $S_1$ conditions. Since we have shown above that $\Sigma^o$ is
  generically integral, all its local rings clearly satisfy the $R_0$ condition
  (i.e., the localization at every minimal prime is regular). It remains to show that
  the local rings of $\Sigma^o$ satisfy the $S_1$ condition.
  Equivalently, we have to show that if $x \in \Sigma^o$ is a point of codimension at least one,
  then the local ring $\sO_{\Sigma^o,x}$ contains a non-zero divisor.

 We let $Z = \ov{\{x\}} \subset \Sigma^o$ with its reduced induced closed subscheme structure.
  Since $\theta'$ is projective and $Z$ is integral, its scheme-theoretic image $\theta'(Z)$
  is an integral closed subscheme of $\Phi^{-1}_m(G) \times G$. We let $U = \Spec(A) \subset
  \Phi^{-1}_m(G) \times G$ be an affine neighborhood of $\theta'(x)$ and let $V = \Spec(A') \subset \Sigma^o$
  be an affine neighborhood of $x$ such that $\theta'(V) \subset U$. Let $\fp$ be the prime ideal
  of $A$ such that $\theta'(Z) \cap U = V(\fp)$ and let $\fp'$ be the prime ideal of
  $A'$ such that $Z \cap V = V(\fp')$. Since $\theta'$ is dominant  whose base is integral,
  it follows that it induces a $k$-algebra monomorphism $A \inj A'$ such that
  $\fp' \cap A = \fp$.

  We have to consider two cases. Suppose first that $\fp = \{0\}$. Then the inclusion $A \inj A'$
  induces an inclusion $k(A) \inj A'_{\fp'}$, where $k(A)$ is the quotient field of $A$.
  Since we have shown above that the generic fiber of $\theta'$ is integral, it follows that
  $A'_{\fp'}$ is integral. Since the codimension of $Z$ is at least one in $\Sigma^o$,
it follows that $A'_{\fp'}$ has dimension at least one. In particular, it contains
non-zero divisors.

  In the second remaining case, we can assume that $\fp \neq \{0\}$.
  Let $a \in \fp$ be any non-zero element. Then $a \in \fp'$. Since $A$ is an integral domain,
  $a$ is a non-zero divisor in $A$. 
  Since we have shown above that $\theta'$ is flat, it follows that the map
  $A \to A'_{\fp'}$ is also flat. This implies that $a \in A'_{\fp'} = \sO_{\Sigma^o,x}$ must be a non-zero divisor.
  This finishes the proof.
\end{proof}

Let $p_{BG}$ denote the projection map 
$B \times \P^1_k \times G \to B \times G$ and let
$p^o_{BG} = {p_{BG}}|_{\Phi^{-1}_m(G)\times G}$.
If $p_B \colon B \times \P^1_k \to B$ is 
the projection map, then note that $p_{BG} = {p_B} \times \id_G$. Let
$\pi^o \colon \Sigma^o \to B \times G$ denote the composition
$p^o_{BG}\circ \theta'$.

\begin{lem}\label{lem:Sigma-rational}
$\pi^o$ is a flat and surjective morphism whose generic and rational
fibers are integral.
\end{lem}
\begin{proof}
Since $\theta'$ is flat and surjective and $p_{BG}$ is smooth, it
follows that $\pi^o$ is flat. Since $p_B$ is surjective, 
by \lemref{lem:homotopy-2} (2), it follows that so are
$p^o_{BG}$ and $\pi^o$ (note that $p_{BG}$ is trivially surjective). 
Since $\Sigma^o$ (see \lemref{lem:Sigma1}) and $B \times G$ are integral, 
it follows that the generic fiber of $\pi^o$ is integral.

We now fix a point $w \in (B\times G)(k)$. We can write
this point uniquely as $w = (g_1, g_2, g) \in G^{\times}(k) \times 
G^{\times}(k) \times G(k)$. Then $p^{-1}_{BG}(w) = \Spec(k(w)) \times \P^1_k
\cong \P^1_k$ and $U^o_{w} := (p^o_{BG})^{-1}(w) \subset \P^1_k$ is  
open.
Since the projection $p_B \colon \Phi^{-1}_m(G) \to B$ is surjective
by \lemref{lem:homotopy-2} (2), we see that $U^o_{w}$ is dense open in 
$\P^1_k$. Let $\Sigma^o_{w} = (\pi^o)^{-1}(w)$.

Let $\phi_{w}$ denote the restriction of $\phi$ to the closed
subscheme $U^o_{w}$. It follows from ~\eqref{eqn:Sigma1-0} that
$\Sigma^o_{w} = U^o_{w} \times_G \Gamma$ via the map $\phi_{w}$.
It follows from this that every rational fiber of
the projection $p^o_{w} \colon \Sigma^o_{w} \to U^o_{w}$ is 
isomorphic to $Y$. In particular, it is integral.
Since $\theta'$ is flat and surjective, so is $p^o_{w}$. In
particular, it is open.
Since $U^o_{w}(k)$ is Zariski dense in $U^o_{w}$, and the latter
is irreducible, we conclude from \lemref{lem:Irr-fibers} that
$\Sigma^o_{w}$ is irreducible.  

We shall show that $\Sigma^o_{w}$ is reduced by following the same
argument that we used for proving this property for $\Sigma^o$.
We first need to show that the generic fiber of $p^o_{w}$ is integral.

Let $S = \Spec(k(U^o_{w}))$ denote the generic point of $U^o_{w}$.
We need to show that $S \times_G \Gamma$ is integral via the maps
$\phi_{w} \colon S \to G$ and $\theta \colon \Gamma \to G$. 
For this, we consider the 
diagram
\begin{equation}\label{eqn:Sigma-rational-0}
\xymatrix@C.8pc{
\Gamma \ar[r]^-{\iota_Y'} \ar[d]_-{\wt{\mu}'_H} & G \times H \ar[dr]^-{p_1} 
\ar[d]_-{\wt{\mu}'_H} &  \\
G \times Y \ar[r]^-{\id \times \iota_Y} & G \times H \ar[r]^-{p_1} &
G \\
S \times Y \ar[r]^-{\id \times \iota_Y} \ar[u] & S \times H \ar[u] 
\ar[r]^-{p_1} & S \ar[u]_-{\phi_{w}},}
\end{equation}
where $p_1$ is the
projection to the first factor of its source and other notations
are as in  ~\eqref{eqn:Sigma-0-0}.

It is clear that all squares and the lone triangle in 
~\eqref{eqn:Sigma-rational-0} are commutative. Moreover, all 
squares are Cartesian and $\theta = p_1 \circ \iota'_Y$.
In particular, we get 
$\theta =  p_1 \circ (\id \times \iota_Y) \circ \wt{\mu}'_H$.
Since $\wt{\mu}'_H$ is an isomorphism, the integrality of
$S \times_G \Gamma$ is equivalent to showing that
$S \times_G (G \times Y)$ is integral via the maps
$\phi_{w}$ and $p_1 \circ (\id \times \iota_Y)$.
Since all lower squares are Cartesian, the desired integrality is
finally equivalent to saying that $S \times Y$ is integral.
But this is clear because $Y$ is integral and 
$k(U^o_{w})$ is purely transcendental over $k$.

We have now shown that $\Sigma^o_{w}$ is irreducible, the map
$p^o_{w} \colon \Sigma^o_{w} \to U^o_{w}$ is flat and surjective
whose generic fiber is integral. Since $\Sigma^o_{w} =  U^o_{w} \times_G
\Gamma$ and $\theta \colon \Gamma \to G$ is projective, it follows that
$p^o_{w}$ is also projective. We can now repeat the
argument of the proof of \lemref{lem:Sigma1} to conclude that
$\Sigma^o_{w}$ is reduced. Hence, it is integral.
This finishes the proof of the lemma.
\end{proof}

\subsection{The parameter space of rational equivalences}
\label{sec:Pspace}
Let $\Sigma$ denote the scheme-theoretic closure of $\Sigma^o$ in $B \times \P^1_k \times G \times H$. Let $\theta'$ be the projection map 
$\Sigma \to B \times  \P^1_k \times G$ and let $\pi$ be the composition $\Sigma \xrightarrow{\theta'}
B \times \P^1_k \times G \xrightarrow{p_{BG}} B \times G$.
\begin{lem}\label{lem:extralemsigma}
In the above notations, $\Sigma$ is a $G$-invariant, geometrically integral closed subscheme of $B \times \P^1_k \times G \times H$, for the $G$-action induced by the  trivial action on
$B \times \P^1_k$ and the diagonal action on $G \times H$. Moreover, the  projection map $\pi$ is $G$-equivariant.
\begin{proof}
Since all factors of $B \times \P^1_k \times G \times H$
are geometrically integral, we know that it is (geometrically) integral. Since
$\Sigma^o$ is an integral closed subscheme of an open subscheme of $B \times \P^1_k \times G \times H$,
it follows that $\Sigma$ is an integral closed subscheme of $B \times \P^1_k \times G \times H$.  
Since $\theta'$ is projective on $\Sigma$ (since $H$ projective) and it is dominant on 
$\Sigma^o$, it follows that the map $\theta' \colon
\Sigma \to B \times  \P^1_k \times G$ is surjective. 
Furthermore, 
\begin{equation}\label{eqn:Sigma-closed-open}
\Sigma^o = \Sigma \cap (\Phi^{-1}_m(G) \times G \times H)
= \Sigma \cap \theta'^{-1}(\Phi^{-1}_m(G) \times G).
\end{equation}
Let $\nu \colon G \times (B \times \P^1_k \times G \times H) \to
B \times \P^1_k \times G \times H$ denote the action map of the lemma.
We have seen in ~\eqref{eqn:Sigma-0-1} that $\Sigma^o$ is $G$-invariant
with respect to this action.
Since $G$ is geometrically integral and $\Sigma$ is integral, it follows
that $G \times \Sigma$ is an integral closed subscheme of
$G \times (B \times \P^1_k \times G \times H)$.
Since $\Sigma^o$ is dense open in $\Sigma$, it follows that
$G \times \Sigma^o$ is dense open in $G \times \Sigma$.
Since $\Sigma^o$ is $G$-invariant, it follows that
$G \times \Sigma^o \subset \nu^{-1}(\Sigma^o) \subset \nu^{-1}(\Sigma)$.
As $\Sigma$ is closed in $B \times \P^1_k \times G \times H$, it
follows that $G \times \Sigma = \ov{G \times \Sigma^o} 
\subset \nu^{-1}(\Sigma)$.
But this means that $\Sigma$ is a $G$-invariant closed
subscheme of $B \times \P^1_k \times G \times H$.
Moreover, the projections $\Sigma \xrightarrow{\theta'}
B \times \P^1_k \times G \xrightarrow{p_{BG}} B \times G$ are
$G$-equivariant. 
\end{proof}
    \end{lem}
Let $\pi \colon \Sigma \to B \times G$ denote the composite 
projection map. For $w \in B \times G$, we let $\Sigma_w$
denote the scheme-theoretic fiber $\pi^{-1}(w)$.

%There is a $G$-action on
%$B \times \P^1_k \times G \times H$ via the trivial action on
%$B \times \P^1_k$ and the diagonal action on $G \times H$.

\begin{lem}\label{lem:Sigma2}
There exists an open dense subscheme $V_Y$ of $B$ such that the
projection map $\Sigma_Y := \pi^{-1}(V_Y \times G) \xrightarrow{\pi}
V_Y\times G$ has following properties.
\begin{enumerate}
\item
$\pi$ is flat and surjective.
\item
The generic fiber of $\pi$ is integral.
\item
For every $w \in (V_Y \times G)(k)$, 
there exists an open neighborhood (depending on $w$) $U^o_w$  
of $w \times \{0, \infty\}$
in $\P^1_{k(w)} = p^{-1}_{BG}(w)$ such that $\pi^{-1}(w) \cap \theta'^{-1}(U^o_w)$ 
is an integral scheme which is flat over $U^o_w$ under the projection
map $\theta' \colon \pi^{-1}(w) \to \P^1_{k(w)}$.
\item
If $Y$ is geometrically reduced (e.g., if $k$ is perfect), then we can 
choose $V_Y$ so that 
$\pi^{-1}(w)$ is geometrically reduced for every $w \in V_Y \times G$.
\item
If $Y$ is geometrically irreducible (e.g., if $k$ is separably closed), 
then we can choose $V_Y$ so that 
$\pi^{-1}(w)$ is geometrically irreducible for every $w \in V_Y \times G$.

\end{enumerate}
\end{lem}
\begin{proof}
Since $\theta'$ and $p_{BG}$ are surjective (see the proof of
\lemref{lem:extralemsigma}), we see that $\pi$ is
a surjective morphism whose base is integral. Hence, the generic
flatness theorem says that there is a dense open subscheme
$V' \subset B \times G$ over which $\pi$ is flat.
We let $V_Y$ be the image of $V'$ under the projection $B \times G \to B$.
Since this projection is open, $V_Y$ is open dense in $B$. 
Since $B \times G$ (resp. $G$) is rational, every open subset in $B \times G$
(resp. $G$) 
has the property that the set of its $k$-rational points is Zariski dense.
Since $G$ acts transitively on itself, and $\pi$ is $G$-equivariant by Lemma \ref{lem:extralemsigma}, 
we deduce that $\pi$ must be flat over $V_Y \times G$. 
Since $\Sigma$ and $V_Y \times G$ are integral, it follows that the 
generic fiber of $\pi$ is integral.

We now fix a point $w \in (V_Y\times G)(k)$. We can write
this point uniquely as $w = (g_1, g_2, g) \in G^{\times}(k) \times 
G^{\times}(k) \times G(k)$. Then $p^{-1}_{BG}(w) = \Spec(k(w)) \times \P^1_k
\cong \P^1_k$ and $U^o_{w} := (p^o_{BG})^{-1}(w) \subset \P^1_k$ is  
open. Here, $p^o_{BG}$ is the projection $(\Phi^{-1}_m(G) \cap (V_Y \times \P^1_k))
\times G \to V_Y \times G$.
Since the projection $p_B \colon \Phi^{-1}_m(G) \cap (V_Y \times \P^1_k) \to 
V_Y$ is surjective, by \lemref{lem:homotopy-2} (2), we see that 
$U^o_{w}$ is dense open in $\P^1_k$. 
It also follows from \lemref{lem:homotopy-2} (2) that
$\{0, \infty\} \subset U_w$.
Let $\Sigma^o_{Y,w} = (\pi^o)^{-1}(w)$ and
$\Sigma_{Y,w} = \pi^{-1}(w)$.

Under the above situation, it follows from ~\eqref{eqn:Sigma-closed-open}
that we have a Cartesian square
\begin{equation}\label{eqn:Sigma2-0}
\xymatrix@C.8pc{
\Sigma^o_{Y,w} \ar[r]^-{j'} \ar[d]_-{p^o_{w}} & \Sigma_{Y,w} 
\ar[d]^-{p_{w}} \\
U^o_{w} \ar[r]^-{j} & \P^1_{k}.}
\end{equation}

We saw in the proof of \lemref{lem:Sigma-rational} that
$p^o_{w}$ is surjective. Hence, it is flat, since $U^o_w$ is a regular scheme of dimension $1$.
It also follows from the same lemma
that $\Sigma^o_{Y,w}$ is integral.
Since $p_w$ is proper, it follows
that the map $\ov{\Sigma^o_{Y,w}} \to \P^1_k$ is flat and surjective,
where $\ov{\Sigma^o_{Y,w}}$ is the scheme-theoretic closure 
of $\Sigma^o_{Y,w}$ in $\Sigma_{Y,w}$. 
We have thus shown (1) $\sim$ (3).

Suppose now that $Y$ is geometrically reduced.
Recall that a scheme $W$ over a field $k'$ is 
geometrically reduced if and only if $W_{k''}$ is reduced
for every field extension $k' \subset k''$.
%Stack-exchangeL: varieties, Lemma~6.4.
Thus we see using ~\eqref{eqn:Sigma-0-0} and \lemref{lem:Gamma} that
the generic fiber of the map $\theta \colon \Gamma \to G$ is
geometrically reduced. Since $\phi$ is dominant, the same token shows that
the generic fiber of $\theta'$ is geometrically reduced.

Let $\eta$ (resp. $\eta'$) denote the generic point of 
$B \times G$ (resp. $\Phi^{-1}_m(G) \times G$).
Then (2) shows that $\Sigma_{Y, \eta} = \Sigma^o_{Y, \eta} = \pi^{-1}(\eta) =
\theta'^{-1}(\P^1_{\eta})$ is integral. Clearly, $\eta'$ is the generic
point of $\P^1_{\eta}$. Since $\theta'^{-1}(\eta')$ is 
geometrically reduced, as shown above, it follows
(e.g., see \cite[Proposition~5.49]{Gortz}) that 
the function field of $\Sigma_{Y, \eta}$ is separable over $k(\eta')$.
Since $k(\eta')$ is purely transcendental over $k(\eta)$, it follows
that the function field of $\Sigma_{Y, \eta}$ is separable over $k(\eta)$.
We thus see that $\Sigma_{Y, \eta} \in \Sch_{k(\eta)}$
is integral whose function field is separable over $k(\eta)$.
We conclude from EGA IV$_2$ 4.6.3 (see also \cite[Proposition~5.49]{Gortz})
that $\Sigma_{Y, \eta}$ is geometrically reduced.

Since $\pi$ is a flat and projective over $V_Y \times G$ whose
generic fiber is geometrically reduced, it follows from
EGA IV$_3$ 12.2.1 that there is a dense open subscheme
$V' \subset V_Y \times G$ such that $\pi^{-1}(w)$ is
geometrically reduced for every $w \in V'$.
If we let $V'_Y \subset V_Y$ be the image of $V'$ in $V_Y$, then 
using again that $G$ acts transitively on itself, and $\pi$ is $G$-equivariant, 
we deduce that $\pi^{-1}(w)$ is geometrically reduced for every
$w \in V'_Y \times G$. We have thus shown that by shrinking
$V_Y$ if necessary, we can achieve (4).

Suppose now that $Y$ is geometrically irreducible.
Recall that a scheme $W$ over a field $k'$ is 
geometrically irreducible if and only if $W_{k''}$ is geometrically
irreducible for every field extension $k' \subset k''$.
%(e.g., see \cite[Lemma~8.2] in Stacks project (varieties)
It follows therefore from 
~\eqref{eqn:Sigma-0-0} and \lemref{lem:Gamma}
that the generic fiber of the map $\theta \colon \Gamma \to G$ is
geometrically irreducible. Since $\phi$ is dominant, the same holds for
the generic fiber of $\theta'$.
In particular, $k(\eta')$ is separably closed in 
the function field of $\Sigma_{Y, \eta} = \Sigma^o_{Y, \eta}$ by
\cite[Proposition~5.50]{Gortz}. 

Since $k(\eta')$ is
purely transcendental over $k(\eta)$, it follows that
the latter is separably closed in 
the function field of $\Sigma_{Y, \eta}$.
We thus see that $\Sigma_{Y, \eta} \in \Sch_{k(\eta)}$
is integral such that $k(\eta)$ is separably closed in
the function field of $\Sigma_{Y, \eta}$.
We conclude from \cite[Proposition~5.50]{Gortz} that
$\Sigma_{Y, \eta}$ is geometrically irreducible.
Since $V_Y \times G$ is integral, we can apply
\cite[{Tag 0559}]{stacks-project} to find a dense open
subscheme $V' \subset V_Y \times G$ such that
$\pi^{-1}(w)$ is geometrically irreducible for every $w \in V'$.
We can now argue as before to show that
after shrinking $V_Y$, we can achieve (5).
This finishes the proof.
\end{proof}

\begin{lem}\label{lem:Sigma3}
Let $V_Y \subset B$ be the open subscheme of
\lemref{lem:Sigma2}. Assume that $Y \subset H$ has codimension $i$.
Let $w = (v, g) \in (V_Y \times G)(k)$ and let
$p_w \colon \Sigma_{Y,w} \to \P^1_k$ be the composite projection
$\Sigma \to B \times \P^1_k \times G \to \P^1_k$. 
Then the following hold.
\begin{enumerate}
\item
$\Sigma_{Y,w}$ is a closed subscheme of $H \times \P^1_k$ of pure
codimension $i$.
\item
$p_w$ is flat over an open neighborhood of $\{0, \infty\}$.
\item
If we write $v = (g_1, g_2) \in B(k) = G^{\times}(k) \times
G^{\times}(k)$, then
\[
p^{-1}_w(0) = g_1 g \cdot Y \ \ \mbox{and} \ \ 
p^{-1}_w(\infty) = g_2 g \cdot Y.
\]
\end{enumerate}
\end{lem}
\begin{proof}
The property (1) follows because $\pi \colon \Sigma_Y \to V_Y \times G$
is a flat surjective morphism between two integral schemes and hence has
equidimensional fibers. The property (2) follows
from \lemref{lem:Sigma2} (3) because the composition
$\P^1_{k(w)} \inj B \times \P^1_k \times G \to \P^1_k$ is an 
isomorphism. The last property follows directly 
\lemref{lem:homotopy-2} (2) and the definition of
$\phi$ in ~\eqref{eqn:Sigma1-0}.
\end{proof}

\section{Pull-backs of cycles from homogeneous spaces}
\label{sec:Pull-back}
We shall now use the results of the previous sections to construct
the pull-back maps from the Chow groups of codimension two cycles
on suitable homogeneous spaces to the Levine-Weibel Chow group of
a singular surface. Note that for smooth surfaces, this is an
easy consequence of the Chow moving lemma or Fulton's
deformation to normal cone techniques. We can not use 
the latter trick in the non-$\A^1$-invariant world.
So we shall use some sort of moving lemma tricks to achieve our goal.
Over algebraically closed fields, this construction is
due to Levine \cite{Levine-2}. We shall follow Levine's
outline to carry this out for surfaces over any infinite field.

We fix an infinite field $k$.
We consider the linear algebraic group
$G = \stackrel{r}{\underset{i = 1}\prod} \GL_{n_i, k}$ 
and let $H = G/P$ be a projective 
homogeneous space for $G$ as in \S~\ref{sec:Set-up}.
We let $B = G^{\times} \times G^{\times} \subset G \times G$
be the open subscheme as in \S~\ref{sec:Hom*}.
For an integral closed subscheme $Y \subset H$, we let $V_Y \subset B$ 
be the open and $\Sigma_Y \subset B \times \P^1_k \times G \times H$ the closed
subschemes of \lemref{lem:Sigma2}.
We shall continue to follow the notations of the previous sections.

We recall the maps used in \lemref{lem:Sigma2} 
in the following diagram as we shall need to use them in this section.

\begin{equation}\label{eqn:Main-maps}
\xymatrix@C.8pc{
\Sigma_{Y,w} \ar@{^{(}->}[d] \ar[r]^-{\theta'} & \P^1_{k(w)} \ar[r]^-{p_{BG}}
\ar@{^{(}->}[d] & \Spec(k(w)) \ar@{^{(}->}[d] \\
\Sigma_Y \ar[r]^-{\theta'} \ar@/_5pc/[rr]_-{\pi}  & 
V_Y \times \P^1_k \times G \ar[d]^-{p_2} \ar[r]^-{p_{BG}} & V_Y \times G \\
& \P^1_k &}
\end{equation}

Here, $p_2$ is the projection map. The two squares on the top are 
Cartesian and the composite vertical
arrow in the middle is an isomorphism if $w \in (V_Y \times G)(k)$.
We shall let $p_w$ denote the composite map
$p_w \colon \Sigma_{Y,w} \inj \Sigma_Y \xrightarrow{p_2 \circ \theta'}
\P^1_k$.

We fix an equidimensional reduced
quasi-projective surface $X$ over $k$. 
Note that since $X$ is reduced, it is regular in codimension zero.
That is, $\dim(X_\sing) \le 1$.
Since the associated and minimal primes of a reduced Noetherian 
commutative ring coincide (e.g., see \cite[Lemma~3.3]{GK}), 
we note furthermore that $X$ is Cohen-Macaulay
in codimension one. Since the Cohen-Macaulay locus $X_{\rm CM}$ of $X$ is
open (e.g., see EGA IV$_2$ 6.11.2), it follows that its complement
$X_{\rm nCM}$ is a finite closed subscheme (with reduced induced structure) 
of $X$ contained in $X_\sing$.
We fix a morphism $f \colon X \to H$.

\subsection{Construction of the map $f^*\colon \sZ^2(H) \to \CH^L_0(X)$}
\label{sec:PB-1}
Let $\sZ^2(H)$ be the free abelian group of integral
cycles on $H$  of codimension two. We shall first define a
map
\[
f^* \colon \sZ^2(H) \to \CH^L_0(X).
\]

Let $Y \subset H$ be an integral closed subscheme of codimension two.
By applying \thmref{thm:Kleiman-Levine} to $X$ and $X_\sing$
simultaneously, we find a dense open
subscheme $U(f,Y) \subset G$ such that for every $g \in U(f,Y)(k)$,
the pull-back $f^{-1}(gY) := gY \times_H X$ has pure codimension
two and  $f^{-1}(gY) \cap X_\sing = \emptyset$.
We can then define
\begin{equation}\label{eq:def-pullback-1} 
f^*[Y] = [ f^{-1}(gY)] \in \sZ^L_0(X),
\end{equation}
where $[f^{-1}(gY)]$ is the fundamental class of the closed subscheme 
$f^{-1}(gY)$ in the sense of \cite[\S~1.5]{Fulton}.
Our goal now is to show that (after possibly shrinking $U(f,Y)$)
this class does not depend on
the choice of $g \in U(f,Y)(k)$ up to rational equivalence in 
$\sZ^L_0(X, X_\sing)$. We shall prove this using the following lemmas.

Until we have proven the desired independence, we fix $Y \subset H$ as
above and simplify our notations. We write $U(f,Y)$ simply as $U$.
Recall also that if $\Sigma \subset B \times \P^1_k \times G \times H$
is as in \S~\ref{sec:Pspace}, then $\Sigma_Y = \pi^{-1}(V_Y \times G)$
for the projection map $\pi \colon \Sigma \to B \times G$.
Hence, we have $\Sigma_{w} = \Sigma_{Y,w}$ for
all $w \in V_Y \times G$. We shall therefore use
the common notation $\Sigma_w$.
We recall our notation $\wt{f}$ for the map
$f \times \id \colon X \times \P^1_k \to H \times \P^1_k$
from \S~\ref{sec:KLTT}.
For every $w \in (V_Y \times G)(k)$,
we know that $\Sigma_{w}$ is a closed subscheme of $\P^1_H$.
We shall write $\wt{f}^{-1}(\Sigma_{w}) = \Sigma_{w} \times_{\P^1_H} \P^1_X$ as 
$\Gamma_w$.

We consider the composite map
\begin{equation}\label{eqn:def-pi-Sigma}
\tau \colon V_Y \times G \inj B \times G \inj G \times G \times G 
\to G \times G;
\end{equation}
\[
(g_1, g_2, g) \mapsto (g \cdot \Phi_m(g_1, g_2, 0), 
g \cdot \Phi_m(g_1, g_2, \infty)) = (gg_1, gg_2).
\]
Note that this definition makes sense by \lemref{lem:homotopy-2} (2).
The equality on the right follows from \lemref{lem:homotopy-2} (4).
It is easy to check that the last map in
~\eqref{eqn:def-pi-Sigma} is a surjective morphism between regular
schemes with smooth fibers, namely, $G$. Hence, this map
is flat by \cite[Exc.III.10.9]{Hartshorne}, and therefore smooth.
Since the first and second maps in ~\eqref{eqn:def-pi-Sigma} are open
immersions, it follows that $\tau$ is a smooth morphism.
If $\tau(w) = (g_1, g_2) \in( G \times G)(k)$, then 
\lemref{lem:Sigma3} says that 
\begin{equation}\label{eqn:def-pi-Sigma-0}
\Sigma_{w} \times_{\P^1_k} \{0\} = g_1 Y \ \ \mbox{and} \ \
\Sigma_{w} \times_{\P^1_k} \{\infty\} = g_2 Y.
\end{equation}

We let $\wt{U} := \tau^{-1}(U \times U)$ and consider the subset
$W \subset \wt{U} \subset V_Y \times G$
consisting of points $w$ having the following properties.

\begin{listabc}
\item 
$\Gamma_w$ is of pure dimension $1$.
\item 
$\Gamma_w \cap \P^1_{X_{\rm nCM}} = \emptyset$.
\item  
$(\Sigma_w)_{\rm nsm}\cap f(X_{\rm sing}) \times \P^1_k = \emptyset$.
\item 
$\Gamma_w \cap \P^1_{X_{\rm sing}}$ is a finite set.
\end{listabc}

Our goal is to show that $W$ is a constructible subset of
$V_Y \times G$ containing the generic point of the latter.
In particular, it is non-empty. Note that $\wt{U}$ is dense open
in $V_Y \times G$ because $\tau$ is smooth and $G \times G$ is
irreducible.

\begin{lem}\label{lem:Constructible}
$W$ is a constructible subset of $V_Y \times G$.
\end{lem}
\begin{proof}
Note that $W$ is constructible in $V_Y \times G$ if and only if
it is so in $\wt{U}$.
We consider the maps $f \times \id \colon V_Y \times \P^1_k \times G \times X 
\to
V_Y \times \P^1_k \times G \times H$. We let $\Sigma' =
(f \times \id)^{-1}(\Sigma_Y)$ and $f' = (f \times \id)|_{\Sigma'}$.
This gives us maps $\Sigma' \xrightarrow{f'} \Sigma \xrightarrow{\pi}
V_Y \times G$. 
Note that for every $w \in V_Y \times G$, the induced map on the
fibers $f'_w \colon \Sigma'_w \to \Sigma_w$ is same as the 
restriction of the 
map $\wt{f} \colon (\P^1_X)_{k(w)} \to  
(\P^1_H)_{k(w)}$ to $\wt{f}^{-1}(\Sigma_w)$. 
We can therefore replace $\Gamma_w = \wt{f}^{-1}(\Sigma_w)$ in (a) $\sim$ (d) 
above by $\Sigma'_w$. We let $\pi'$ be the restriction of the
map  $\pi \circ f' \colon \Sigma' \to V_Y \times G$ to 
$\Sigma'_U := (\pi \circ f')^{-1}(\wt{U})$.

We let $W_a \subset \wt{U}$ be the set of points 
which satisfy condition (a) above. We define $W_b, W_c$ and $W_d$
similarly. It suffices to show separately that each of the sets
$W_a, W_b, W_c$ and $W_d$ is constructible.

Now, the constructibility of $W_a$ follows by applying
\cite[E.1 (5)]{Gortz} to the morphism $\pi' \colon \Sigma'_U \to \wt{U}$.
The constructibility of $W_b$ and $W_d$ follows by applying 
\cite[E.1 (4)]{Gortz}
to the compositions $\Sigma'_U \times_X X_{\ncm} \inj \Sigma'_U
\xrightarrow{\pi'} \wt{U}$ and
$\Sigma'_U \times_X X_{\sing} \inj \Sigma'_U 
\xrightarrow{\pi'} \wt{U}$, respectively.
For $W_c$, we first observe that since $\Sigma \xrightarrow{\pi} V_Y \times G$
is flat by Lemma \ref{lem:Sigma2}(1), one checks that for any point $y \in \Sigma$, the morphism
$\pi$ is smooth at $y$ if and only if
$\Sigma_{\pi(y)}$ is smooth over $\Spec(k(\pi(y)))$ at $y$
(e.g., see \cite[{Tag 07BT}]{stacks-project}). 
Hence, we have the equality $(\Sigma_w)_\nsm = (\Sigma_\nsm)_w$.
We can now deduce the constructibility of $W_c$ by
replacing $\pi$ by its composition with the inclusion
$\Sigma_\nsm \cap (V_Y \times \P^1_k \times G \times f(X_\sing)) \cap 
\Sigma'_U \inj
\Sigma'_U$ and applying \cite[E.1 (4)]{Gortz}.
\end{proof}

\begin{lem}\label{lem:Pull-back-1}
For every point $w \in (V_Y \times G)(k)$ lying in $W$, the following
hold.
\begin{enumerate}
\item
$\Gamma_w \cap \P^1_{X_\sing} \subset \P^1_{X_\cm}$.
\item
For $x \in \Gamma_w \cap \P^1_{X_\sing}$, one has
that $\wt{f}(x) \in (\Sigma_w)_\sm$.
\item
For every $x \in \Gamma_w \cap \P^1_{X_\sing}$, the ideal of $\Sigma_w$ in the 
local ring $\sO_{\P^1_H, \wt{f}(x)}$ is a complete intersection of height two.
\end{enumerate}
\end{lem}
\begin{proof}
The item (1) follows from the item (b) before \lemref{lem:Constructible}
and (2) follows from the item (c) before \lemref{lem:Constructible}.
By \lemref{lem:Sigma3}, $\Sigma_w$ is of pure codimension two in $\P^1_H$.
Since $\Sigma_w$ is smooth over $\Spec(k(w))$ at $\wt{f}(x)$ by (2), 
it is regular at $\wt{f}(x)$. In particular, $\Sigma_w \subset \P^1_H$
is a closed immersion of schemes both of which are regular at $\wt{f}(x)$.
Hence, (3) follows.
\end{proof}

\begin{lem}\label{lem:Pull-back-2}
For every point $w \in (V_Y \times G)(k)$ lying in $W$, the following
hold.
\begin{enumerate}
\item
$\Gamma_w$ is purely 1-dimensional.
\item
$\Gamma_w \cap \P^1_{X_\sing}$ is a finite set.
\item
$\Gamma_w \cap (X_\sing \times \{0, \infty\}) = \emptyset$.
\item
The inclusion $\Gamma_w \subset \P^1_X$ is a local complete intersection
morphism at every point of $\Gamma_w \cap \P^1_{X_\sing}$.
\item
The projection map
$\Gamma_w \to \P^1_k$ is flat over an open neighborhood of $\{0, \infty\}$.
\end{enumerate}
\end{lem}
\begin{proof}
The items (1), (2) and (3) are respectively same as
the items (a), (c) and (d) before \lemref{lem:Constructible}.
If $x \in \Gamma_w \cap \P^1_{X_\sing}$, then it must be a closed point
of $\P^1_X$ by (2). Let $J \subset \sO_{\P^1_X, x}$ be the ideal of
$\Gamma_w$ in $\sO_{\P^1_X, x}$.
Now, item (3) of \lemref{lem:Pull-back-2} shows that
the ideal $I$ of $\Sigma_w$ in the local ring $\sO_{\P^1_H, \wt{f}(x)}$
is a complete intersection of height two. In particular, we
can write $I = (a_1, a_2)$. We have the morphism of local rings
$\sO_{\P^1_H, \wt{f}(x)} \to \sO_{\P^1_X, x}$ and $J = I\sO_{\P^1_X, x}$.
This implies that $J = (a_1, a_2)$. On the other hand,
$\sO_{\P^1_X, x}$ is Cohen-Macaulay by \lemref{lem:Pull-back-1} (1).
Since ${\rm ht} (J) = 2$, it must be a complete intersection 
(e.g., see \cite[Theorem~17.4]{Matsumura}). This proves (4).

For (5), we first note that the projection map $\Sigma_w \to \P^1_k$
is flat over an open neighborhood of $\{0, \infty\}$ by
\lemref{lem:Sigma3}. Hence, if we repeat the proof of
\propref{prop:Kleiman-Levine-P1} (3), the only  thing we need to
know is that 
$\Tor^{\sO_{H_\epsilon}}_1(\sO_{(\Sigma_w)_\epsilon}, \sO_{X_\epsilon}) = 0$
for $\epsilon = 0, \infty$.
But for $w=(v,g)$ with $v\in V_Y(k)$ and $g\in G(k)$ as above, we have
\begin{equation}\label{eqn:Ratl-equiv}
\Gamma_w\times_{\P^1}\{0\} = 
f^{-1}(g\Phi_m(v, 0) \cdot Y) = f^{-1}(g_1 Y) \ \ \mbox{and} 
\end{equation}
\[
\Gamma_w\times_{\P^1}\{\infty\} = f^{-1}(g \Phi_m(v,\infty) \cdot Y) = 
f^{-1}(g_2 Y),
\]
where $\tau(v,g) = (g_1, g_2) \in U \times U$ 
(see ~\eqref{eqn:def-pi-Sigma-0}). 
Hence, the desired Tor-vanishing follows from our choice of $U$
and \lemref{lem:Tor-ind}.
\end{proof}

We shall need the following elementary lemma to show that $W$ contains a 
dense open subset of $V_Y \times G$.

\begin{lem}\label{lem:Elem*}
Let $f \colon X_1 \to X_2$ be a morphism in $\Sch_k$ such that $X_1(k)$
is Zariski dense in $X_1$. Let $Z \subset X_1$ be a closed subset
which has the property that $f^{-1}(x) \cap Z$ is dense in $f^{-1}(x)$ 
for every $x \in X_2(k)$.
Then $Z = X_1$.
\end{lem}
\begin{proof}
Suppose that $Z \neq X_1$. Our assumption then implies that 
$(X_1 \setminus Z)(k) \neq \emptyset$. Choose a point in
$x' \in (X_1 \setminus Z)(k)$ and let $x = f(x')$. 
Then $x \in X_2(k)$ and
$f^{-1}(x) \cap Z$ is a proper closed subset of $f^{-1}(x)$.
Hence, it can not be dense in $f^{-1}(x)$, a contradiction.
\end{proof}

\begin{lem}\label{lem:Pull-back-3}
$W$ contains a dense open subset of $V_Y \times G$.
%In particular, $\wt{U} \neq \emptyset$.
\end{lem}
\begin{proof}
Since $V_Y \times G$ is irreducible, it suffices to show that
$W$ contains the generic point $\eta$ of $V_Y \times G$. 
Suppose to the contrary that $\eta \notin W$. As $W$ is
constructible and $V_Y \times G$ irreducible, this means that
$\eta \notin \ov{W}$, where the latter is the closure of
$W$ in $V_Y \times G$.  
Since $(V_Y \times G)(k)$ is Zariski dense in $V_Y \times G$,
it follows from \lemref{lem:Elem*} that there exists
$v_0 \in V_Y(k)$ such that $p^{-1}_1(v_0) \cap \ov{W}$ is
not dense in $p^{-1}_1(v_0)$.

Since $p^{-1}_1(v_0) \cong G$ has a Zariski dense set of
$k$-rational points, we can choose
$w_0 = (v_0, g_0) \in (V_Y \times G)(k)$. Let
$\Sigma_{w_0} = \pi^{-1}(w_0) \subset \P^1_H$. By applying
\propref{prop:Kleiman-Levine-P1} to $\Sigma_{w_0}$,
we get a dense open subscheme $U_{w_0} \subset G$ such that
for every $g \in U_{w_0}(k)$, the pull-back $F_g := \wt{f}^{-1}(g\Sigma_{w_0})$
satisfies the following.
\begin{enumerate}
\item
$F_{g}$ has pure dimension $1$. 
\item $F_{g}\cap\P^1_{ X_{\ncm}} = \emptyset$.
\item
$F_{g}\cap\P^1_{ X_{\sing}}$ is finite and 
$F_g \cap (X_\sing \times \{0, \infty\}) = \emptyset$. 
\item
The inclusion $F_{g} \subset \P^1_X$ is a local complete intersection 
at each point of $\P^1_{X_\sing}$.
\item
The projection $F_{g}\to \P^1_k$ is flat over a neighborhood of 
$\{0, \infty\}$. 
\end{enumerate}

Since $G$ acts on $\Sigma \subset B \times \P^1_k \times G \times H$ by Lemma \ref{lem:extralemsigma}
where its action is trivial on the first two factors and diagonal on the 
product of the other two factors (see \S~\ref{sec:Sigma}), it follows that
every $g \in G(k)$ acts as an automorphism of 
$B \times \P^1_k \times G \times H$ which takes $g\Sigma_{w_0}$ onto
$\Sigma_{(v_0, gg_0)}$.
Hence, we see that for every $g \in U_{w_0}(k)$, the pull-back
$\wt{f}^{-1}(\Sigma_{(v_0, gg_0)})$ satisfies (1) $\sim$ (5) above.

We let $U'_{w_0} = (U_{w_0})g_0 \subset G$.
Then the right multiplication by $g_0$ defines an automorphism
$G \xrightarrow{\cong} G$ which takes $U_{w_0}$ onto $U'_{w_0}$.
It follows that $U'_{w_0}$ is dense open in $G$ and
$\wt{f}^{-1}(\Sigma_{(v_0, g)})$ satisfies (1) $\sim$ (5) above
for every $g \in U'_{w_0}(k)$. 
In particular, $\{v_0\} \times U'_{w_0} \subset W$.
But $\{v_0\} \times U'_{w_0}$ is clearly dense in $p^{-1}_1(v_0)$. 
This implies that $p^{-1}_1(v_0) \cap W$ is dense in $p^{-1}_1(v_0)$.
This leads to a contradiction.
We have thus proven the lemma.
\end{proof}

\begin{lem}
%[see Lemma 1.5 of \cite{Levine-2}]
\label{lem:pullback-well-defined} 
There exists a dense open subscheme $V(f, Y) \subset U(f,Y)$  of $G$, 
depending only on $f$ and $Y$, such that for every 
$g_1, g_2\in V(f, Y)(k)$, the classes 
$[f^{-1}(g_1 Y)]$ and $[f^{-1}(g_2 Y)]$ 
define the same element of $\CH^L_0(X)$. In particular, the assignment 
\eqref{eq:def-pullback-1} gives a  group homomorphism 
\[
f^*\colon \sZ^2(H) \to \CH^L_0(X).
\]
\end{lem}
\begin{proof}
Let $W' \subset W$ be a subset which is dense open in $V_Y \times G$.
Such a subset exists by \lemref{lem:Pull-back-3}.
Since $\tau$ is smooth (this is shown just below ~\eqref{eqn:def-pi-Sigma}),
its image $\tau(W')$ is dense open in $G \times G$. If we take the
projection of $\tau(W')$ to the first factor of $G \times G$, it is 
also dense open in $G$. Hence, it contains a $k$-rational point $g_0$.
Then $p^{-1}_1(g_0)\cap \tau(W')$ is dense open in $p^{-1}_1(g_0)$.
Since the composite map $p^{-1}_1(g_0)  \inj G \times G 
\xrightarrow{p_2} G$ is an isomorphism, it takes 
$p^{-1}_1(g_0) \cap \tau(W')$ to a dense open subscheme 
of $G$. Let $V(f,Y)$ be such a dense open subscheme.
Note that $V(f,Y) \subset U(f,Y)$ since $W\subset \tau^{-1}(U\times U)$.

For every $g \in V(f,Y)(k)$, we have that $(g_0, g) \in \tau(W)$.
Let $w \in W$ be such that $\tau(w) = (g_0, g)$. It is
clear from ~\eqref{eqn:def-pi-Sigma} that $w \in (V_Y \times G)(k)$.
It follows from \lemref{lem:Pull-back-2} and Definition~\ref{defn:LW-defn}
that $\Gamma_w$ is a Cartier curve on $\P^1_X$. Furthermore, it follows
from ~\eqref{eqn:Ratl-equiv}
that $\Gamma_w$ defines a rational equivalence between
$[f^{-1}(g_0Y)]$ and $[f^{-1}(gY)]$.
It follows that for any pair of points $g_1, g_2 \in V(f,Y)(k)$,
the cycles $[f^{-1}(g_1Y)]$ and $[f^{-1}(g_2Y)]$ are both 
rationally equivalent to $[f^{-1}(g_0Y)]$. Hence, they
are rationally equivalent each other.
This finishes the proof.
\end{proof}

\subsection{The pull-back map on the Chow groups}
\label{sec:Pull-back-Chow}
We let $G, H, X$ and $f \colon X \to H$ be as we described in the beginning of 
\S~\ref{sec:Pull-back}. We have shown in 
\S~\ref{sec:PB-1} that $f$ induces a pull-back map
$f^* \colon \sZ^2(H) \to \CH^L_0(X)$. We shall now show that this map
factors through the rational equivalence in the classical Chow group of
$H$.

\begin{prop}\label{prop:pullback-resp-rationaleq} 
The morphism $f^*\colon \sZ^2(H)\to \CH^L_0(X)$ descends to a  homomorphism 

\[
f^*\colon \CH^2(H)\to \CH^L_0(X)
\]
satisfying the following property: 
for every $Y\subset H$ integral subvariety of codimension $2$, 
the equality $f^{*}[Y] = [f^{-1}(gY)]$ holds in $\CH^L_0(X)$ for every 
$k$-point $g\in V(Y,f)$, where $V(Y,g)$ is the open dense subset of $G$ of 
\lemref{lem:pullback-well-defined}.
\end{prop}
\begin{proof}
We only need to show that $f^*$ respect the rational equivalence, since it is 
already clear by construction that on generators the pullback can be 
computed by first translating $Y$ by a rational point of $V(Y,f)$ 
followed by the naive pullback. 

It is well-known that the rational equivalence in $\CH^2(H)$ can be presented 
in terms of subvarieties of $H$ of codimension $1$ with rational functions on 
them, or in terms of subvarieties of codimension $2$ of $\P^1_H$. 
The latter presentation is easier to handle in our context. 
Following Levine \cite{Levine-2},  let $W\subset \P^1_H$ be an integral 
subscheme of codimension $2$, flat over $\P^1_k$, and let 
$p'_1 \colon \P^1_H \to H$ be the projection.

Let $W_0 = p'_1(W\cdot (H\times \{0\}))$ and $W_\infty = 
p'_1(W\cdot (H\times \{\infty\}))$. Note that the intersection product is 
well defined (this is simply the pull-back to $W$ of the Cartier divisors 
$0$ and $\infty$ of $\P^1_k$, which is well defined since the map is flat). 
We need to show that the difference $f^*[W_0] - f^*[W_\infty]$ is an element of 
$\sR^L_0(X, X_{\rm sing})$.
Let $V_0 = V(f, W_0)$ and let $V_\infty = V(f, W_{\infty})$. If we take 
$g\in (V_0\cap V_\infty)(k)$, we have  that $f^*[W_0] = [f^{-1} {(g W_0)}]$ and 
$f^{*}[W_\infty] = [f^{-1}(g W_\infty)]$ (for the same $k$-point $g$ of $G$).

By Proposition~\ref{prop:Kleiman-Levine-P1}, there exists an open dense subset 
$U=U(f\times \id_{\P^1}, W)$ of $G$ such that for every $g\in U(k)$, 
the inverse image $(f\times \id_{\P^1})^{-1}( gW)$ defines a purely 
$1$-dimensional closed subscheme $\Gamma$ of $\P^1_X$ satisfying the following 
properties.

\begin{enumerate}
\item The inclusion $\Gamma\subset \P^1_X$ is a local complete intersection 
morphism at every point of $\Gamma \cap \P^1_{X_\sing}$. 
\item $\Gamma \cap \P^1_{X_\sing}$ is a finite set and 
$\Gamma \cap (X_{\rm sing} \times \{ 0, \infty\}) =\emptyset$.
\item The induced map $\Gamma\to \P^1_k$ is flat over an open neighborhood of 
$\{0, \infty\}$.  
\end{enumerate}

In particular, 
$(p_1)_{*} (\Gamma \cdot (X\times \{ 0\})  ) - (p_1)_{ *} 
(\Gamma \cdot (X\times\{\infty \})  ) \in \sR^L_0(X, X_{\rm sing})$.  
If we choose $g\in (U\cap V_0\cap V_\infty)(k)$, we conclude that
\[
\begin{array}{lll}
f^*[W_0] = [f^{-1} (g W_0)] & = & (p_1)_{*} (\Gamma \cdot (X\times \{ 0\})) 
\sim (p_1)_{*} (\Gamma \cdot (X\times \{  \infty \})  ) \\
& = &    
[f^{-1} (g W_\infty )] = f^*[W_\infty],
\end{array}
\] 
as required.
\end{proof}

\subsection{Functoriality of $f^*$ in $G$ and $H$}\label{sec:Functorial}
We now show that the pull-back on the Chow groups behaves well with 
respect to the change in the group $G$ and the homogeneous space $H$. 
Let $(G, H)$ and $(G', H')$ be two pairs of groups and their
homogeneous spaces of the types described in the beginning of 
\S~\ref{sec:Pull-back}.
A morphism $\Psi \colon (G,H) \to (G', H')$ is a pair of maps
$\rho \colon G \to G'$ and $h \colon H \to H'$ such that $\rho$ is
a group homomorphism and $h$ is $G$-equivariant, where the latter acts
on $H'$ via $\rho$.

Let $\Psi \colon (G,H) \to (G'H')$ be as above.
Let $X$ be an equidimensional reduced quasi-projective surface
and let there be a commutative diagram
\begin{equation}\label{eqn:Change-of-groups}
\xymatrix@C.8pc{
X \ar[r]^-{f} \ar[dr]_-{f'} & H \ar[d]^-{h} \\
& H'.}
\end{equation}

We now prove the following functoriality property with respect to the change
of homogeneous spaces.

\begin{lem}\label{lem:functoriality-H}
There is a commutative diagram
\begin{equation}\label{eqn:functorialiy-H-0}
\xymatrix@C.8pc{
\CH^2(H') \ar[r]^-{h^*} \ar[dr]_-{f'^*} & \CH^2(H) \ar[d]^-{f^*} \\
& \CH^L_0(X).}
\end{equation}
\end{lem}
\begin{proof}
Let $Y$ be an integral subscheme of $H'$ of codimension $2$. 
%Using ~\eqref{eqn:Mlemma}, we can assume that $Y \in \sZ^2_{h,f'}(H')$.
Using \lemref{lem:pullback-well-defined} (applied simultaneously to
$h$ and $f'$), we let $U' \subset G'$ be a dense open such that 
$\alpha^*([Y]) = [\alpha^{-1}(gY)]$ and $[\alpha^{-1}(gY)] \sim
[\alpha^{-1}(g'Y)]$ for $g, g' \in U'(k)$ and $\alpha \in \{h, f'\}$.
We fix $g_0 \in U'(k)$ and let $Y_0 = g_0Y$ and $U'_0 = (U')g^{-1}_0$.   
Then $U'_0$ is a dense open subset such that 
$\alpha^*([Y_0]) = [\alpha^{-1}(gY_0)]$ and $[\alpha^{-1}(gY_0)] \sim 
[\alpha^{-1}(g'Y_0)]$ for $g, g' \in U'_0(k)$ and $\alpha \in \{h, f'\}$.

Using \lemref{lem:pullback-well-defined} (applied to $f$), we now let $U_1 \subset G$ be a 
dense open subset such that
\begin{equation}\label{eqn:Change-of-groups-0}
f^*(h^*([Y])) = f^*([h^{-1}(Y_0)]) = [f^{-1}(gh^{-1}(Y_0))] \ \ \mbox{and}
\end{equation}
\[
[f^{-1}(gh^{-1}(Y_0))] \sim [f^{-1}(g'h^{-1}(Y_0))] \ \ \mbox{for \ any} \ \
g, g' \in U_1(k).
\]
Since $U'_0$ contains the identity element of
$G'$ by its construction, it follows that $U'_0 \cap \rho(G) \neq
\emptyset$. In particular, $\rho^{-1}(U'_0)$ is dense open in $G$.
Hence, $U_1 \cap \rho^{-1}(U'_0)$ is dense open in $G$. We can therefore
replace $U_1$ by $U := U_1 \cap \rho^{-1}(U'_0)$ and 
~\eqref{eqn:Change-of-groups-0} continues to hold for $g, g' \in U(k)$.

We now choose any $g \in U(k)$ so that $g':= \rho(g) \in U'_0$.
Then we get
\[
\begin{array}{lll}
f^* h^* [Y]   =  f^*([h^{-1}(Y_0)])  & = & 
[f^{-1} (gh^{-1}(Y_0))] \\
& {=}^1 & [f^{-1}(h^{-1}(g'Y_0))] \\
& = & [f'^{-1}(g'Y_0)] = [f^{-1}(g'g_0Y)] \\
& {=}^{2} & f'^*([Y]),
\end{array}
\]
where ${=}^1$ follows from the $G$-equivariance of $h$ and ${=}^2$ follows
from the choice $U'$ because $g'g_0 \in U'(k)$.
This proves the lemma.
\end{proof}

We shall also need the following obvious functoriality property of
the pull-back map on the Chow groups.

\begin{lem}\label{lem:Pullback-RS}
Let $\pi \colon \wt{X} \to X$ be a resolution of singularities of $X$ and 
let $\wt{f} \colon \wt{X} \to X \to H$ be the composite map.
Then $\wt{f}^* = \pi^* \circ f^*$ as maps between the Chow groups
$\CH^2(H)$ and $\CH^2(\wt{X})$ of smooth schemes.
\end{lem}
\begin{proof}
Let $Y \subset H$ be an integral closed subscheme of codimension two.
Let $V(f, Y) \subset G$ be as in \lemref{lem:pullback-well-defined}
so that $f^*([Y]) = [f^{-1}(gY)]$ for every $g \in V(f, Y)(k)$.
Since $f^{-1}(gY)$ is a 0-cycle supported on $X_\reg$, it 
follows from the definition of $\pi^*$ in \propref{prop:Desingularization}
that $\pi^*([f^{-1}(gY)]) = [f^{-1}(gY)] = [\wt{f}^{-1}(gY)] = \wt{f}^*([gY])$.
Since $G$ is rationally connected, we know on the other hand
that $[Y] \sim [gY]$ in $\sZ^2(H)$. Hence, we get $\wt{f}^*([gY]) = 
\wt{f}^*([Y])$ in $\CH^2(\wt{X})$. We have thus shown that
$\pi^*(f^*([Y])) = \wt{f}^*([Y])$. 
\end{proof}

\subsection{$\P^1$-homotopy between two maps}\label{sec:PHI}
We shall now prove a suitable $\P^1$-invariance of maps between
the Chow groups. Let $H = G/P$ be a projective homogeneous space as 
before. For any $t \in \P^1_k(k)$, let $\iota_t \colon \{t\} \inj
\P^1_k$ be the inclusion.

\begin{lem}\label{lem:P-inv}
Let $F \colon X \times \P^1_k \to H$ be a morphism. Let $f_t =
F \circ \iota_t \colon X \to H$.
Then $f^*_0 = f^*_\infty \colon \CH^2(H) \to \CH^L_0(X)$.
\end{lem}
\begin{proof}
Let $Y \subset H$ be an integral cycle of codimension two.
Let $U \subset G$ be the intersection of the dense open subsets 
obtained by applying \lemref{lem:pullback-well-defined} to
$f_0$ and $f_\infty$. Let $\wt{F} \colon X \times \P^1_k \to H \times \P^1_k$
be the morphism $\wt{F} = (F, p_2)$, where
$p_2 \colon X \times \P^1_k \to \P^1_k$ is the projection.

By applying \thmref{thm:Kleiman-Levine} to 
$(\wt{G} := G \times \PGL_{2,k}, H \times \P^1_k)$, we 
get an open dense subset $\wt{U} \subset \wt{G}$ such that for
$\wt{g} = (g, \sigma) \in \wt{U}(k)$, the subscheme 
$\wt{F}^{-1}(gY \times \P^1_k) = \wt{F}^{-1}(\wt{g} \cdot \P^1_Y)$
satisfies (1) $\sim$ (3) of \thmref{thm:Kleiman-Levine}.
Since $(U \times \PGL_{2,k}) \cap \wt{U}$ is dense open in $\wt{G}$,
we can find a $k$-rational point $\wt{g} = (g, \sigma)$ in this 
intersection. 

We let $\Gamma = \wt{F}^{-1}(gY \times \P^1_k)$.
Then for any point $\wt{g}$ as above, we see from the choice of $U$ that
the fibers of the map $\Gamma \to \P^1_k$
over $\{0, \infty\}$ are supported on $X_\reg$
and $\Tor^{\sO_H}_i(\sO_{gY}, \sO_X) = 0$ for $i \ge 1$ with respect to
the maps $f_0$ and $f_\infty$. In particular, the argument in the proof of 
\propref{prop:Kleiman-Levine-P1} (3) shows that 
this map is flat over an open neighborhood of $\{0, \infty\}$.
Hence, $\Gamma$ determines a
Cartier curve on $\P^1_X$ such that $\Gamma_0 = [f^{-1}_0(gY)] = f^*_0([Y])$
and $\Gamma_\infty = [f^{-1}_\infty(gY)] = f^*_\infty([Y])$.
This finishes the proof.
\end{proof}

\section{The Chern classes for singular surfaces}\label{sec:Ring}
We fix an infinite field $k$ and a connected equidimensional reduced 
quasi-projective surface $X$ over $k$. In this section, we shall
review the ring structure on $\CH^L_0(X)$ due to Levine \cite{Levine-2}
and define Chern classes of vector bundles in this ring.

\subsection{Intersection product of Cartier divisors}\label{sec:product}
For a closed subscheme $D \subset X$, we let $|D|$ denote the support of
$D$. We let $\sK_X$ denote the Zariski sheaf of total quotient rings on $X$.
We let $\sZ^1(X, X_\sing)$ denote the subgroup of Cartier divisors
$D \in H^0_\zar(X, {\sK^{\times}_X}/{\sO^{\times}_X)}$ such that
$|D| \cap X_\sing$ is finite.
We let $\sR^1(X, X_\sing)$ be the subgroup of $\sZ^1(X, X_\sing)$ generated by
principal Cartier divisors $\divf(f)$, where $f \in H^0_\zar(X, \sK^{\times}_X)$
is invertible at all generic points of $X_\sing$.
We let $\CH^1(X) = {\sZ^1(X, X_\sing)}/{\sR^1(X, X_\sing)}$.
It is easy to check that
there is a canonical injective map $\CH^1(X) \hookrightarrow \Pic(X)$
which takes a Cartier divisor $D$ to the associated line bundle $\sO_X(D)$.
We let $\CH^0(X)$ denote the free abelian group on the cycle
$[X]$ so that there is a canonical isomorphism
$\Z \xrightarrow{\cong} \CH^0(X)$.
We let $\CH^*(X) = \CH^0(X) \oplus \CH^1(X) \oplus \CH^L_0(X)$.
Then a ring structure on $\CH^*(X)$ is completely determined by
defining the intersection product
\begin{equation}\label{eqn:product-divisors}
\CH^1(X) \otimes \CH^1(X) \to \CH^L_0(X).
\end{equation}

This construction is identical to the case of
smooth surfaces. We recall it here. 
Let $D, E$ be two effective Cartier divisors on $X$ with no common 
components. Let $x\in X_\reg$ be a closed point. Recall that 
the intersection multiplicity of $D$ and $E$ at $x$ is defined as
\[ i_x(D,E) = \ell(\cO_{X,x} /\cO_X(-D-E)_x).
\]
The number $i_x(D,E)$ is non-zero only  if $x\in \Supp(D)\cap \Supp(E)$, 
and satisfies the standard properties of symmetry and additivity. 
In particular, if $|D| \cap |E| \subset X_{\reg}$, the product 
$(D\cdot E) = \sum_{x\in |D| \cap |E|} i_x(D,E)[x]$
is a well-defined element of the Chow group $\CH^L_0(X)$. 

Note that since $X$ is quasi-projective, the map $\CH^1(X) \hookrightarrow \Pic(X)$ is in fact an isomorphism. In fact, we can write up to linear 
equivalence $D = D_1-D_2$ as difference of effective and very ample Cartier 
divisors without common components (e.g., see \cite[II Ex.7.5]{Hartshorne}).
Since $k$ is infinite, we can use the classical Bertini theorem
(e.g., see \cite[6.3 b)]{Jou}) to ensure 
that $|D_1| \cap |D_2| \cap X_\sing = \emptyset$ and
$|D_i| \cap X_\sing$ finite for $i = 1,2$.
We do the same for $E$, which we write (up to linear equivalence) as 
$E_1-E_2$ with the additional property that $D_i\cap E_j$ is finite and 
disjoint from $X_{\sing}$.
At this point, we can define 
\[(D\cdot E) = (D_1-D_2)\cdot (E_1-E_2) = D_1\cdot E_1 -D_1\cdot 
E_2 -D_2\cdot E_1 + D_2\cdot E_2.
\]
It is straightforward  to check that the product is 
well-defined (i.e., it does not depend on the presentation of 
$D$ and $E$). We leave out the details.

\subsection{Compatibility of product with pull-backs}
\label{sec:Prod-pullback}
Let $G$ and $H$ be as in \S~\ref{sec:Pull-back}. We let
$\CH^*(H) = \oplus_{i \ge 0} \CH^i(H)$ denote the classical Chow ring of $H$.
We assume that $\dim(H) \ge 2$.
Let $f \colon 
X \to H$ be a morphism. We have a well-defined pull-back map
$f^* \colon \Pic(H) \to \Pic(X)$. Equivalently, a pull-back map
$f^* \colon \CH^1(H) \to \CH^1(X)$.
To give a more explicit description of this, note that
for any  integral divisor $D \subset H$, we can apply
\thmref{thm:Kleiman-Levine} (to $X$ and $X_\sing$) 
to find a dense open $U \subset G$ such that
for every $g \in U(k)$, the scheme-theoretic pull-back $f^*(gD)$
satisfies conditions (1) $\sim$ (3) of \thmref{thm:Kleiman-Levine}.
But this precisely means that $f^*(gD)$ is an effective Cartier divisor
on $X$ whose no irreducible component is contained in $X_\sing$.
It is then clear from the definition of $\CH^1(X)$ in the beginning of
\S~\ref{sec:product} that $f^*([D]) = [f^*(gD)] \in \CH^1(X)$.

We have also defined the pull-back $f^* \colon \CH^2(H) \to \CH^L_0(X)$.
Since $f^* \colon \CH^0(H) \to \CH^0(X)$ is identity as $X$ is connected,
we have a well-defined codimension preserving (we define $f^*$ to be
zero on $\CH^{\ge 3}(H)$) group homomorphism $f^* \colon 
\CH^*(H) \to \CH^*(X)$. 

\begin{prop}\label{prop:compatibility-product-pullback}
The map $f^* \colon \CH^*(H) \to \CH^*(X)$ is a ring homomorphism.
If $\Psi = (\rho, h) \colon$ \\
$(G,H) \to (G', H')$ is a morphism as
in \S~\ref{sec:Functorial}, then $(h \circ f)^* = f^* \circ h^*$
as ring homomorphisms.
\end{prop}
\begin{proof}
We only need to show this for the product of cycles lying in
$\CH^1(H)$. For this, we note that $\CH^1(H) \cong \Pic(H)$
and we know that every line bundle on $H$ is a difference of
two very ample line bundles. Furthermore, if
$D$ is a Cartier divisor on $H$ such that $\sO_H(D)$ is very ample,
then the classical Bertini smoothness theorem over infinite fields
tells us that $D \sim D'$, where $D'$ is a smooth effective Cartier
divisor on $H$. It follows therefore that $\CH^1(H)$ is generated
by smooth very ample divisors. 

Now, given two smooth very ample divisors $D_1$ and $D_2$ on $H$,
we can use the Bertini theorem of Jouanolou (see also
\cite[Theorem~1]{KA}) to find elements
$D'_1$ and $D'_2$ in the linear systems $|H^0(H, \sO_H(D_1))|$
and $|H^0(H, \sO_H(D_2))|$, respectively such that $D'_1$ and $D'_2$
are geometrically integral smooth schemes which 
intersect transversely in a smooth and 
(geometrically) integral codimension two subscheme $Y \subset H$. 
We remind here that $H$ is smooth and geometrically integral. 
%Hence, general hypersurface section preserve the same property.
%Here, the connectedness of $D'_1$ and $D'_2$
%is forced due to the fact that $H$ is
%a geometrically integral smooth projective scheme of dimension
%at least two, in particular, $H^0(H, \sO_H) \cong k$.
%The integrality of $Y$ follows from the Bertini theorem
%\cite[Theorem~12]{Seidenberg}. 
We therefore need to show that if $D_1$ and $D_2$ are two smooth
connected very ample effective Cartier divisors on $H$ whose scheme-theoretic
intersection $Y$ is smooth and integral codimension two cycle on
$H$, then $[f^*(D_1)] \cdot [f^*(D_2)] = [f^*(Y)]$ in $\CH^L_0(X)$.

We now apply Theorem~\ref{thm:Kleiman-Levine} again. 
We can then find an open dense subset $U_i$ of $G$, for $i=1,2$ such that, 
for every $g\in U_i(k)$, the scheme-theoretic pull-back $f^{-1}(gD_i)$ is 
an effective Cartier divisor on $X$ having finite intersection with $X_\sing$. 
By replacing $U_1$ and $U_2$ by their intersection,
we can assume that $U_1=U_2 = U$ and that, the same $g$ works for both 
$D_1$ and $D_2$.

Let $V(f,Y)$ be the open subset of 
\lemref{lem:pullback-well-defined}. 
Then for every $g\in V(f,Y)(k)$, we have $f^*(D_1\cdot D_2) = f^*[Y]
= [f^{-1}(gY)]$, with $f^{-1}(gY)\cap X_{\rm sing} =\emptyset$. 
Choosing $g\in (V(f, Y)\cap U)(k)$, we obtain 
\[  
\begin{array}{lll}
f^*([D_1]\cdot [D_2])=[f^{-1}(gY)] & = & [f^{-1}([gD_1] \cdot [gD_2])] \\
& = & [f^{-1}(g D_1)]\cdot [f^{-1}(g D_2)] \\
& = & f^*[D_1]\cdot f^*[D_2],
\end{array}
\]
as required (the second equality follows directly from the local definition 
of the intersection product in \S~\ref{sec:product}).
The second assertion of the proposition follows directly from
\lemref{lem:functoriality-H}.
\end{proof}

\subsection{The Chern classes}\label{sec:C-class}
Let $X$ be as in the beginning of \S~\ref{sec:Ring}.
We now review the construction of Chern classes in $\CH^*(X)$ of vector bundles
on $X$. This construction depends solely on the results of the previous 
sections of this manuscript and does not use
any further information about the nature of the field $k$.
Hence, all the proofs in the construction of the Chern classes given
in \cite{Levine-2} and \cite{BSri} verbatim remain valid and we
have nothing extra to add.
We shall therefore only recall it very briefly and refer to \cite[\S~5]{BSri}
for details.

For any vector bundle $\sE$ of rank $n$ on $X$, we let $c_0(\sE) = 1$
and $c_1(\sE) = c_1(\bigwedge^n(\sE)) = [\bigwedge^n(\sE)] \in \Pic(X) = 
\CH^1(X)$. It is easy to check that
$c_1(\sE) = c_1(\sE') + c_1(\sE'')$ if 
\[
0 \to \sE' \to \sE \to \sE'' \to 0
\]
is an exact sequence of vector bundles.
These definitions are identical to the classical case.

We now recall the construction of $c_2(\sE)$. 
We first recall that for integers $0 \le n \le r$,
the Grassmanian variety $\Gr_k(n, r)$ is a representable functor
on $\Sch_k$ which to any $Y \in \Sch_k$ associates the set of
quotients $\sO^r_Y \surj \sF$, where $\sF$ is locally free of
rank $n$ on $Y$. Note that this construction holds over any base scheme $S$.
It is classically known that 
$\GL_{r,k}$ acts transitively on $\Gr_k(n, r)$ and the latter
is projective. Hence, we have that $\Gr_k(n, r) \cong {\GL_{r,k}}/P$ with
$P$ parabolic.

Let us now assume that $\sE$ is globally generated by the sections
$s_1, \ldots , s_r \in H^0(X, \sE)$. Then it follows from the 
above definition of $\Gr_k(n, r)$ that there is a unique
$k$-morphism $f \colon X \to \Gr_k(n, r)$ such that $\sE \cong
f^*(\sQ_{n,r})$, where $\sQ_{n,r}$ is the universal quotient vector
bundle on $\Gr_k(n, r)$.
We let $c_2(\sE):= f^*(c_2(\sQ_{n,r}))$. Note that this makes sense
since we have a well-defined theory of Chern classes of vector bundles
on smooth schemes. The fact that this definition does not depend on the
choice of the chosen sections follows from the following.

\begin{lem}\label{lem:indepchoice_sections}
Let $\{t_1, \ldots, t_r\}$ and $\{g_1, \ldots, g_s\}$ 
be two sets of global sections 
generating $\mathcal{E}$. Let $f_1\colon X\to \Gr_k(n, r)$ and 
$f_2\colon X\to \Gr_k(n, s)$ be the classifying morphisms. 
Then $f_1^*(c_2(\mathcal{Q}_{n,r})) = f_2^*(c_2(\mathcal{Q}_{n,s}))$ in 
$\CH^L_0(X)$. 
\end{lem}
\begin{proof}
This is straightforward using
Lemmas~\ref{lem:functoriality-H} and ~\ref{lem:P-inv}
(e.g., see \cite[Lemma~10]{BSri}).
\end{proof}

\begin{lem}\label{lem:Not-globally-gen}
Suppose that $\sE$ is globally generated by  $\{t_1, \ldots, t_r\}$ and
let $f \colon X \to \Gr_k(n, r)$ be the associated classifying morphism.
Let $\sL$ be a line bundle on $X$ globally generated by
$\{g_1, \ldots, g_s\}$. Let $h \colon X\to  \Gr_k(n, rs)$
be the classifying morphism for the vector bundle $\sE \otimes_{\sO_X} \sL$
and sections $\{t_i \otimes g_j| 1 \le i \le s, 1 \le j \le s\}$.
Then
\begin{equation}\label{eqn:Not-globally-gen-0}
h^*(c_2(\sQ_{n, rs})) = f^*(c_2(\mathcal{Q}_{n,r})) + 
(n-1) c_1(\mathcal{E})\cdot c_1(\mathcal{L})   + \binom{n}{2} 
c_1(\mathcal{L})^2.
\end{equation}
\end{lem}
\begin{proof}
This is straightforward using \lemref{lem:functoriality-H} 
and the theory of Chern classes on smooth schemes
(e.g., see \cite[Lemma~11]{BSri}).
\end{proof}

Let $\mathcal{E}$ be a rank $n$ vector bundle on $X$.
Let $\sL$ be a globally generated line bundle on $X$ such that
$\mathcal{E} \otimes_{\sO_X} \mathcal{L}$ is generated by  global sections 
$t_1, \ldots, t_r$. Let $f\colon X\to \Gr_k(n, r)$ be the classifying 
morphism given by the $t_i$'s. In view of 
Lemmas~\ref{lem:indepchoice_sections} and ~\ref{lem:Not-globally-gen},
the following definition is meaningful.

\begin{defn}\label{def:C-2}
We let 
\[
c_2(\sE) = f^*(c_2(\sQ_{n,r})) - (n-1) c_1(\mathcal{E})\cdot c_1(\mathcal{L})
- \binom{n}{2} c_1(\mathcal{L})^2.
\]

We define the total Chern class of $\sE$ to be
\[
c(\sE) = 1 + c_1(\sE) + c_2(\sE).
\]
\end{defn}

Note that the above definitions coincide with the classical definitions
of Chern classes if $X$ happened to be non-singular.

\begin{lem}\label{lem:functorialityChern}
The following hold for vector bundles on $X$ and their Chern
classes.
\begin{enumerate}
\item
Suppose $H = G/P$ is as in \S~\ref{sec:Hom*} and $f \colon X \to H$
is a morphism. If $\sE$ is a $G$-equivariant vector bundle on $H$, then
$f^*(c(\sE)) = c(f^*(\sE))$.
\item
If $\sE$ is a vector bundle on $X$, then there exists
$H = G/P$ as in \S~\ref{sec:Hom*} and a $G$-equivariant vector bundle
$\sE'$ on $H$ such that $\sE \cong f^*(\sE')$.
\item
If 
\[
0 \to \sE' \to \sE \to \sE'' \to 0
\]
is an exact sequence of vector bundles on $X$, then
$c(\sE) = c(\sE') \cdot c(\sE'')$.
\item
If $\pi \colon \wt{X} \to X$ is a resolution of singularities of $X$,
then $c(\pi^*(\sE)) = \pi^*(c(\sE))$.
\end{enumerate}
\end{lem}
\begin{proof}
We first prove (1).
It suffices to prove the lemma for $c_2(\sE)$ as it is obvious for
other Chern classes by their definitions.
We now note that since $G$ and $H$ are smooth and connected, 
it follows from \cite[Lemma~2.10]{Thomason} that $H$ admits
$G$-equivariant very ample line bundles. Hence, we can find a
very ample $G$-equivariant line bundle $\sL$ generated by its global sections
such that $\sE' := \sE \otimes_{\sO_H} \sL$ 
is also generated by its global sections.

Let $\{s_1, \ldots , s_r\}$ be a $k$-basis of 
$V:= H^0(H, \sE')$. Then $V$ becomes a rational
representation of $G$, giving a group homomorphism
$\rho \colon G \to \GL_k(V)$. Furthermore, the classifying morphism
$h \colon H \to \Gr_k(n,V)$ given by the above basis of $V$
is $G$-equivariant. It follows that $h \circ f : X \to
\Gr_k(n,V)$ is the classifying morphism for the vector bundle
$f^*(\sE')$ with sections $\{f^*(s_1), \ldots , f^*(s_r)\}$.
The item (1) of the lemma is clear for $h \circ f$ by the construction of
the Chern classes on $X$. Since
$\sE' = h^*(\sQ_{n, r})$, it follows from the theory
of Chern classes on smooth schemes that
$c_2(\sE') = h^*(c_2(\sQ_{n, r}))$.
We now conclude the proof of (1) by using 
Lemmas~\ref{lem:functoriality-H},
~\ref{lem:indepchoice_sections} and ~\ref{lem:Not-globally-gen}. 

For (2), we choose a globally generated line bundle $\sL$
such that $\sE \otimes_{\sO_X} \sL$ is also globally generated.
This gives rise to the classifying morphisms
$f_1 \colon X \to \Gr_k(n, r)$ and $f_2 \colon X \to \P^s_k$ for
some $r, s \ge 1$, where $n = {\rm rank}(\sE)$. We thus get a map
$f \colon X \to \Gr_k(n, r) \times \P^s_k$, and it is clear from
various universal properties that $\sE \cong f^*(\sQ_{n,r} \boxtimes 
\sO_{\P^s_k}(-1))$.

The item (3) is obtained exactly as (2).
Let ${\rm Fl}_k(n, n'', r)$ be the (partial) flag variety of 
quotients $k^r \surj V \surj W$, where $V$ and $W$ are $k$-vector 
spaces of ranks $n = {\rm rank}(\sE)$ and $n'' =  {\rm rank}(\sE'')$,
respectively. Then ${\rm Fl}_k(n, n'', r)$ is the projective homogeneous
space for $\GL_{r, k}$. This is a closed subvariety of
$\Gr_k(n,r) \times \Gr_k(n'',r)$ and the projections to these
Grassmanian varieties define a universal exact sequence of
$\GL_{r, k}$-equivariant vector bundles
\[
0 \to \sS \to p^*_1(\sQ_{n,r}) \to p^*_2(\sQ_{n'',r}) \to 0.
\]

We now choose a globally generated line bundle $\sL$
such that $\sE \otimes_{\sO_X} \sL$ is also globally generated.
Then $\sE'' \otimes \sL$ is globally generated too. As in (2), 
this gives rise to the classifying morphism 
$f \colon {\rm Fl}_k(n, n'', r) \times \P^s_k$
such that 
\[
0 \to \sS \boxtimes 
\sO_{\P^s_k}(-1) \to p^*_1(\sQ_{n,r}) \boxtimes 
\sO_{\P^s_k}(-1) \to p^*_2(\sQ_{n'',r}) \boxtimes 
\sO_{\P^s_k}(-1) \to 0
\]
pulls back to the exact sequence of (3) on $X$ via $f$.
We  are now done by (1) and \propref{prop:compatibility-product-pullback}
because (3) is well-known for exact sequences of vector bundles on
smooth schemes. 

The item (4) is again clear for $c_0$ and $c_1$. For $c_2$, we let 
$f \colon X \to \Gr_k(n,r)$ be the morphism obtained just above
Definition~\ref{def:C-2}. We let $\wt{f} \colon \wt{X} \xrightarrow{\pi}
X \to \Gr_k(n,r)$ be the composite map. 
We then get
\[
\begin{array}{lll}
c_2 \circ \pi^*(\sE) = c_2 \circ \pi^* \circ f^*(\sQ_{n,r}) 
& = & c_2 \circ \wt{f}^*(\sQ_{n,r}) \ \ {=}^1 \ \ \wt{f}^* \circ c_2(\sQ_{n,r}) \\
& {=}^2 & \pi^* \circ f^* \circ c_2(\sQ_{n,r}) \ \ {=}^3 \ \ \pi^* \circ c_2 \circ
f^*(\sQ_{n,r}) \\
& = & \pi^* \circ c_2(\sE).
\end{array}
\]
In the above, the equality ${=}^1$ follows from the known functoriality
of Chern classes on smooth schemes, ${=}^2$ follows from
\lemref{lem:Pullback-RS} and ${=}^3$ follows from the item (1) of the
lemma. This finishes the proof of (4).
\end{proof}

\subsection{The Chern classes on $K_0(X)$}\label{sec:CK-0}
Let $X$ be as in the beginning of \S~\ref{sec:Ring}. We let
$\CH^*(X)^{\times}$ denote the multiplicative group of units in $\CH^*(X)$
whose codimension zero part is equal to 1 (see \cite[\S~15.3]{Fulton}).
An immediate consequence of \lemref{lem:functorialityChern} (3) is that the
total Chern class defines a group homomorphism
\begin{equation}\label{eqn:TotalChern-K-0}
c_X \colon K_0(X) \to \CH^*(X)^{\times},
\end{equation}
where the left hand side uses the additive group structure of the ring
$K_0(X)$.

Let $\wt{K}_0(X)$ be the kernel of the the rank map
${\rm rk} \colon K_0(X) \surj \Z$. There is 
a canonical map $\det \colon \wt{K}_0(X) \to \Pic(X)$, induced by
taking a vector bundle to its determinant. This map is split
by the natural map $\Pic(X) \to \wt{K}_0(X)$ that sends a line bundle
$\sL$ to $[\sL] - [\sO_X]$. Let $SK_0(X)$ denote the kernel of $\det$.
It follows that there is a natural decomposition
\begin{equation}\label{eqn:TotalChern-K-1}
K_0(X) = \Z \oplus \Pic(X) \oplus SK_0(X).
\end{equation}

We let $x \in X_\reg$ be a closed point and let $U = X \setminus \{x\}$.
Let $j \colon U \inj X$ be the inclusion and let $j^* \colon
K_0(X) \to K_0(U)$ be the induced map.
It is then clear that the rank map of $K_0(X)$ factors through
$K_0(U)$. Furthermore, as $x$ is a regular closed point of $X$ and
$\dim(X) \ge 2$,
one knows that the map $\Pic(X) \to \Pic(U)$ is an isomorphism.
Since $j^*([\sO_{\{x\}}]) = 0$, it follows that the first and the
second components of $[\sO_{\{x\}}] \in K_0(X)$ under the decomposition
~\eqref{eqn:TotalChern-K-1} are zero. We conclude that
the cycle class map of ~\eqref{eqn:CCM-Levine-lci-0} canonically
factors through
\begin{equation}\label{eqn:TotalChern-K-2}
cyc^L_X \colon \CH^L_0(X) \to SK_0(X).
\end{equation}

\vskip .3cm

We shall denote the image of the cycle class map $cyc^L_X$ by $F^2K_0(X)$.
It also follows from the above discussion and the definitions 
of Chern classes that $c_i(cyc^L_X([x])) = 0$
for $i = 0,1$. In particular, we get $c_X \circ cyc^L_X([x]) = 
c_2 \circ cyc^L_X([x])$.
This also implies that $c_2 \colon F^2K_0(X) \to \CH^L_0(X)$ is
a group homomorphism.

\begin{lem}\label{lem:Cycle-class-injective}
The composite map $\CH^L_0(X) \xrightarrow{cyc^L_X} F^2K_0(X) 
\xrightarrow{- c_2} \CH^L_0(X)$ is identity.
\end{lem}
\begin{proof}
This is proven in \cite{Levine-2} and \cite[Proposition~2]{BSri}.
We reproduce the latter proof for sake of completeness.
It is enough to check that $c_2 \circ cyc^L_X([x]) + [x] = 0$ if
$x \in X_\reg$ is a closed point.
Fix such a point and let $j \colon U = X \setminus \{x\} \inj X$ be as
above. Let $j^* \colon \CH^L_0(X) \to \CH^L_0(U)$ be the
restriction map of \propref{prop:Localization}. Then it is clear from 
the construction of the Chern classes that $j^* \circ c_X = c_U \circ j^*$.
It follows therefore from the above argument that $j^*\circ c_2 \circ
cyc^L_X([x]) = 0$. We conclude from \propref{prop:Localization} that
$c_2 \circ cyc^L_X([x]) = m[x]$ for some $m \in \Z$.
It remains to show that $m = -1$.

We let $\pi \colon \wt{X} \to X$ denote a resolution of singularities
of $X$. This exists over all base fields since $\dim(X) = 2$.
Since $x \in X_\reg \subset \wt{X}$, we see that
$\pi^*([x]) = [x]$ under the pull-back map $\pi^* \colon
\CH^L_0(X) \to \CH^2(\wt{X})$.
Lemma~\ref{lem:functorialityChern} (4) easily implies that
$\pi^* \circ c_2 \circ cyc^L_X([x]) = c_2 \circ \pi^* \circ  cyc^L_X([x])$.
This yields
\[
\begin{array}{lll}
c_2 \circ  cyc_{\wt{X}}([x]) & = & c_2 \circ  cyc_{\wt{X}} \circ \pi^*([x]) 
\ {=}^{\dagger} \  c_2 \circ \pi^* \circ  cyc^L_X([x]) \\
& = &  \pi^* \circ c_2 \circ cyc^L_X([x]) = \pi^*(m[x]) \\
& = & m \pi^*([x]) = m[x],
\end{array}
\]
where ${=}^{\dagger}$ follows from \propref{prop:Desingularization}.
We thus get $c_2 \circ  cyc_{\wt{X}}([x]) = m[x]$ on $\wt{X}$.
Since the lemma is well known for non-singular surfaces, this
forces $m = -1$.
\end{proof}

The above lemma leads us to the following final result of this
section which generalizes Levine's result \cite{Levine-1} to
all infinite fields.

\begin{thm}\label{thm:Cycl-iso}
Let $X$ be an equidimensional reduced quasi-projective surface over an
infinite field. Then the cycle class map
\[
cyc^L_X \colon \CH^L_0(X) \to F^2K_0(X)
\]
is an isomorphism.
\end{thm}
\begin{proof}
The map $cyc^L_X$ is surjective by definition and is injective by
\lemref{lem:Cycle-class-injective}. 
\end{proof}

Combining \thmref{thm:Cycl-iso}, \corref{cor:CCM-Levine-lci} and
Lemmas~\ref{lem:K-iso}, ~\ref{lem:Bloch-cycle-maps} and
~\ref{lem:Bloch-surjection}, we obtain the following.

\begin{cor}\label{cor:BQ-for-LW}
Let $X$ be an equidimensional reduced quasi-projective surface over an
infinite field. Then the cycle class map  induces the isomorphisms
\[
\CH^L_0(X) \xrightarrow{\cong}H^2_{\zar}(X, \sK^M_{2,X})
\xrightarrow{\cong} H^2_{\zar}(X, \sK_{2,X}).
\]
\end{cor}

\section{The main results}
\label{sec:BFormula}
In this section, we shall prove the main results of this paper.
We shall first prove the Bloch-Kato formula for the
lci Chow group of singular surfaces and then use it
prove Theorems~\ref{thm:Main-1} and ~\ref{thm:Main-2}.

\subsection{The Bloch-Kato formula for singular surface}
\label{sec:BFSS}
Let $k$ be any field.
Let $X$ be an equidimensional reduced
quasi-projective surface over $k$.
In \S~\ref{sec:B-map}, we constructed the maps
\begin{equation}\label{eqn:BQ-cyc}
\CH_0(X) \xrightarrow{\rho_X} H^2_{\zar}(X, \sK^M_{2,X})
\xrightarrow{\lambda_X} H^2_{\nis}(X, \sK^M_{2,X}) 
\xrightarrow{\gamma_X} K_0(X).
\end{equation}
The following result extends \thmref{thm:Cycl-iso} and \corref{cor:BQ-for-LW}
to all fields if we use the lci Chow group.

\begin{thm}\label{thm:BQ-for-lci}
The cycle class map
\[
cyc_X \colon \CH_0(X) \to F^2K_0(X)
\]
is an isomorphism.
In particular, the map
\[
\rho_X \colon \CH_0(X) \to H^2_{\zar}(X, \sK^M_{2,X})
\]
is an isomorphism.
If $X_\reg$ is smooth over $k$ (e.g., if $k$ is perfect),
then $\lambda_X \circ \rho_X$ is also an isomorphism.
\end{thm}
\begin{proof}
In view of \thmref{thm:CCM-main}, we only need to show that
$cyc_X$ is injective to prove the theorem. 
If $k$ is finite, we can use \propref{prop:PF-PB},
\lemref{lem:Cycle-K-0} and the standard pro-$\ell$ extension
trick to reduce our problem to the case of infinite fields.
In the latter case, the desired injectivity follows 
directly from \corref{cor:CCM-Levine-lci} 
and \thmref{thm:Cycl-iso}. 
\end{proof}

\subsection{The Bloch-Kato formula for 0-cycles with modulus}
\label{sec:BFSS-mod}
We again assume $k$ to be an arbitrary field.
Let $X$ be a smooth quasi-projective surface over $k$
and let $D \subset X$ be an effective Cartier divisor.
Let $cyc_{X|D}$ and $\rho_{X|D}$ be the cycle class and the Bloch-Kato 
maps constructed in \S~\ref{sec:CCMaps} and \S~\ref{sec:BQM-mod}.
We let $F^2K_0(X,D)$ denote the image of the map $cyc_{X|D}$.
Consider the following diagram (see \S~\ref{sec:B-map} and
\S~\ref{sec:BQM-mod}).

\begin{equation}\label{eqn:BQ-for-mod-diag}
\xymatrix@C.8pc{
\CH_0(X|D) \ar[r]^-{\rho_{X|D}} & H^2_{\zar}(X, \sK^M_{2,(X,D)}) 
\ar[r]^-{\lambda_{(X,D)}} \ar[d] & H^2_{\nis}(X, \sK^M_{2,(X,D)}) 
\ar[dr]^-{\gamma_{(X,d)}} \ar[d] & \\
& H^2_{\zar}(X, \sK_{2,(X,D)}) 
\ar[r]^-{\lambda_{(X,D)}}  & H^2_{\nis}(X, \sK_{2,(X,D)}) 
\ar[r]^-{\gamma_{(X,d)}} & F^2K_0(X,D),}
\end{equation}
where the vertical arrows are the canonical maps from
the cohomologies of Milnor to Quillen $K$-theory sheaves.
The following result proves \thmref{thm:Main-1}.

\begin{thm}\label{thm:BQ-for-mod}
All of the maps in ~\eqref{eqn:BQ-for-mod-diag} are
isomorphisms.
\end{thm}
\begin{proof}
In view of \lemref{lem:K-iso} and
\thmref{thm:CCM-main-mod}, the proof of the theorem 
reduces to showing that $cyc_{X|D}$ is injective.
Using \propref{prop:PF-fields-mod}, \corref{cor:Cycle-K-mod}
and the standard pro-$\ell$ extension trick,
we can reduce our problem to showing the injectivity of $cyc_{X|D}$
when $k$ is infinite.

We now assume $k$ is infinite and consider the diagram
\begin{equation}\label{eqn:BQ-for-mod-0}
\xymatrix@C.8pc{
0 \ar[r] & \CH_0(X|D) \ar[r]^-{p_{+ *}} \ar[d]_{cyc_{(S_X, X_-)}} &
\CH_0(S_X) \ar[r]^-{\iota^*_-} \ar[d]^-{cyc_{S_X}} & \CH_0(X) \ar[r]
\ar[d]^-{cyc_X} & 0 \\
0 \ar[r] & K_0(S_X, X_-) \ar[r]^-{p_{+ *}} \ar[d]_-{\iota^*_+}
& K_0(S_X) \ar[r]^-{\iota^*_-} \ar[d]_-{\iota^*_+} & 
K_0(X) \ar[d]^-{\iota^*_D} \ar[r] & 0 \\
& K_0(X,D) \ar[r]^-{p_{+ *}} & K_0(X) \ar[r]^-{\iota^*_D} &
K_0(D). &}
\end{equation}
The top row is exact by \thmref{thm:BS-main}.
The middle row is split exact and the bottom row is exact.
It follows by applying \thmref{thm:BQ-for-lci} to $S_X$ that
$cyc_{S_X}$ is injective. It follows that $cyc_{(S_X, X_-)}$ is also
injective. The bottom left vertical arrow is an isomorphism by
\cite[Proposition~13.2]{BK} because $k$ is infinite.
Since $cyc_{X|D} = \iota^*_+ \circ  cyc_{(S_X, X_-)}$ by
~\eqref{eqn:CCM-main-mod-1}, we conclude that
$cyc_{X|D}$ is injective.
\end{proof}

For a connected projective variety $X$ over a field 
and an effective Cartier divisor 
$D \subset X$, let $\CH_0(X|D)^0$ denote the kernel of the
degree map ${\rm deg} \colon \CH_0(X|D) \to \Z$.
\thmref{thm:BQ-for-mod} has the following important consequence on the
finiteness of $\CH_0(X|D)^0$. We shall generalize this to higher
dimensions in \corref{cor:Finite-Chow}.

\begin{cor}\label{cor:Finite-dim-2}
Let $X$ be a smooth and connected projective surface over a finite
field and let $D \subset X$ be an effective Cartier divisor.
Then $\CH_0(X|D)^0$ is a finite abelian group.
\end{cor}
\begin{proof}
Let $H^2_{\nis}(X, \sK^M_{2,(X,D)})^0$ be the image of
$\CH_0(X|D)^0$ under the Bloch-Kato cycle class map
$\rho^{\nis}_{X|D} = \lambda_{(X,D)} \circ \rho_{X|D} \colon 
\CH_0(X|D) \to H^2_{\nis}(X, \sK^M_{2,(X,D)})$.
The group $H^2_{\nis}(X, \sK^M_{2,(X,D)})^0$ is finite by the
Kato-Saito class field theory (see \cite[Theorem~9.1]{Kato-Saito}
or \cite[Theorem~12.8]{GK-20}).
The corollary now follows by \thmref{thm:BQ-for-mod}.
\end{proof}

\subsection{The theorem of Kerz-Saito}\label{sec:KST}
Let $k$ be a finite field of characteristic $p > 0$.
Let $U$ be a smooth and 
%geometrically 
connected quasi-projective scheme over $k$ of dimension
$d \ge 1$. Choose a compactification $U \subset X$ with
$X$ normal and proper over $k$ such that $C = (X \setminus U)_\red$ is the 
support of an effective Cartier divisor on $X$. 
We let
\begin{equation}\label{eqn:Limit-Chow group}
C(U) = {\underset{D}\varprojlim} \ \CH_0(X|D),
\end{equation}
where the limit is taken over all effective Cartier divisors on $X$
supported on $C$. We endow each $\CH_0(X|D)$ with the discrete topology
and $C(U)$ with the inverse limit topology.
It is known (e.g., see \cite[Lemma~3.1]{KeS}) that $C(U)$ is independent of
the choice of the compactification $X$. 
We have the degree map $\deg \colon C(U) \to \Z$ which takes any
0-cycle to its degree. We let $C(U)^0 = {\rm Ker}(\deg)$.

%Since $k$ is finite, this map is surjective
%if $U$ is geometrically connected. We let $C(U)^0 = {\rm Ker}(\deg)$.

Let $\pi^{\ab}_1(U)$ denote the abelianized {\'e}tale fundamental group
of $U$. We let $\pi^{\ab}_1(U)^0 = {\rm ker}(\pi^{\ab}_1(U) \to 
{\rm Gal}({\ov{k}}/{k}))$.
If $x \in U$ is a closed point, then the inclusion
$\iota_x \colon \Spec(k(x)) \inj U$ defines the natural map
$(\iota_x)_* \colon \pi_1(\Spec(k(x))) \to \pi^{\ab}_1(U)$.
Letting $\rho_U(x) = (\iota_x)_*(F_x)$ (with $F_x$ being the Frobenius
element of the Galois group of $k(x)$) and extending linearly,
we get a group homomorphism $\rho_U \colon \sZ_0(U) \to \pi^{\ab}_1(U)$.
It follows from \cite[Proposition~3.2]{KeS} that this map induces a 
reciprocity map $\rho_U \colon C(U) \to \pi^{\ab}_1(U)$.

If $D \subset X$ is an effective Cartier divisor supported on $C$, we let
\begin{equation}\label{eqn:Limit-Chow group-0-***}
\pi^{\ab}_1(X,D) = \Hom_{\rm cont}({\rm fil}_DH^1(U, {\Q}/{\Z}), {\Q}/{\Z}),
\end{equation}
where ${\rm fil}_DH^1(U, {\Q}/{\Z})$ is the group of continuous 
characters $\chi \colon \pi^{\ab}_1(U) \to {\Q}/{\Z}$ such that for any
integral curve $Z \subset U$, the restriction
$\chi|_Z \colon \pi^{\ab}_1(Z) \to {\Q}/{\Z}$ satisfies
the following inequality of Cartier divisors on $\ov{Z}^N$:
\[
{\underset{y \in \psi_Z^{-1}(C)}\sum} {\rm art}_y(\chi|_Z)[y] \le 
\psi^*_Z(D).
\]
Here, $\psi_Z \colon \ov{Z}^N \to X$ is the projection map
from the normalization of the closure of $Z$ in $X$ and 
${\rm art}_y(\chi|_Z)$ is the Artin conductor of the
restriction of $\chi$ to ${\rm Gal}(k(Z)_y)$, where $k(Z)_y$ 
is the completion of $k(Z)$ at $y$ (see \cite{Serre}).

It is easy to check that there is an exact sequence
\begin{equation}\label{eqn:Limit-Chow group-0-**}
{\underset{Z \subset U}\bigoplus}\big( 
{\underset{y \in |\psi^{-1}_Z(D)|}\bigoplus} G^{n_y}_{k(Z)_y}\big) \to \pi^{\ab}_1(U)
\to \pi^{\ab}_1(X,D) \to 0,
\end{equation}
where the sum on the left runs over all integral curves
$Z \subset U$. Here, $\psi_Z^*(D) = \sum_{y \in |\psi^{-1}_Z(D)|} n_y [y]$
and $G^{n_y}_{k(Z)_y}$ is the higher ramification subgroup (for the upper 
numbering) of 
${\rm Gal}(k(Z)_y)$.
The first map in ~\eqref{eqn:Limit-Chow group-0-**} is the composite
\[
G^{n_y}_{k(Z)_y} \to
{\rm Gal}(k(Z^N)) \to \pi_1(Z^N) \xrightarrow{(\psi_Z)_*} \pi^{\ab}_1(U).
\]

It follows from ~\eqref{eqn:Limit-Chow group-0-***} that 
$\pi^{\ab}_1({X}, D)$ is the unique quotient of
$\pi^{\ab}_1(U)$ which classifies all finite abelian Galois covers
of $U$ whose ramification away from $U$ is bounded by the divisor
$D$.
It follows \cite[Proposition~3.9]{EK} (see also \cite[Proposition~2.10]{KeS})
that the limit of these quotient maps induces an
isomorphism 
\begin{equation}\label{eqn:Limit-Chow group-1}
\pi^{\ab}_1(U) \xrightarrow{\cong} {\underset{D}\varprojlim} \
\pi^{\ab}_1(X,D).
\end{equation}

We let $\rho_{(X,D)}$ denote the composition
$\sZ_0(U) \xrightarrow{\rho_U} \pi^{\ab}_1(U) \surj \pi^{\ab}_1(X,D)$.
It follows from the classical ramified class field theory for curves
that this descends to a map
$\rho_{(X,D)} \colon \CH_0(X|D) \to \pi^{\ab}_1(X,D)$, compatible with
$\rho_U$. Furthermore, 
$\rho_U = {\underset{D}\varprojlim} \ \rho_{(X,D)}$.

We let $\pi^{\ab}_1(X,D)^0$ denote the
kernel of the map $\pi^{\ab}_1(X,D) \to {\rm Gal}({\ov{k}}/{k})) \cong \wh{\Z}$.
It follows from \cite[Corollary~1.2]{KeS-1}
that $\pi^{\ab}_1(X,D)^0$ is a finite abelian group.
The reciprocity map $\rho_{(X,D)}$ restricts to a continuous homomorphism
of discrete abelian groups
$\rho^0_{(X,D)} \colon \CH_0(X|D)^0 \to \pi^{\ab}_1(X,D)^0$.
Moreover, $\rho^0_U = {\underset{D}\varprojlim} \ \rho^0_{(X,D)}$.

Using ~\eqref{eqn:Limit-Chow group-1} and the finiteness of
$\pi^{\ab}_1(X,D)^0$, we get a commutative diagram
of exact sequences of 
topological abelian groups and continuous group homomorphisms
\begin{equation}\label{eqn:Limit-Chow group-0}
\xymatrix@C.8pc{
0 \ar[r] &  C(U)^0 \ar[r] \ar[d]_-{\rho^0_U} & C(U) \ar[r]^-{\rm deg}
\ar[d]^-{\rho_U} & \Z \ar@{^{(}->}[d] \\
0 \ar[r] & \pi^{\ab}_1(U)^0 \ar[r] & \pi^{\ab}_1(U) \ar[r] &
\wh{\Z},}
\end{equation}
where the groups on the bottom have their canonical pro-finite topology
and the right-most vertical arrow is the pro-finite completion morphism.
The horizontal arrows on the right are surjective if $U$ is
geometrically connected over $k$.

An easy consequence of ~\eqref{eqn:Limit-Chow group},
~\eqref{eqn:Limit-Chow group-0-**} and
~\eqref{eqn:Limit-Chow group-1} is the following
(see \cite[Corollary~3.4]{KeS}).

\begin{lem}\label{lem:Limit-Chow group-2}
The map $\rho^0_U$ is an isomorphism of topological abelian 
groups if and only if
$\rho^0_{(X,D)} \colon \CH_0(X|D)^0 \to \pi^{\ab}_1(X,D)^0$
is an isomorphism of finite abelian groups for all effective Cartier
divisors $D \subset X$ which are supported on $C$.
\end{lem}

The following result is due to Kerz and Saito \cite{KeS}
when $p \neq 2$.
We shall prove this using \thmref{thm:BQ-for-mod} and the Kato-Saito
class field theory.

\begin{thm}\label{thm:Kerz-Saito-main}
$\rho^0_U$ is an isomorphism of topological abelian groups.
\end{thm}
\begin{proof}
By using Wiesend's trick (see \cite[Lemma~3.6]{KeS}), we can replace
the chosen compactification $X$ by any of its alterations in the sense of
de Jong. We can therefore assume that $X$ is smooth projective
and $C = X \setminus U$ is a simple normal crossing divisor on $X$.

Suppose now that $d \ge 3$. By \lemref{lem:Limit-Chow group-2},
it suffices to show that $\rho^0_{(X,D)}$ is an isomorphism for
all $D$. Since $\pi^{\ab}_1(X,D)^0$ is finite by
\cite[Corollary~1.2]{KeS-1}, it follows immediately from 
the Chebotarev-Lang
density theorem (e.g., see \cite[Theorem~7]{Serre}
or \cite[Theorem~5.8.16]{Tamas})
that $\rho^0_{(X,D)}$ is surjective. The heart of the
proof therefore is to show that 
$\rho_{(X,D)} \colon \CH_0(X|D) \to \pi^{\ab}_1(X,D) $ is injective.

%We have already seen that it is surjective. We need to prove
%that it is injective. 
Let $\alpha \in \CH_0(X|D)$ be
a 0-cycle such that $\rho_{(X,D)}(\alpha) = 0$.
By a generalized version of Poonen's Bertini theorem over
finite fields (see \cite[Corollary~5.4]{GhK}), 
we can find a very ample line bundle $\sL$ on
$X$ and a section $s \in H^0(X, \sL)$ such that its zero locus
$Y = Z(s) \subset X$ is smooth, $Y \times_X C$ is a simple
normal crossing divisor on $Y$ and $|\alpha| \subset Y$.
Let $\iota \colon Y \inj X$ be the inclusion
and let $E = Y \times_X D$. 
By the choice of $Y$, there exists $\alpha' \in \CH_0(Y|E)$ such that
$\alpha = \iota_*(\alpha')$.

We now have the diagram:
\begin{equation}\label{eqn:Limit-Chow group-3}
\xymatrix@C.8pc{
\CH_0(Y|E) \ar[r]^-{\rho_{(Y, E)}} \ar[d]_-{\iota_*} &
\pi^{\ab}_1(Y,E) \ar[d]^-{\iota_*} \\
\CH_0(X|D) \ar[r]^-{\rho_{(X, D)}} & \pi^{\ab}_1(X,D),}
\end{equation}
whose commutativity is immediate from
the definition of the reciprocity maps.

By \cite[Theorem~1.1]{KeS-1}, we can choose $\sL$ ample enough
(depending on $D$) so that the right vertical arrow 
in ~\eqref{eqn:Limit-Chow group-3} is an isomorphism.
It follows that $\rho_{(Y,E)}(\alpha') = 0$.
We have therefore inductively reduced the proof of the theorem
to the case when $d \le 2$.

We shall now show that $\rho^0_U$ is an isomorphism  when $d \le 2$.
Since we have already seen that the map $\rho^0_{(X,D)} \colon
\CH_0(X|D)^0 \to \pi^{\ab}_1(X,D)^0$ is surjective
for all $D$ supported on $C$ and since $\CH_0(X|D)^0$ is finite
by \corref{cor:Finite-dim-2} (note that the $d =1$ case of this finiteness
is classical), it follows that the
map $\rho_U^0 \colon C(U)^0 \to \pi^{\ab}_1(U)^0$ is surjective.
It suffices therefore to show that the map
$\rho_U \colon C(U) \to \pi^{\ab}_1(U)$ is injective.

Let $\rho^{\nis}_{X|D} = \lambda_{(X,D)} \circ \rho_{X|D}
\colon \CH_0(X|D) \to H^d_{\nis}(X, \sK^M_{d, (X,D)})$
be the cycle class map from  ~\eqref{eqn:BQ-for-mod-diag} for any $D$.
Let $H^d_{\nis}(X, \sK^M_{d, (X,D)})^0$ be the image of 
$\CH_0(X|D)^0$ under this map (see the proof of \corref{cor:Finite-dim-2}).
It is clear from the definition of $\rho^\nis_{X|D}$ (see \S~\ref{sec:BQM-mod})
that it is compatible
with the inclusions of effective Cartier divisors $D \subset D'$.
We let $\wt{\rho}_U :=  {\underset{D}\varprojlim} \ \rho^{\nis}_{X|D}$.

We now consider our key diagram
\begin{equation}\label{eqn:Limit-Chow group-4}
\xymatrix@C.8pc{
C(U) \ar[dr]_-{\rho_U} \ar[r]^-{\wt{\rho}_U} & {\underset{D}\varprojlim} \
H^d_{\nis}(X, \sK^M_{d, (X,D)}) \ar[d]^-{\wh{\rho}_U} \\
& \pi^{\ab}_1(U),}
\end{equation}
where $\wh{\rho}_U$ is the reciprocity map of Kato-Saito
\cite[\S~3]{Kato-Saito}. The map $\wt{\rho}_U$ is the inverse limit of
the cycle class maps $\rho^{\nis}_{X|D}$ from
~\eqref{eqn:BQ-for-mod-diag}, taken
over effective Cartier divisors $D \subset X$ supported on $C$.
It is immediate from the construction of
the three maps in ~\eqref{eqn:Limit-Chow group-4} and by 
\cite[Proposition 3.8]{Kato-Saito} that this diagram is commutative.

Since $H^d_{\nis}(X, \sK^M_{d, (X,D)})^0$ 
%$\CH_0(X|D)^0 \cong H^d_{\nis}(X, \sK^M_{d, (X,D)})^0$ 
is finite for all effective Cartier divisors $D \subset X$ supported on $C$, 
it follows that the canonical map 
${\underset{D}\varprojlim} \
H^d_{\nis}(X, \sK^M_{d, (X,D)}) \to 
{\underset{D, m}\varprojlim} \ H^d_{\nis}(X, \sK^M_{d, (X,D)})
\otimes {\Z}/m$
is injective.
On the other hand, the reciprocity map
$\wh{\rho}_U$ has a factorization
\[
{\underset{D}\varprojlim} \
H^d_{\nis}(X, \sK^M_{d, (X,D)}) \inj
{\underset{D, m}\varprojlim} \ H^d_{\nis}(X, \sK^M_{d, (X,D)})
\otimes {\Z}/m \xrightarrow{\wh{\rho}_U} \pi^{\ab}_1(U).
\]
The latter arrow is an isomorphism by \cite[Theorem~9.1 (3)]{Kato-Saito}.
It follows that the vertical arrow on the right in 
~\eqref{eqn:Limit-Chow group-4} is injective.
Since the horizontal arrow on the top in ~\eqref{eqn:Limit-Chow group-4}
is an isomorphism by \thmref{thm:BQ-for-mod} 
(see \cite[Lemma~3.1]{Krishna-1} when $d =1$) and a limit argument, 
we conclude that $\rho_U$ is injective. 
This finishes the proof.
\end{proof}

\vskip .3cm

\begin{cor}\label{cor:Pro-fin-comp}
The reciprocity maps $\{\rho_{(X,D)}\}_{|D| \subset C}$ induce an isomorphism of
pro-finite topological groups
\[
\rho_U \colon {\underset{D, m}\varprojlim} \ \CH_0(X|D)
\otimes {\Z}/m \xrightarrow{\cong} \pi^{\ab}_1(U).
\]
\end{cor}

%\corref{cor:Finite-Chow}

%\begin{cor}\label{cor:Finite-dim-2}

The following result provides an independent and a $K$-theoretic proof of a 
finiteness result of Deligne (see \cite[Theorem~8.1]{EK}).
This was also obtained independently by Kerz-Saito \cite{KeS} in characteristic
$\neq 2$.

\begin{cor}\label{cor:Finite-Chow}
Let $X$ be a normal and connected projective variety over a finite
field. Let $D \subset X$ be an effective Cartier divisor such that
$X \setminus D$ is regular.
Then $\CH_0(X|D)^0$ is a finite abelian group.
\end{cor}
\begin{proof}
Combine \lemref{lem:Limit-Chow group-2}, \thmref{thm:Kerz-Saito-main}
and \cite[Corollary~1.2]{KeS-1}.
\end{proof}

\vskip .5cm

\noindent\emph{Acknowledgements.}
The authors would like to thank the referee for a thorough reading of
the paper and providing helpful comments which led to an improved
presentation.

%\vskip .5cm
\bibliographystyle{amsalpha}
\bibliography{jag-BindaKrishnaSaito}

\end{document}